\documentclass[final,leqno]{siamltex}
\usepackage{amsmath}
\usepackage{amssymb}
\usepackage{float}
\usepackage{comment}
\usepackage{graphicx}
\usepackage{hyperref} %Darstellung von Web-Adressen
\hypersetup{draft=false, colorlinks=true, linkcolor=black, urlcolor=blue, citecolor=black}

\newtheorem{algorithm}[theorem]{Algorithm}
\newtheorem{remark}[theorem]{Remark}

\def \< {\langle} % Skalarprodukt links
\def \> {\rangle} % Skalarprodukt rechts

\def\be{\begin{equation}}
\def\ee{\end{equation}}
\def\bea{\begin{eqnarray}}
\def\eea{\end{eqnarray}}
\def\bean{\begin{eqnarray*}}
\def\eean{\end{eqnarray*}}

\def\NN{\mathbb{N}}
\def\RR{\mathbb{R}}

\def\ZZ{\mathbb{Z}}

\def\BB{\mathbb{B}}

\def\tn{n}

\def\tE{\tilde E}
\def\spann{\operatorname{span}}

\def\mG{\mathcal{G}}

\def\mT{\mathcal{T}}

\newtheorem{example}[theorem]{Example}

\title{An iterative method to reconstruct the refractive index of a medium from time-of-flight measurements}

\author{Udo Schr\"oder\thanks{Department of Mathematics, Saarland University, PO Box 15 11 50, 66041 Saarbr\"ucken, Germany ({\tt schreoder@math.uni-sb.de}).}
        \and Thomas Schuster\thanks{Department of Mathematics, Saarland University, PO Box 15 11 50, 66041 Saarbr\"ucken, Germany ({\tt thomas.schuster@num.uni-sb.de}).}}

\begin{document}

\maketitle

\begin{abstract}
The article deals with a classical inverse problem: the computation of the refractive index of a medium from ultrasound time-of-flight (TOF) measurements.
This problem is very popular in seismics but also for tomographic problems in inhomogeneous media. For example ultrasound vector field tomography needs a priori knowledge of the sound
speed. According to Fermat's principle ultrasound signals travel along geodesic curves of a Riemannian metric which is associated with the refractive index.
The inverse problem thus consists of determining the index of refraction from integrals along geodesics curves associated with the integrand leading to a nonlinear
problem. In this article we describe a numerical solver for this problem scheme based on an iterative minimization method for an appropriate Tikhonov functional.
The outcome of the method is a stable approximation of the sought index of refraction as well as a corresponding set of geodesic curves.
We prove some analytical convergence results for this method and demonstrate its performance by means of several numerical experiments. Another novelty in this article
is the explicit representation of the backprojection operator for the ray transform in Riemannian geometry and its numerical realization relying on a corresponding phase function that is determined 
by the metric. This gives a natural extension of the conventional backprojection from 2D computerized tomography to inhomogeneous geometries.
\end{abstract}

\begin{keywords} 
refractive index, ray transform, Riemannian metric, Fermat's principle, geodesic curve, backprojection operator, Tikhonov functional
\end{keywords}
\begin{AMS}
45G10, 53A35, 58C35, 65R20, 65R32
\end{AMS}

\pagestyle{myheadings}
\thispagestyle{plain}
\markboth{U. Schr\"oder and T. Schuster}{An iterative method to reconstruct the refractive index from TOF measurements}

%%%%%%%%%%%%%%%% EINLEITUNG %%%%%%%%%%%%%%%%%%%%%%%%%%%%%

\section{Introduction}

In this article we consider the inverse problem of computing the refractive index of a medium from ultrasound time-of-flight (TOF) measurements.
On the one side this task is a tomographic problem of its own and often called the \emph{inverse kinematic problem} which has important
applications, e.g. in seismics. On the other side the knowledge of the sound speed, resp. refractive index, of an object $M\subset\RR^2$ under consideration
is essential for inverse problems in inhomogeneous media such as photoacoustic or ultrasound vector tomography. The idea is very simple: an ultrasound
signal is emitted at a transmitter $A$ and its travel time is acquired at a detector $B$, see Figure \ref{motivation}. Of course the TOF
depends of the sound speed $c(x)$ in the medium. The \emph{refractive index} $\tn (x) = c_0/c(x)$, where $c_0$ denotes the constant sound speed of a 
reference medium like water or air outside the object, causes refractions of the ultrasound beam.
Assuming a constant $\tn$ (e.g. by setting $\tn=1$) in applications such as photoacoustic or vector tomography hence might cause severe artifacts.
We emphasized this in Figure \ref{motivation}. There, the blue triangle would be detected at the wrong place if we assume that the ultrasound signal
travels along a straight line and that there is no refraction at all.\\
\begin{figure}[H]
\centering
% \includegraphics[width=.50\textwidth]{./Grafiken/1_DoppelteAorten.png}
% \vspace{5mm}
 \includegraphics[width=.40\textwidth]{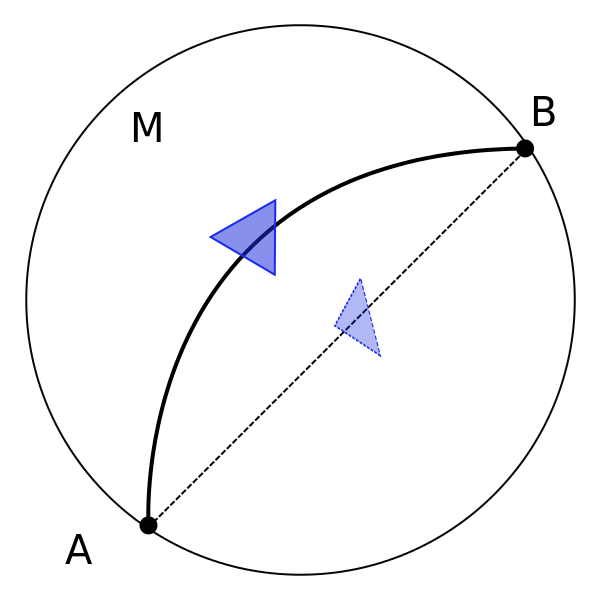}
%\vspace{3in}
\caption{TOF tomography measures the travel time a ultrasound signal needs to propagate from a source point $A$ to a detector $B$. 
If $\tn\not=1$, then the signal no longer travels along straight lines and reconstructions which neglect refraction show artifacts. 
E.g. the blue triangle would be detected at the wrong (dotted) place 
if we assumed straight lines as signal paths.} 
\label{motivation} 
\end{figure}
\noindent Let us briefly illuminate the situation in vector tomography. Norton \cite{n}
derived as mathematical model to compute a solenoidal vector field $v$ from TOF measurements the Doppler transform along straight lines
\[   D v (\ell) = \int_\ell \frac{\mathrm{d}l(x)}{c(x) + \langle v(x),\theta \rangle},   \]
where $\theta$ denotes the vector of direction of the line $\ell$ which connects the points $A$ and $B$. Hence, to solve the inverse problem of computing $v$ from data $D v$ it is necessary
to know the sound speed $c(x)$. But \emph{Fermat's pronciple} says that the propagation paths are geodesic curves of the Riemannian metric
\be\label{n-metric}  \mathrm{d}s^2 = \tn^2 (x) \, |\mathrm{d}x|^2   \ee
leading to an improved model
\[   D v (\gamma) = \int_\gamma \frac{\mathrm{d}l(x)}{c(x) + \langle v(x),\dot{\gamma} (x) \rangle},   \]
where $\gamma$ is a geodesic curve associated with the metric (\ref{n-metric}) and connecting the points $A$ and $B$. The problem now arises that
the integrand determines the integration curve $\gamma$ turning the inverse problem into a nonlinear one. Also in photoacoustic tomography a variable
sound speed leads to quite other analysis and numerics, see e.g. \cite{AGRANOVSKY;KUCHMENT:07,QIAN;ET;AL:11}. Hence, following Fermat's principle, then instead of the Euclidean space $(M,g_0)$
with the Euclidean metric tensor $g_0=(\delta_{ij})$, we have to consider the Riemannian manifold $(M,g)$ with metric tensor $g_{ij} (x) = \tn^2 (x) \delta_{ij}$.\\
This is motivation to investigate the following inverse problem: Given TOF measurements $u^{meas}$ for transmitter / detector pairs $(A,B)$, compute the
index of refraction $\tn (x)$ satisfying 
\be\label{IP-refractive}   R (\tn) = u^{meas} ,  \ee
where the ray transform $R$ is given as
\be\label{ray-transform}   R (\tn)(x,\xi) := \int_{\gamma^\tn_{x,\xi}} \mathrm{d} s = \int_{\tau}^0 \tn \big( \gamma^\tn_{x,\xi}(t) \big) |\dot{\gamma}^\tn_{x,\xi}(t)|\, \mathrm{d} t . 
\ee
Here, $\gamma^\tn_{x,\xi}$ denotes a geodesic curve of the metric (\ref{n-metric}) with $\gamma_{x,\xi} (0)=x\in\partial M$ and $\dot{\gamma}_{x,\xi}^\tn (0)=\xi$ 
for $\xi\in T_x M$ a tangent vector at $x$.\\
This inverse problem and its research has a long lasting history. We summarize some important references in this context.
Herglotz \cite{herglotz} was among the first researchers who has taken inhomogeneities into account. He investigated the earth's inner structure by considering travel times of seismic waves.
Mukhumetov \cite{muku1975} proved that the determination of simple metrics in two dimensions from travel times is possible. The extension to the three-dimensional case
was achieved independently by Romanov \cite{roma1978} and Mukhumetov \cite{muku1981}. General results on the inverse kinematic problem have been proven by Stefanov and Uhlmann
in \cite{STEFANOV;UHLMANN:09}, Chung et al. \cite{CHUNG;ET;AL:07} and Sharafutdinov \cite{s2}; a further uniqueness and stability result can be found in \cite{STEFANOV;UHLMANN;VASY:14}. The 2D problem for anisotropic metrics was solved by Pestov and Uhlmann \cite{pestov2005}; the approach contained therein is constructive.
The question of a unique solution, the so called \emph{boundary rigidity problem},
is not entirely solved by now. First results in 2D were achieved by Michel \cite{michel1981}, Croke \cite{croke1990} and Otal \cite{otal1990}. Pestov and Uhlmann \cite{stefa1998}
showed uniqueness for simple, two-dimensional Riemannian manifolds. A microlocal treatment can be found in Stefanov and Uhlmann \cite{STEFANOV;UHLMANN:04}. Local and semi-global results were presented by 
Croke et al. \cite{croke2000}, Stefanov and Uhlmann \cite{stefa1998}, Gromov et al. \cite{gromo1983}, Croke \cite{croke1991}, \cite{croke2004} and 
Lassas et al. \cite{lassas2003}, partly for special metrics only. For more references concerning analytical results for the inverse kinematic problem and the boundary rigidity problem
we refer to the book of Sharafutdinov \cite{s} and the references therein. Based on the Pestov-Uhlmann reconstruction formulas from \cite{pestov2005} Monard \cite{MONARD:14} derived
a numerical solver for the linear geodesic ray transform. Another numerical solution scheme which relies on Beylkin's theory \cite{Beylkin:1984} is presented in \cite{ps}. The influence of refraction
to reconstruction results in 2D emission tomography have been studied in \cite{DEREVTSOV;ET;AL:00}, a numerical solver for the geodesic ray transform based on B-splines is presented in
\cite{svetov}. A further numerical solver for the inversion of $R$ (\ref{ray-transform}) is found in Klibanov and Romanov \cite{KLIBANOV;ROMANOV:15}; here the linearization is done by replacing the geodesic curves by straight lines.\\
The novelty of our article is twofold: On the one hand we consider the nonlinear problem \eqref{IP-refractive} and our numerical solver linearizes $R$ in each iteration step using the old iterate $\tn_k$ to compute the geodesic curves. On the other hand we use an explicit representation of the geodesic backprojection operator and show how to implement it. To this end the construction of a so called
\emph{geodesic projection} was necessary.\\
\emph{Outline.} In Section 2 we provide essential results from Riemannian geometry which are necessary for our further considerations as well as the mathematical model for the inverse problem.
Additionally we collect some mathematical properties of the nonlinear forward operator $R$ (\ref{ray-transform}). The iterative solver which we develop in this article demands for evaluation of
integrals along geodesic curves. The computation of these curves is done using the method of characteristics which is outlined in Section 2.5. The regularizing solution scheme is subject of Section 3. 
We formulate an appropriate Tikhonov functional and linearize $R$ for its minimization. The derivative of the so arising functional contains the backprojection operator of the geodesic ray tranform. 
We give an explicit expression of this operator using the concept of \emph{phase functions} and \emph{geodesic projection} yielding in that sense an analogon to the conventional 2D backprojektion operator in Euclidean geometry (Section 3.1). The iterative minimization scheme and its implementation is described in Section 3.2. Section 4 finally contains numerical evaluations of the method for several
refractive indices $\tn$ with exact and noisy data. Section 5 concludes the article.\\

%%%%%%%%%%%%%%%%%%% Section Mathematical Setting %%%%%%%%%%%%%%%%%%%%%%%%%%%%%%%%%%%

\section{Mathematical setup and modeling}

\subsection{Basics from Riemannian geometry}

We collect some fundamental results from differential geometry which are useful for our later considerations. Throughout the article
we assume $M \subset \RR^2$ to be a compact and convex domain
which is seen as a submanifold of $\RR^2$.\\
\begin{definition}[Riemannian metric, metric tensor] \label{def:2_RiemMetrik}
On $M \subset \RR^2$ we define a \emph{Riemannian metric} as a differentiable mapping $M \ni x \mapsto g_x=g(x)$, such that 
 \[
 g_x:T_xM \times T_xM \to \RR
 \]
is a positive definite, symmetric bilinear form on the tangent space $T_xM$ in $x$. We have for $\xi, \eta \in T_xM$
\begin{itemize}
 \item[1.] $g_x(\xi,\eta) = g_x(\eta,\xi)$,
 \item[2.] $g_x(\xi,\xi) > 0$, if $\xi \neq 0$ and
 \item[3.] for diffentiable vector fields $X,Y:M \to TM$ is $x \mapsto g_x(X_x,Y_x)$ a differentiable mapping.
\end{itemize}
Here $TM=\{(x,\xi) : x\in M,\; \xi\in T_xM \}$ is the tangent bundle on $M$.\\
A representation of the metric $g_x$ with respect to local coordinates is given by
\be
 g_x = \sum_{i,j=1}^n g_{ij}(x) \cdot dx^i|_x \otimes dx^j|_x. \label{gl:2_DefiMetrikSum}
\ee
The third condition is then equivalent to the requirement that the coefficient functions $g_{ij}(x)$ are differentiable independently of the chart.
The local coordinates $(g_{ij}) = (g_{ij})_{i,j=1}^2$ are called \emph{metric tensor}.\\
If it is convenient, then we use the Einstein notation. That means, that we sum up over doubled indices. The representation \eqref{gl:2_DefiMetrikSum} then becomes
\[
 g_x = g_{ij}(x) \cdot dx^i|_x \otimes dx^j|_x.
\]
The tuple $(M,g)$ is called \emph{Riemannian manifold}.\\[1ex]
\end{definition} 
\begin{example}
Let $g_0$ be the metric tensor of Euclidean geometry (${g_0}_{ij} = \delta_{ij}$). Then $(\RR^n, g_0)$ is a Riemannian manifold and is called \emph{Euclidean space}.\\[1ex]
\end{example}
We introduce a specific metric tensor which plays a crucial role when studying ultrasound wave propagation in an inhomogeneous medium. 
Let $c(x)$ be the speed of sound at $x\in M$ and $c_0$ be the sound speed of a reference medium (e.g. air or water). Then, $\tn (x) = c_0/c(x)$ denotes
the \emph{index of refraction}. Especially we assume $c=c_0$ in $\RR^2\backslash M$.\\
\begin{lemma}\label{lem:2_Messgeometrie}
For $x\in M$ let
\be\label{metric-n}
 g_x^{\tn} = g_{ij}^{\tn}(x)\, dx^i \otimes dx^j
\ee
with metric tensor
\[
 g_{ij}^{\tn}(x) := \tn ^2(x)\delta_{ij},
\]
where the index of refraction $\tn = \frac{c_0}{c}: M \to \RR^+$ is supposed to be positive and differentiable. Then $(M,g^{\tn})$ is a Riemannian manifold. The element of length is then given as
\[  \mathrm{d} s^2 = \tn^2 (x) |\mathrm{d} x|^2\,.   \]
\end{lemma}
\begin{proof}
The submanifold $M \subset \RR^2$ can be canonically embedded into $\RR^2$ such that we can choose the identity as chart which is differentiable. The metric tensor $g^{\tn}$ satisfies
the requirements of Definition \ref{def:2_RiemMetrik}, 
since it is symmetric ($g_{ij}^{\tn} = g_{ji}^{\tn}$), differentiable (the component functions $\tn$ are differentiable) and positive definite ($\tn > 0$).
\end{proof}
\\[2mm]
For simplicity we set $c_0 = 1$ for the rest of the article.

\subsection{Geodesic curves}

Our aim is to model the propagation of ultrasound waves in a medium with variable sound speed by geodesic curves associated with the metric tensor (\ref{metric-n}).
This is due Fermat's principle (see Section 2.3). We summarize the basics of geodesics. For details we refer to standard textbooks such as \cite{LOVETT:10}. 
\newline
\begin{definition}[Geodesic curve]
Let $(M,g)$ be a Riemannian manifold, $I=[\tau_0,\tau_1]$. The curve $\gamma: I \to M$ is called \emph{geodesic curve} or \emph{geodesic}, if it satisfies the \emph{geodesic equation}
 \be
  \ddot \gamma^i + \Gamma_{jk}^i(\gamma) \dot \gamma^j \dot \gamma^k = 0. \label{gl:2_GeodDGL}
 \ee
The \emph{Christoffel symbols} $\Gamma_{jk}^i$ for $x \in M$ are given as
\be
 \Gamma_{jk}^i(x) := \frac{1}{2}g^{ip}(x)\left( \frac{\partial g_{jp}}{\partial x^k}(x) + \frac{\partial g_{kp}}{\partial x^j}(x) - \frac{\partial g_{jk}}{\partial x^p}(x) \right), 
\label{gl:2_Christoffel}
\ee
where $g^{ip}$ denote the coefficients of the inverse of the metric tensor $(g_{ip})$.
\newline
We define the \emph{length} $T(\gamma)$ of $\gamma$ by
\bean
 T(\gamma) &=& \int_{\gamma} \, ds \\
 &=& \int_{\tau_0}^{\tau_1} \sqrt{g_{\gamma(t)}\left( \dot \gamma(t),\dot \gamma(t) \right)} \, dt.
\eean
\end{definition}
\\[2mm]
Equipped with initial conditions $\gamma(0) = x\in M$ as starting point and $\dot \gamma(0) = \xi \in T_xM$ as an initial direction, it follows by the Picard-Lindel\"of theorem 
that \eqref{gl:2_GeodDGL} has a unique 
solution which is then denoted by $\gamma_{x,\xi}$. If we want to point out on which metric the geodesic depends, then we write $\gamma_{x,\xi} = \gamma^g_{x,\xi}$. For $\gamma_{x,\xi}^{\tn}:=
\gamma_{x,\xi}^{g^{\tn}}$ the length $T(\gamma_{x,\xi}^{\tn})$ coincides with the TOF of an ultrasound signal emitted from $x$ in direction $\xi$.
\newline
\begin{definition}[Distance]\label{D-distance}
 Let $(M,g)$ be a Riemannian manifold and $x,y \in M$. We define the \emph{distance} between $x$ and $y$ by
 \[
  d(x,y) := \inf_{\gamma \in C^\infty(I,M) \atop \gamma(0) = x, \gamma(\tau) = y} T(\gamma).
 \]
Any curve $\tilde \gamma$ in $M$ which attains the infimum is called \emph{shortest curve} from $x$ to $y$.
\end{definition}
\\[2mm]
In the Euclidean space $(\RR^2,g_0)$ all geodesic curves are straight lines and vice versa. Particularly all geodesics are at the same time shortest paths between two points.
On the sphere $S^2 := \left\{ x\in\RR^3 : \|x\| = 1 \right\}$ equipped with the Euclidean metric induced from $\RR^3$, all great circles are geodesic curves. 
But for two separate points on the sphere there are 
two geodesic curves which connect them and in general only one of them is a shortes curve between these points.
\newline
\begin{lemma}\label{kor:2_StetigeGeodaeten}
Let $\tn \in C^2(M)$. Then all geodesic curves and their first derivatives depend continuously on the initial values and the refractive index $\tn$.\\[1ex]
\end{lemma}
\begin{proof}
Setting
\bean
z_1(t) &:=& \gamma_1(t) \\
z_2(t) &:=& \gamma_2(t) \\
z_3(t) &:=& \dot z_1(t) = \dot \gamma_1(t) \\
z_4(t) &:=& \dot z_2(t) = \dot \gamma_2(t).
\eean
with initial values $\gamma(0) = x \in M$ und $\dot \gamma(0) = \xi \in T_xM$, we transform the geodesic equation (\ref{gl:2_GeodDGL}) in a system of first order
\[   \dot z (t) = f(t,z) ,  \]
where
\bean
\begin{matrix}
f_1(t,z(t)) &=& z_3(t), \\
f_2(t,z(t)) &=& z_4(t), \\
f_3(t,z(t)) &=& -\tn^{-1}(\overline z(t)) \left[ \overline z^T(t)\begin{pmatrix}
                                              \frac{\partial \tn(x)}{\partial x_1} & \frac{\partial \tn(x)}{\partial x_2} \\
                                              \frac{\partial \tn(x)}{\partial x_2} & -\frac{\partial \tn(x)}{\partial x_1}
                                             \end{pmatrix} \overline z(t) \right] \text{ and} \\ &&\\
f_4(t,z(t)) &=& - \tn^{-1}(\overline z(t)) \left[ \overline z^T(t)\begin{pmatrix}
                                              -\frac{\partial \tn(x)}{\partial x_2} & \frac{\partial \tn(x)}{\partial x_1} \\
                                              \frac{\partial \tn(x)}{\partial x_1} & \frac{\partial \tn(x)}{\partial x_2}
                                              \end{pmatrix} \overline z(t) \right].
\end{matrix}
 \eean
Thereby $\overline z(t) :=  (z_1(t) , \, z_2(t))^T$ for $t\in I$, $I\subset \RR$ compact. Because $\tn$ is continuously differentiable, $f$ is continuously differentiable, too. Furthermore $f$ fullfills a Lipschitz condition, which we show by using the mean value theorem. For any $\alpha > 0$ let
\[  S_\alpha = \big\{ (t,y) : t\in I,\; \|y-z(t)\|\leq \alpha \big\} \subset I\times \RR^4.  \]
The mean value theorem guarantees the existence of $(t,z)\in S_\alpha$ such that
\[  \big( f(t, x_1(t)) - f(t, x_2(t)) \big) = \nabla_z f (t, z(t)) \cdot \big(x_1 (t) - x_2 (t) \big)  \]
for all $(t,x_1), (t,x_2)\in S_\alpha$. Since $\tn\in C^2 (M)$ and $\tn>0$, it is easy to prove that
\[   L := \sup_{(t,z)\in S_\alpha} \|\nabla_z f (t,z)\| < \infty  \]
The assertion now follows from \cite[Ch. III]{WALTER:98}.
\end{proof}
\\[1mm]
\begin{definition}
A Riemannian metric $g$ on a compact manifold $M$ is called \emph{simple}, if the boundary $\partial M$ is strictly convex and every two points $x,y\in M$ are connected by a unique
geodesic curve which depends smoothly on $x,y$. A geodesic $\gamma : [a,b]\to M$ is called \emph{maximal} if it can not be extended to a segment
$[a-\varepsilon_1,b+\varepsilon_2]$ for any $\varepsilon_1,\varepsilon_2\geq 0$. The metric $g$ is called \emph{dissipative}, if it is simple and if for every point $x\in M$ and vector $0\not= \xi \in T_x M$ the maximal geodesic $\gamma_{x,\xi} (t)$ is defined on a finite segment $[\tau_- (x,\xi), \tau_+ (x,\xi)]$.\\[1ex]
\end{definition}
If the metric $g$ is simple, then obviously every geodesic is also a shortest curve in the sense of Definition \ref{D-distance}.

\subsection{Modeling TOF measurements}

Our modelling bases on an important physically axiom, \emph{Fermat's principle}. It can be summarized as follows:\\
\begin{quote}
A wave signal, which moves from one point to another, always follows the locally shortest path, such that the time of flight is at its minimum. This means that the acceleration in every point 
disappears in path direction.\\[1ex]
\end{quote}
As a consequence of this axiom it follows that the signals move along geodesic curves. 
According to this axiom  we are going to model ultrasound beams in an inhomogeneous medium with refractive index $\tn (x)$ as geodesic curves associated with the metric
\be\label{metric-fermat}
g^{\tn} (x) = \tn^2 (x) (\delta_{ij})\,.
\ee
We have now all ingredients together to describe the mathematical model of our measurement process.
\newline
\begin{definition}[Time-of-flight mapping] \label{def:2_Laufzeitabbildung}
Let $(M,g^{\tn})$ be a compact Riemannian manifold, where $g^{\tn}$ is the metric (\ref{metric-fermat}). We call the mapping
 \be\label{TOF-mapping}
  u:T^0M \to \RR^+, (x,\xi) \mapsto \int_{\tau_-(x,\xi)}^0 \tn\left( \gamma_{x,\xi}(t) \right) \, dt
 \ee
\emph{time-of-flight mapping} (TOF mapping), where
\[
 T^0M := \left\{ (x,\xi) \in TM : \xi \neq 0 \right\},
\]
$\gamma_{x,\xi}^{\tn}:\RR \to M$ is parametrised with respect to arc length and
\[
 \tau_- (x,\xi) := \max\left\{ \tau \in (-\infty,0] : \gamma_{x,\xi}^{\tn}(\tau) \cap \partial M \neq \emptyset \right\}
\]
is the moment where $\gamma_{x,\xi}^{\tn}$ intersects the boundary $\partial M$ for the first time.
\end{definition}
\\[2mm]
The following definition addresses the practical situation that we have measurements at the boundary $\partial M$. In this case we have to distinguish whether the ultrasound wave enters or leaves
the domain $M$.\\
\begin{figure}[H]
 \centering
 \includegraphics[width=.5\textwidth]{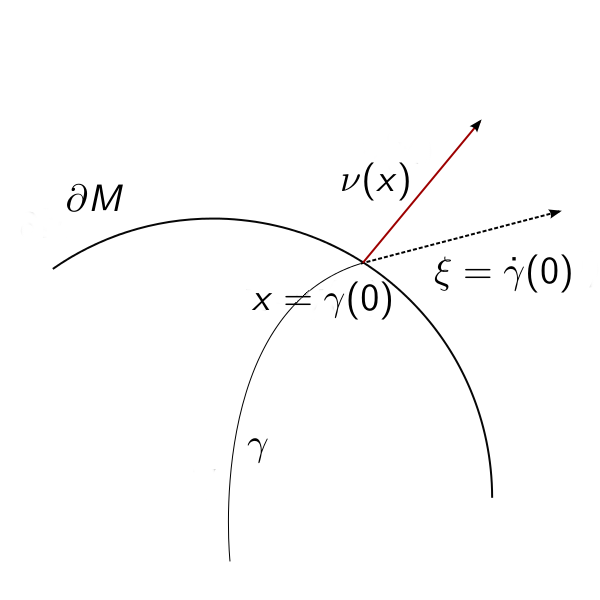}
 \caption{In $x \in \partial M$ we have $\< \xi, \nu(x) \> \geq 0$ (red arrow) at $\partial M$ and $\xi$ is the 
tangential vector (dotted arrow) in $x$. There is an ``Outflow'' with measured data $T(\gamma) = \int_{\gamma} \tn (y) \, dl(y) \geq 0$. In the case of an ``Inflow'',
$\< \xi, \nu(x) \> < 0$, the geodesic $\gamma$ would be outside from $M$ and hence $T(\gamma) = 0$.}
 \label{fig:2_Messrichtung}
\end{figure}
\begin{definition}[Inflow, Outflow]
 Let 
$$ \Omega M := \left\{ (x,\xi) \in TM : \|\xi\|_g^2 := g_{ij}(x)\xi^i\xi^j  = 1 \right\}.$$
We call
 \[
  \partial_+\Omega M := \left\{ (x,\xi) \in \Omega M : x \in \partial M, \< \xi, \nu(x) \> \geq 0 \right\}
 \]
\emph{Outflow} and
 \[
  \partial_-\Omega M := \left\{ (x,\xi) \in \Omega M : x \in \partial M, \< \xi, \nu(x) \> < 0 \right\}
 \]
\emph{Inflow}, where $\nu(x)$ is the outer normal at $M$ in $x$. We have
\[
 \partial \Omega M := \partial_+ \Omega M \cup \partial_- \Omega M
\]
and $\partial_{\pm}\Omega M$ are compact manifolds.
\end{definition}
\\[2mm]
The situation is illustrated in Figure \ref{fig:2_Messrichtung}.
The following lemma is proven in \cite{s}.\\
\begin{lemma}[{\cite[Lemma 4.1.1]{s}}]\label{L-tau-smooth}
Let $(M,g)$ be a compact, dissipative Riemannian manifold. Then the mapping $\tau_- : \partial_+ \Omega M \to \RR$ is a smooth function.\\[1ex]
\end{lemma}
The data acquired by TOF measurements can now be modeled as integrals of $\tn$ along geodesics associated with the metric $g^{\tn}$. To this end let $\tn \in C^\infty(M)$
and $\gamma^{\tn}_{x,\xi}$ the solution of the geodesic equation (\ref{gl:2_GeodDGL}) with respect to the metric $g^{\tn}$. The mapping $u^{meas} : \partial\Omega M \to \RR^+$ with
\be
 u^{meas}(x,\xi) = \begin{cases}
               0 & \text{ if } (x,\xi) \in \partial_-\Omega M \\
               \int_{\gamma^{\tn}_{x,\xi}} \tn (y) \, \mathrm{d}l(y) & \text{ if } (x,\xi) \in \partial_+\Omega M \\
              \end{cases}, \label{gl:2_Randwerte}
\ee
assigns a geodesic curve starting in $x\in \partial M$ with tangent $\xi\in T_x M$ its travel time until it leaves the domain. This is why we call $u^{meas}$ the \emph{TOF mapping}.
The inverse problem consists of computing the refractive index $\tn$ from $u^{meas}$. The forward operator is given by the ray transform
\be
R (\tn )(x, \xi) := \int_{\gamma^{\tn}_{x,\xi}} \tn (z) \, \mathrm{d}l(z), \qquad \mbox{for } (x,\xi) \in \partial_+ \Omega M .\label{gl:Operatorgleichung}
\ee
For $a \in C^\infty(M)$ we furthermore define the \emph{linearized ray transform}
\[
  R_a (\tn )(x, \xi) := \int_{\gamma^a_{x,\xi} } \tn (z) \, dl, \qquad \forall (x,\xi) \in \partial_+ \Omega M
 \]
where $\gamma^a_{x,\xi}$ is the solution of \eqref{gl:2_GeodDGL} with respect to the metric $g^a$ and initial values $\gamma^a_{x,\xi}(0) = x$ and 
$\dot \gamma^a_{x,\xi}(0) = \xi$. If $a = \tn$, we obviously have
 \[
  R_{\tn}(\tn ) = R(\tn ).
 \]
The inverse problem of determining the refractive index $\tn$ from TOF measurements finally means to find a solution of
\be\label{IP-TOF}
R (\tn) = u^{meas} .
\ee
Note that the crucial point is that the curve along which $R$ integrates depends on the integrand $\tn$ turning (\ref{IP-TOF}) into a highly nonlinear, ill-posed problem. 
\newline

\subsection{Mathematical properties of $R$ and $R_a$}

From now on we assume that the refractive index $\tn\in C^\infty (M)$ and $(M,g^{\tn})$ is a compact, dissipative Riemannian manifold (CDRM). We recall that this implies that for
any two points $x,y\in M$ there exists a unique, maximal geodesic $\gamma$ connecting $x$ and $y$ which at the same time is the shortest path between these points.
For proving the continuity of $R$ it is useful to introduce the phase function following the outlines of Guillemin and Sternberg \cite{gs}. Compare also
\cite[Section 2]{ps}.\\
For $s\in\RR$ and $\theta=(\cos\varphi,\sin\varphi)^\top \in S^1 := \{ \theta\in\RR^2 : |\theta|=1 \}$,
$\varphi\in [0,2\pi)$, we denote by $\gamma_{\theta,s} : [\tau_-,0]\to\RR^2$ the unique solution of the geodesic
equation (\ref{gl:2_GeodDGL}) with respect to the refractive index $\tn$ and initial values
\be\label{phase-setting}
 \gamma_{\theta, s}(0) = r \theta + s \theta^\perp \text{ and } \dot \gamma_{\theta,s}(0) = \theta.
\ee
Here, $r > \mathrm{diam}(M)/2$ and $\tau_- < 0$ is the unique parameter where $\gamma_{\theta,s}$ intersects the boundary $\partial M$ for the second time
and $\theta^\perp = (-\sin\varphi,\cos\varphi)^\top\in S^1$ is perpendicular to $\theta$.
The situation is illustrated in Figure \ref{fig:2_Phasenfunktion}.
It is now possible to define $\gamma_{\theta,s}$ as level curves of a phase function $\Phi$.\\
\begin{figure}[H]
 \centering
 \includegraphics[width=.5\textwidth]{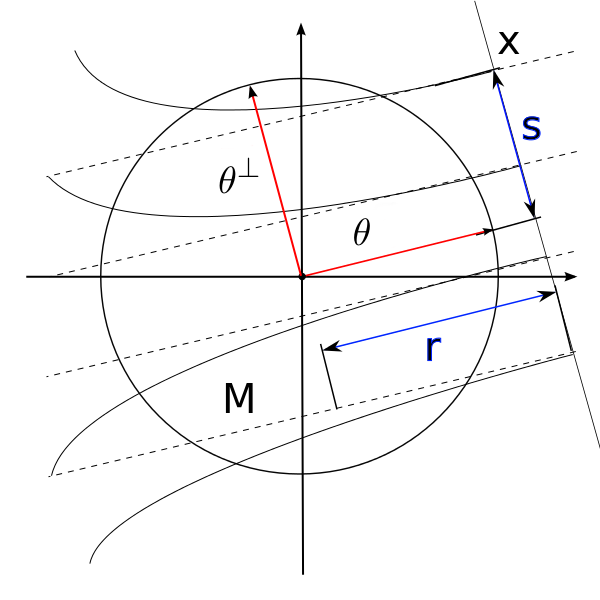}
 \caption{For a given direction $\theta \in S^1$ and $\theta^\perp \in S^1$ (red arrows) a point $x \in 
\RR^2$ is defined by the two parameters $r,s \in \RR$ (blue) such that $x = r \theta +  s \theta^\perp$. The black curves represent geodesics passing $M$ initiating from the straight line $r\theta + 
\spann{\theta^\perp}$ with vector of direction $\theta$.}
 \label{fig:2_Phasenfunktion}
\end{figure}
\begin{lemma}[{\cite[cf. Prop. 5.2]{gs}}]\label{lem:2_Phasenfunktion}
Let $g$ be a metric such that for $\theta \in S^1$ and $s, s' \in \RR$ with $s \neq s'$ the geodesic curves $\gamma_{\theta,s}$ and $\gamma_{\theta, s'}$ do not intersect. Then 
the matrix $\big( \partial_t \gamma_{\theta,s} (t),\partial_s \gamma_{\theta,s} (t)\big)$ is regular for fixed $\theta$ and there exists a \emph{phase 
function}
\[
 \Phi: M\times S^1 \to \RR,\qquad (y,\theta) \mapsto \Phi(y, \theta),
\]
such that the geodesic $\gamma_{\theta, s}$ is implicitly given and uniquely determined by
\[
 \Phi\left( y, \theta  \right) = s\quad \Longleftrightarrow \quad y \in \mathrm{Tr} (\gamma_{\theta, s}), \qquad s \in \RR.
\]
\end{lemma}
\hfill\\
In that sense geodesic curves can be interpreted as manifolds of constant phase of a wave field.\\
\begin{example}
Let $M:=\BB:= \{x\in\RR^2 : |x|\leq 1\}$ and consider $(\BB,g_0)$, i.e. the unit disk equipped with the Euclidean metric. Then the geodsics are straight lines and $\Phi : \BB\times S^1 \to \RR$ is
given by
\[   \Phi (y,\theta) = \langle y,\theta^\perp \rangle \]
and $\Phi (y,\theta)=s$ is the usual parametrization of straight lines in parallel geometry.
Here, $\theta=\theta (\varphi) = (\cos\varphi,\sin\varphi)^\top$ and $\theta^\perp = \theta (\varphi+\pi/2) = (-\sin\varphi,\cos\varphi)^\top$ for $\varphi \in [0,2\pi)$.
Please note that by (\ref{phase-setting}) the normal vector of a line is $\theta^\perp$ instead of $\theta$ as usual.\\[1ex]
\end{example}
In Euclidean geometry (i.e. $\tn=1$) it is easy to determine the boundary intersection points $p,q\in \partial M$
of the straight line starting at $x\in M$ with vector of direction $\theta\in S^1$. For $\tn\not=1$ the situation is different.
This is why we construct a so-called \emph{geodesic projection}, a mapping which determines the intersection point with the boundary and the corresponding tangential vector
of a geodesic curve that passes through $y\in M$.\\
%
%\\[2mm]
%The phase function can be understood as a geodesic projection in the sense of mapping an inner point $y \in M$ to a boundary point $x\in\partial M$ by given direction $\theta$ such that a geodesic with 
%initial values $(x,\theta)$ moves through $y$. We can catch the boundary point $x \in \partial M$ if we do not follow the geodesic to the straight line $r\theta + \spann{\theta^\perp}$, but to the 
%boundary $\partial M$,
%To do this we define for a given direction $\theta \in S^1$ the sets
%\bean
% \Lambda_\theta &:=& 
%\left\{ x = \gamma_{\theta, s}(\tau_0) \in \partial M : s \in \RR,\: \tau_0 = \max \big\{ \gamma_{\theta,s}^{-1}\big( \mathrm{Tr}(\gamma_{\theta,s}) \cap \partial M\big) \big\} \right\}\\
% \dot{\Lambda}_\theta & := & 
%\left\{ \xi = \dot{\gamma}_{\theta, s}(\tau_0) \in \partial M : s \in \RR,\: \tau_0 = \max \big\{ \gamma_{\theta,s}^{-1}\big( \mathrm{Tr}(\gamma_{\theta,s}) \cap \partial M\big) \big\} \right\}
%\eean
%and define for fixed $\theta\in S^1$ the mappings
%\begin{align*}
%& \Psi_\theta: \Lambda_\theta \to \RR\,,\qquad \Psi_\theta (x) := \Phi (x,\theta)\,,\\
%& \dot{\Psi}_\theta: \dot{\Lambda}_\theta \to 
%\end{align*}

\begin{definition}[Geodesic projection] \label{def:2_GeodProjektion}
 Adopt the assumptions of Lemma \ref{lem:2_Phasenfunktion}. The mapping $\Pi: S^1 \times \RR \to \partial_+ \Omega M$, given by
 \[
  (\theta, s) \mapsto \Pi(\theta, s) := (x,\xi)\,,
 \]
where $(x,\xi)\in\partial_+ \Omega M$ is uniquely determined by
\[   (x,\xi) = \big(\gamma_{\theta, s}(\tau_0),\dot{\gamma}_{\theta, s}(\tau_0)\big)\quad \mbox{with}
\quad \tau_0 = \max \big\{ \gamma_{\theta,s}^{-1}\big( \mathrm{Tr}(\gamma_{\theta,s}) \cap \partial M\big) \big\} ,
\]
is called \emph{geodesic projection} on the boundary $\partial_+ \Omega M$.\\[1ex]
\end{definition}
The construction of $\Pi (\theta,s)_1=x$ is illustrated in Figure \ref{fig:2_Projektion}. The geodesic projection links the 2D parallel geometry in CT to the
more general case of an inhomogeneous medium with given metric $g^{\tn}$; a fact which proves very useful for
defining the backprojection operator and the implementation of our numerical solution approach for (\ref{IP-TOF}). Please note that we need the assumptions
of Lemma \ref{lem:2_Phasenfunktion} and the existence of a phase function that $\Pi$ is well-defined.\\
\begin{figure}[H]
 \centering
 \includegraphics[width=.5\textwidth]{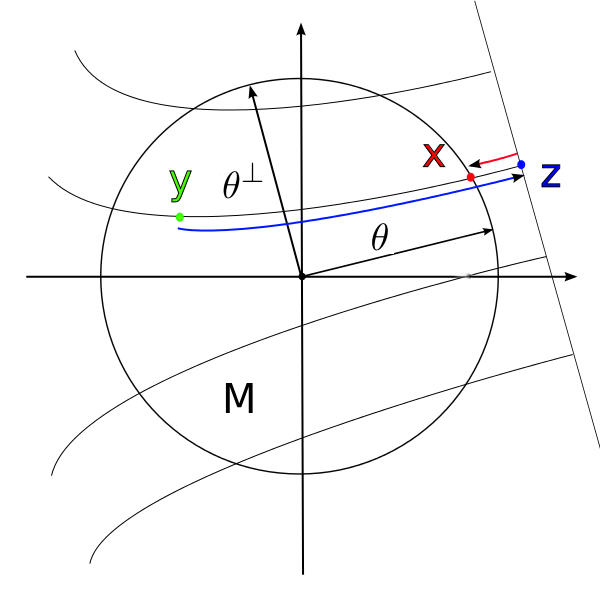}
 \caption{Illustration of the first component of the geodesic projection $\Pi_1$. At first the green point $y \in M$ is mapped by $\Phi$ (blue arrow) on $z = \Phi(y,\theta) \theta^\perp + r \theta$ (blue point) and subsequently it is mapped 
%by $\Psi^{-1}_\theta$ (red arrow) 
to the boundary $x\in \partial M$ (red point).}
\label{fig:2_Projektion}
\end{figure}
If $a\in C^\infty (M)$ generates a metric $g^a$ which is dissipative, then the linearized operator $R_a$ is continuous.\\
\begin{theorem}[Th. 4.2.1 from \cite{s}]\label{sat:2_StetLinVO}
 Let $a \in C^\infty(M)$ such that the corresponding metric $g^a_{ij}=a^2 \delta_{ij}$ turns $(M,g^a)$ into a CDRM. Then, the linearized ray transform 
 $R_a: L^2(M) \to L^2(\partial_+ \Omega M)$ is continuous.\\[1ex]
\end{theorem}
Under the assumption that the functions $a$ generating dissipative metrics $g^a$ are dense in $L^2 (M)$, then we even can prove the continuity
of the nonlinear operator $R$.\\
\begin{theorem}
Suppose that the subset
\[
\mathcal{M} := \big\{ a\in L^2 (M)\cap C^\infty (M) : (M,g^a) \mbox{is a CDRM}  \big\}
\]
is dense in $L^2 (M)$. Then $R: \mathcal{D}(R):=L^2 (M)\cap C^\infty (M) \to L^2 (\partial_+ \Omega M)$ is continuous.\\
\end{theorem}
\begin{proof}
Let $a\in\mathcal{D}(R)$ be arbitrary and $\{a_k\}\subset \mathcal{M}$ be a sequence in $\mathcal{M}$ such that $a_k \to a$ as $k\to\infty$.
Then we can estimate
\bean
 && \left\|R(a) - R(a_k)\right\|_{L^2(\partial_+ \Omega M)} \\
 &=& \left\| R(a) - R_{a_k}(a) + R_{a_k}(a) - R(a_k) \right\|_{L^2(\partial_+ \Omega M)} \\
 &\leq& \left\| R(a) - R_{a_k}(a) \right\|_{L^2(\partial_+ \Omega M)} + \left\| R_{a_k}(a-a_k) \right\|_{L^2(\partial_+ \Omega M)}.
 \eean
Since $a_k\in\mathcal{M}$ we have that $R_{a_k}$ is linear and continuous due to Theorem \ref{sat:2_StetLinVO} and thusly
\[
  \left\| R_{a_k}(a-a_k) \right\|_{L^2(\partial_+ \Omega M)} \leq \|R_{a_k}\|_{L^2(M) \to L^2(\partial_+ \Omega M)} \left\| a - a_k \right\|_{L^2(M)}.
 \]
To tackle the first term we assume $\gamma_{x,\xi}^a (t)$, $\gamma_{x,\xi}^{a_k}(t)$ to be respective solutions of the geodesic equation (\ref{gl:2_GeodDGL})
parametrized on $t\in [0,1]$. We obtain
\bean
 && \left\| R(a) - R_{a_k}(a) \right\|_{L^2(\partial_+ \Omega M)} \\
 &=& \left\| \int_{\gamma^a_{x,\xi}}a(z) \, \mathrm{d}l(z) - \int_{\gamma^{a_k}_{x,\xi}}a(z) \, \mathrm{d}l(z) \right\|_{L^2(\partial_+ \Omega M)} \\
 &=& \left\| \int_0^1 a(\gamma^a_{x,\xi}(t))|\dot \gamma^a_{x,\xi}(t) |\, \mathrm{d}t 
- \int_0^1 a(\gamma^{a_k}_{x,\xi}(t)) | \dot \gamma^{a_k}_{x,\xi}(t) |\, \mathrm{d}t \right\|_{L^2(\partial_+ \Omega M)} \\
 &\leq& \sup_{y \in M} |a(y)|\,  \left\| \int_0^1 |\dot \gamma^a_{x,\xi}(t) - \dot \gamma^{a_k}_{x,\xi}(t) |\, \mathrm{d}t \right\|_{L^2(\partial \Omega M)}.
\eean
Let $\varepsilon>0$ be arbitrary. For $k$ sufficiently large we deduce from Lemma \ref{kor:2_StetigeGeodaeten} and Lemma \ref{L-tau-smooth}
\[   \sup_{(x,\xi)\in TM} \sup_{t\in [0,1]} \|\dot \gamma^a_{x,\xi}(t) - \dot \gamma^{a_k}_{x,\xi}(t) \| < \frac{\varepsilon}{2\sup_{y\in M} |a(y)|} . \]
A sufficiently large $k$ furthermore assures that
\[   \|a-a_k\|_{L^2 (M)} < \big( 2\|R_{a_k}\|_{L^2(M) \to L^2(\partial_+ \Omega M)}\big)^{-1} \,\varepsilon.  \]
Putting all this together we conclude
\bean
\left\|R(a) - R(a_k)\right\|_{L^2(\partial_+ \Omega M)} &\leq& 
\left\| R(a) - R_{a_k}(a) \right\|_{L^2(\partial_+ \Omega M)} + \left\| R_{a_k}(a-a_k) \right\|_{L^2(\partial_+ \Omega M)} \\
&\leq& \varepsilon /2 + \varepsilon / 2 = \varepsilon
\eean
if only $k$ is sufficiently large. This proves the theorem.
\end{proof}
\\[2mm]
Unfortunately by now there is no proof that $\mathcal{M}$ is dense in $L^2 (M)$ or not.\\[1ex]
For completeness we give two existence and uniqueness results which can be found in \cite{s2}.\\
\begin{theorem}[Th. 1.1.1 and Th. 1.2.1 from \cite{s2}]
Let $a \in C^\infty(M)$ such that it induces a simple Riemannian metric.
\begin{itemize}
 \item[i)] If the measured data $u^{meas}$ is generated by a refractive index $\tn \in C^2(M)$, then the operator equation
 \[
 u^{meas} = R_a \tn
\]
has a unique solution $\tn \in L^2(M) \cap C^2(M)$.
\item[ii)] If the measured data $u^{meas}$ is generated by a refractive index $\tn \in C^4(M)$, then the nonlinear operator equation
\[
 u^{meas} = R (\tn)
\]
has a unique solution $\tn \in L^2(M) \cap C^4(M)$.\\
\end{itemize}
\end{theorem}

\subsection{The method of characteristics}

Our numerical solution scheme for (\ref{IP-TOF}) is iterative and demands for evaluating the forward operator, i.e. the computation of the TOF measurements
for a given metric, in each iteration step. We do this using the fact that the TOF function $u$ (\ref{TOF-mapping}) satisfies a partial differential equation
where the differential operator is the so called \emph{geodesic vector field}. 
\newline
\begin{definition}[Geodesic flow] \label{def:3_GeodaetischerFluss}
 Let $(M,g)$ be a Riemannian manifold and $u: T^0M \to \RR^+$ the TOF mapping (\ref{TOF-mapping}) . We call $Hu:T^0M \to \RR$ defined by
\[
 Hu(x,\xi) := \xi^i \frac{\partial u}{\partial x^i} - \Gamma^{i}_{jk}(x) \xi^j \xi^k \frac{\partial u}{\partial \xi^i} \text{, for all } (x,\xi) \in T^0M,
\]
\emph{geodesic vector field}.
\end{definition}
\\[2mm]
There s a fundamental connection between the geodesic vector field and the refractive index regarding our measure geometry.
\newline
\begin{theorem}\label{sat:3_Transportgleichung}
Let $(M,g)$ be a Riemannian manifold with the metric tensor given as $g_{ij}(x) = g^{\tn}_{ij} = \tn^2(x) \delta_{ij}$. Then for all $(x,\xi) \in T^0M$ the 
transport equation
\[
 Hu(x,\xi) = \tn(x)
\]
holds true.
\end{theorem}
\begin{proof}
% Der Beweis orientiert sich an den Ausführungen von \cite{Sharafutdinov1994}. Sei $(x,\xi) \in T^0M$ und $\gamma = \gamma_{x,\xi}:I=[\tau_-(x,\xi),0]$ die Geodäte mit $\gamma(0) = x$ und $\dot 
% \gamma(0) = \xi$. $\tau_-(x,\xi)$ sei der Zeitpunkt des ersten Schnitts mit dem Rand von $M$. Für $t_0 \in I$ gilt der Zusammenhang $\gamma_{\gamma(t_0),\dot\gamma(t_0)}(t) = \gamma(t+t_0)$ für 
%alle 
% $t \in J:=\{ s \in I : s+t_0 \in I \}$, denn $\gamma(0 + t_0) = \gamma_{\gamma(t_0),\dot\gamma(t_0)}(0) = \gamma(t_0)$ und $\dot \gamma(0 + t_0) = \dot \gamma_{\gamma(t_0),\dot\gamma(t_0)}(0) = 
%\dot 
% \gamma(t_0)$. Es gilt dann
% \[
%  u\left( \gamma(t_0),\dot\gamma(t_0) \right) = \int_{\tau_-(x,\xi)}^{t_0} \tn(\gamma(t)) \| \dot\gamma(t) \| \, dt.
% \]
% Leiten wir die Gleichung nun an der Stelle $t_0 = 0$ ab, so erhalten wir
% \[
%  \frac{\partial u}{\partial x^i}\dot \gamma^i + \frac{\partial u}{\partial \xi^i} \ddot \gamma^i = \tn (\gamma(0)) \| \dot \gamma(0) \|.
% \]
% Unter der Ausnutzung von $\dot\gamma(0) = \xi$ und der geodätischen Differentialgleichung \eqref{gl:2_GeodDGL} $\ddot \gamma + \Gamma^{i}_{jk}\dot\gamma^{j}\dot \gamma^{k} = 0$ erhalten wir
% \[
%  \tn (x) = \xi^i \frac{\partial u}{\partial x^i} - \Gamma^{i}_{jk}(x) \xi^j \xi^k \frac{\partial u}{\partial \xi^i} = Hu(x,\xi).
% \]
A proof can be found in \cite[Section 1.2]{s}.
\end{proof}
\newline
%
%%%%%%%%%% EVTL. RAUSNEHMEN %%%%%%%%%%%%%%%%%%%%%
\begin{remark} \label{bem:3_Fluss}
The differential operator
 \[
  H = \xi^i \frac{\partial}{\partial x^i} - \Gamma^{i}_{jk} \xi^j\xi^k \frac{\partial}{\partial \xi^i}
 \]
has a specific geometrical meaning. Let $G(t;\cdot,\cdot) : TM \to TM$ be the solution of
\[
 G' (t;x,\xi) = H(G(t;x,\xi))\,,\qquad G(0;x,\xi) = (x,\xi)\,.
\]
Then we have
\[
 G(t;x,\xi) = \left( \gamma_{x,\xi}(t), \dot \gamma_{x,\xi}(t) \right)
\]
for a geodesic $\gamma_{x,\xi}$. This is the \emph{geodesic flow} associated with the vector field $H$. 
\end{remark}
%%%%%%%%%%%%%%%%%%%%%%%%%%%%%%%%%%%%%%%%%%%%%%%%%
\\[2mm]
\begin{corollary} \label{kor:3_Transportgleichung}
Let $\tn\in C^\infty (M)$ be positive and the metric $g$ on $M$ defined by
 \[
 g^{\tn}(x) = g_{ij}^{\tn}(x)dx^i \, dx^j
\]
with metric tensor
\[
 g_{ij}^{\tn}(x) = \tn ^2(x)\delta_{ij}.
\]
Then
\bea
 Hu(x,\xi) &=& \xi^i \frac{\partial u}{\partial x^i}(x,\xi) + \tn^{-1}(x) \left( \frac{\partial \tn}{\partial x^i}(x) \|\xi\|^2 - 2 \xi^i \< \xi, \nabla \tn(x) \> \right) \frac{\partial u}{\partial 
\xi^i}(x,\xi) \nonumber \\
&=& \tn(x) \label{gl:3_Transportgleichung}
\eea
holds true for all $(x,\xi) \in T^0M$.\\[1ex]
\end{corollary}
\begin{proof}
 Let $(x,\xi) \in T^0M$. For $i=1$ the Christoffel symbols are given by
 \bean
 \left( \Gamma^1_{ij}(x) \right)_{i,j=1}^2 = \tn^{-1}(x)\begin{pmatrix}
                                              \frac{\partial \tn (x)}{\partial x_1} & \frac{\partial \tn(x)}{\partial x_2} \\
                                              \frac{\partial \tn (x)}{\partial x_2} & -\frac{\partial \tn(x)}{\partial x_1}
                                             \end{pmatrix}.
 \eean
 For the partial derivatives we have
 \bean
  \Gamma^1_{jk}(x)\xi^j\xi^k &=& \tn^{-1}(x)\begin{pmatrix}
                                         \xi_1 & \xi_2
                                        \end{pmatrix}
\begin{pmatrix}
                                              \frac{\partial \tn(x)}{\partial x_1} & \frac{\partial \tn(x)}{\partial x_2} \\
                                              \frac{\partial \tn(x)}{\partial x_2} & -\frac{\partial \tn(x)}{\partial x_1}
                                             \end{pmatrix}
                                             \begin{pmatrix}
                                              \xi_1 \\ \xi_2
                                             \end{pmatrix} \\
  &=& \tn^{-1}(x)\left( \frac{\partial \tn(x)}{\partial x_1}\xi_1^2 + 2 \frac{\partial \tn(x)}{\partial x_2} \xi_1 \xi_2 - \frac{\partial \tn(x)}{\partial x_1} \xi_2^2  \right) \\
  &=& \tn^{-1}(x)\left( 2\frac{\partial \tn(x)}{\partial x_1}\xi_1^2 + 2 \frac{\partial \tn(x)}{\partial x_2} \xi_1 \xi_2 - \frac{\partial \tn(x)}{\partial x_1}\|\xi\|^2 \right) \\
  &=& \tn^{-1}(x)\left( 2 \xi_1 \left[ \xi_1\frac{\partial \tn(x)}{\partial x_1} + \xi_2\frac{\partial \tn(x)}{\partial x_2} \right] - \frac{\partial \tn(x)}{\partial x_1}\|\xi\|^2 \right) \\
  &=& \tn^{-1}(x)\left( 2 \xi_1 \< \xi, \nabla \tn(x)\> - \frac{\partial \tn(x)}{\partial x_1}\|\xi\|^2 \right).
 \eean
In the same way we compute
\[
 \Gamma^2_{jk}(x)\xi^j\xi^k = \tn^{-1}(x)\left( 2 \xi_2 \< \xi, \nabla \tn(x)\> - \frac{\partial \tn(x)}{\partial x_2}\|\xi\|^2 \right)
\]
and the proposition follows immediately.
\end{proof}
\\[2mm]
\begin{remark}
In Theorem \ref{sat:3_Transportgleichung} we transformed the inverse problem (\ref{IP-TOF}) into a parameter identification problem for the transport 
equation with boundary conditions \eqref{gl:2_Randwerte}. Again we recognize that the problem is highly nonlinear, because the refractive index $\tn$ 
appears as source-term as well as parameter in the differential operator.\\[2mm]
\end{remark}
According to the transport equation \eqref{gl:3_Transportgleichung} we want to develop a method to compute the geodesic curves for a given refractive index. To
this end we use the method of characteristics.
Let $\operatorname{Graph}(u) := \left\{ \left(x,\xi,u(x,\xi)\right) : (x,\xi) \in TM \right\}$ be the graph of $u$ and
\[
 \left( x(t), \xi(t), z(t) \right) := \left( x(t), \xi(t), u\left( x(t), \xi(t) \right) \right)
\]
a curve in $\operatorname{Graph}(u)$ with $t\in I$, $I$ a compact interval. Differentiating with respect to $t$ yields
\[
 \left( \dot x(t), \dot \xi (t), \dot z (t) \right) = \left( \dot x (t), \dot \xi (t), \frac{\partial u(x(t),\xi(t))}{\partial x^i} \dot x^i(t) + \frac{\partial u(x(t),\xi(t))}{\partial \xi^i} \dot 
\xi^i(t) \right).
\]
Then $\left( \nabla_x u (x(t),\xi(t)), \nabla_\xi u (x(t),\xi(t)), -1 \right)$ is a normal of the curve at $t$ since
\bean
&& \begin{pmatrix} \nabla_x u (x(t),\xi(t)) \\ \nabla_\xi u (x(t),\xi(t)) \\ -1 \end{pmatrix} \left( \dot x (t), \dot \xi (t), \frac{\partial u(x(x(t),\xi(t))}{\partial x^i} \dot x^i(t) + 
\frac{\partial u(x(x(t),\xi(t))}{\partial \xi^i} \dot \xi^i(t) \right) \\ 
&=& \nabla_x u (x(t),\xi(t)) \cdot \dot x(t) + \nabla_\xi u(x(t),\xi(t)) \cdot \dot \xi(t) - \frac{\partial u(x(t),\xi(t))}{\partial x^i} \dot x^i(t) \\
&& - \frac{\partial u(x(t),\xi(t))}{\partial 
\xi^i} \dot \xi^i(t) \\
&=& 0.
\eean
From $Hu = \tn$ we get that $\big( \xi(t), \big(-\Gamma^{i}_{jk}(x(t)) \xi^j(t) \xi^k(t)\big)_i, \tn(x(t)) \big)$ is a tangent vector, because
\bean
&& \begin{pmatrix} \nabla_x u (x(t),\xi(t)) \\ \nabla_\xi u (x(t),\xi(t)) \\ -1 \end{pmatrix} \left( \xi(t), \left(-\Gamma^{i}_{jk}(x(t)) \xi^j(t) \xi^k(t)\right)_i, \tn(x(t)) \right) \\ 
&=& \xi \cdot \nabla_x u(x(t),\xi(t)) - \Gamma^{i}_{jk}(x(t))\xi^i(t)\xi^j(t) \frac{\partial u(x(t),\xi(t))}{\partial \xi^i} - \tn(x(t)) \\
&=& Hu(x(t),\xi(t)) - \tn(x(t)) \\
&=& 0.
\eean
If we choose a boundary point $x_0 \in \partial M$ and a direction $0\not=\xi_0 \in T_{x_0}M$, we finally obtain the initial value problem
\be
\text{(IVP)} \begin{cases}
       \dot x_i(t) = \xi_i(t) & \text{ for } t \in I,\, i=1,2 \\
       \dot \xi_i(t) = -\Gamma^{i}_{jk}(x(t)) \xi^j(t) \xi^k(t) & \text{ for } t \in I,\, i=1,2 \\
       \dot z(t) = \tn(x(t)) & \text{ for } t \in I \\
       & \\
       x(0) = x_0 & \\
       \xi(0) = \xi_0 & \\
       z(0) = 0 & 
      \end{cases}. \label{gl:3_AWP}
\ee
\newline
The solution is a characteristic, i.e. a curve on $\operatorname{Graph}(u)$, which is the union of all characteristics. But the characteristic curves of $H$ are just the geodesics of $(M,g^{\tn})$. Solving (\ref{gl:3_AWP}) thus follows a geodesic curve starting at the boundary until it meets $\partial M$ again say at $t=t_1$.
Then $z(t_1)$ is the TOF of the ultrasound signal propagating from $x_0$ in direction $\xi_0$ until it leaves the boundary $\partial M$ at $x(t_1)$. 
\newline
%

%%%%%%%%%%%%%%%%%%%%%%%%%%%%%%%%%%%%%%%%%%%%%%%%%%%%%%%%%%%%%%%%%%%%%%%%%%%%%%%%%

\section{A regularization method for (\ref{IP-TOF})}

%%%%%%%%%%%%%%%%%%%%%%%%%%%%%%%%%%%%%%%%%%%%%%%%%%%%%%%%%%%%%%%%%%%%%%%%%%%%%%%%%

We intend to solve the nonlinear inverse problem (\ref{IP-TOF}) iteratively by a steepest descent method for a Tikhonov functional where we use the linearized
operator $R_a$. In a first subsection we define this functional, the iteration scheme and some of its properties. The second subsection then is devoted to deal
with implementation issues of this method.\\

%%%%%%%%

\subsection{The Tikhonov functional and its minimization}

%%%%%%%%

%
Again we assume $M\subset \RR^2$ to be compact. For $1\leq p < \infty$ and $\alpha > 0$ we consider the Tikhonov functional 
$J_\alpha : X_p := L^p (M)\cap C^\infty (M) \to \RR$ defined by
\be\label{def:4_TikhonovFunktional}
   J_\alpha (f) = \frac{1}{2}\| R(f) - u^{meas} \|_{Y}^2 + \frac{\alpha}{p}\| f - 1 \|_{L^p(M)}^p.
\ee
Here we used the notation $Y:=L^2(\partial_+ \Omega M)$. The penalty term $\frac{1}{p}\| f - 1 \|_{L^p(M)}^p$ thereby ensures that
the solution gives a refractive index that does not vary too strongly from $\tn=1$, i.e. a homogeneous medium.
This is a supposed a priori information about the exact solution which we want to incorporate. Furthermore this term of course
yields stability of the solution process. The choice $p=1$ furthermore allows also for sparse solutions. Because of the
nonlinearity we can not expect that there exists a unique minimizer of $J_\alpha (f) \leadsto \min$. Usually
a stationary point $f^*$ of $J_\alpha$ is searched by the condition $0\in \partial J_\alpha (f^*)$. There we face
a further problem: the computation of a G\^{a}teaux derivative $R' (f)$ of $R$. The computation is still object of current
research. This is why we linearize the problem by using $R_a$ for fixed $a\in X_p$ instead of the nonlinear operator $R$.
In this way a second Tikhonov functional $J_\alpha^a : L^p (M)\to \RR$ comes up which is defined by
  \[
   J_\alpha^a( f) = \frac{1}{2}\| R_a(f) - u^{meas} \|_{Y}^2 + \frac{\alpha}{p}\| f - 1 \|_{L^p(M)}^p
  \]
with $a\in C^\infty (M)$ fixed. Since $R_a$ is linear, it is simple to compute the subdifferential $\partial J_\alpha^a$.
The last ingredient we need for doing so is the adjoint $R_a^*$ of $R_a$.\\
\begin{lemma}\label{sat:4_AdjOperator}
 Let $a \in C^\infty (M)$ be fixed. Then the adjoint operator $R^*_a:Y \to L^{p^*} (M)$ is given by
 \be\label{Ra-adjoint}
  R_a^*\omega (y) = \int_{S^1} \omega\big( \Pi (\theta, \Phi(y,\theta) )\big)\, 
                    \Big|\mathrm{det} \frac{\partial}{\partial (\theta, s)} \Pi \big(\theta,\Phi ( y,\theta ) \big) \Big|\,\mathrm{d} \theta\,,\qquad y\in M,
 \ee
 where $\Pi: S^1 \times [-r,r] \to \partial_+ \Omega M$ is the geodesic projection from Definition \ref{def:2_GeodProjektion},
 \[ \mathrm{det} \frac{\partial}{\partial (\theta, s)} \Pi \big(\theta,\Phi ( y,\theta ) \big)  \]
 denotes the Jacobian of $\Pi (\theta, s)$ evaluated at $s=\Phi ( y,\theta )$,
 $r=\mathrm{diam}(M)/2$ and
 $p^*$ satisfies $(p^*)^{-1} + p^{-1} = 1$ for $p\in (1,\infty)$ and $p^*=\infty$ for $p=1$.\\
\end{lemma}
\begin{proof}
Let $f \in L^p (M)$ and $\omega \in Y$. Then we have
\bean
\< R_a f(x,\xi), \omega(x,\xi) \> _Y &=& \int_{\partial_+ \Omega M} R_a f(x,\xi) \omega(x,\xi) \, \mathrm{d}x\wedge \mathrm{d}\xi\\
&&\hspace*{-2.5cm} = \int_{\partial_+ \Omega M} \int_{\gamma^a_{x,\xi}} f (z) \omega(x,\xi) \, \mathrm{d}l(z) \,\mathrm{d}x\wedge \mathrm{d}\xi \\
&&\hspace*{-2.5cm} = \int_{S^1} \int_{-r}^r \int_M f (y) \delta \big( s-\Phi (y,\theta ) \big) \omega\big( \Pi(\theta,s) \big) \,
\mathrm{d}y \,\Big|\mathrm{det} \frac{\partial}{\partial (\theta, s)} \Pi (\theta,s) \Big|\,\mathrm{d}s \, d \theta \\
&&\hspace*{-2.5cm} = \int_{M} f(y) \int_{S^1} \int_{-r}^r \omega\big( \Pi(\theta,s) \big) \delta \big( s-\Phi (y,\theta ) \big)\,
\Big|\mathrm{det} \frac{\partial}{\partial (\theta, s)} \Pi (\theta,s) \Big|\,\mathrm{d}s \,\mathrm{d}\theta \, \mathrm{d} y \\
&&\hspace*{-2.5cm} = \int_{M} f (y) \int_{S^1} \omega\big( \Pi(\theta,\Phi (y,\theta)) \big) \,
\Big|\mathrm{det} \frac{\partial}{\partial (\theta, s)} \Pi \big(\theta,\Phi ( y,\theta ) \big) \Big|\,\mathrm{d}\theta \, dy \\
&&\hspace*{-2.5cm} = \< f, R_a^* \omega \> _{L^p \times L^{p^*}}.
\eean
\end{proof}
\\[2mm]
Note that $R_a^*$ is an analogon to the classical backprojection operator, since it integrates along all geodesics passing a given point $y\in M$.
The numerical approximation of the geodesic projection $\Pi(\theta,\Phi (y,\theta))$ will prove the part of our method which is most time consuming.\\
\begin{example}
Consider again the 2D parallel geometry in the unit disk $M=\BB$, equipped with the Euclidean geometry $g_0$. Then the geodesics are straight lines
and the geodesic projection $\Pi : S^1 \times (-1,1) \to \partial_+ \Omega \BB$ is given by
\[   \Pi (\theta,s) = (x,\xi) = ( s \theta + \sqrt{1-s^2} \theta^\perp, \theta^\perp )  , \]
where $\theta=\theta (\varphi) = (\cos\varphi,\sin\varphi)^\top$, $\theta^\perp = \theta(\varphi+\pi/2) = (-\sin\varphi,\cos\varphi)^\top$. In that sense
we can equivalently formulate $\Pi$ as mapping $\Pi : [0,2\pi) \times (-1,1)\to [0,2\pi) \times [0,2\pi)$ by
\[   \Pi (\varphi,s) = \Big( \varphi + \arccos s, \big( \varphi + \frac{\pi}{2} \big)\, \mathrm{mod}\,2\pi \Big)  \]
and obtain
\[
\Big| \mathrm{det} \frac{\partial}{\partial (\theta, s)} \Pi (\theta,s) \Big| = \frac{1}{\sqrt{1-s^2}}.
\]
Since in $(\BB,g_0)$ the phase function is given as $\Phi (y,\theta) = \langle y,\theta^\perp \rangle$, we obtain the usual backprojection operator in $(\BB,g_0)$
\[
R^* \omega (y) = \int_{S^1} \omega \big( \Pi(\theta, \langle y,\theta \rangle ) \big) \,\frac{\mathrm{d}\theta}{\sqrt{1- \langle y,\theta \rangle ^2}}.
\]
Note that in the last identity we substituted $\theta^\perp$ by $\theta$ to stick on the conventional notation. The weight $(1- \langle y,\theta \rangle^2)^{-1/2}$
comes from the fact that the lines $\ell (x,\xi)$ here are parametrized by a boundary point $x\in \partial\BB$ and a unit tangential vector $\xi$ in $x$ with $\langle \xi,x \rangle \geq 0$
instead of the conventional parametrization by a normal vector $\theta$ and offset $s\in\RR$.\\[1ex]
\end{example}
If we would just minimize $J_\alpha^a$ for fixed $a$ we would completely neglect the influence of the integrand $f$ to the integration curves. Thus the idea
of our algorithm is to use a steepest descent method which uses the actual iterate $\tn_{k-1}$ to compute the new one $\tn_k$ by minimizing $J_\alpha^{\tn_{k-1}}$.
This of course means to compute a new set of geodesic curves in each iteration step.
This leads to the following, adaptive iteration scheme.\\
\begin{algorithm}[Iterative, adaptive minimization method] \label{alg:4_IterAdapMin} \newline
\emph{Input:} initial value $\tn_0$, sequence of regularization parameters $\alpha = 
\left(\alpha_k\right)_{k=1,2,\ldots} \subset \RR^+$.%, Abbruchschranke $\varepsilon > 0$
 \begin{itemize}
  \item[(S0)] Compute for $\tn_0$ the set of geodesics $\mG_0 = \mG_{\tn_{0}}$ and set $J_{\alpha_0}^{0}(\tn_{0}) = J_{\alpha_0}^{\tn_{0}}(\tn_{0})$.
  \item[(S1)] Repeat for $k = 1, 2, \ldots$
  \begin{itemize}
   \item[(a)] Minimize $J_\alpha^{n_{k-1}}$, i.e. compute 
       $$\tn_{k} := \arg \min_{f\in L^p (M)} J_{\alpha_{k-1}}^{\tn_{k-1}}(f).$$
   \item[(b)] Compute the new set of geodesics $\mG_k := \mG_{\tn_{k}}$ and evaluate $J_{\alpha_k}^k(\tn_{k}) = J_{\alpha_k}^{\tn_{k}}(\tn_{k})$.
   %\item[(c)] Abbruch, falls $|J_\alpha^{k-1}(\tn_{k-1}) - J_\alpha^{k}(\tn_{k}) | \leq \varepsilon$. \todo{Abbruchkriterium aufweichen}
  \end{itemize}
 \end{itemize}
 \emph{Output:} $\tn_{k^*}$ for a stopping index $k^*$, set of geodesics $\mG_{k^*}$.
\end{algorithm}
\\[2mm]
The minimization step (S1a) is done iteratively using the Landweber method, compare e.g. \cite{skhk}. Let $\tn_k$ be the actual iterate in step (S1), then
this iteration reads as
\bean
 \tn_k^0 &=& \tn_k \\
 \tn_{k}^{l+1} &=& \tn_k^{l} - \mu_l \left( A_k^* (A_k\tn_k^l - u^{meas}) + \frac{\alpha_k}{p} \tn_{k,l}^* \right), \qquad l=0,1,2,\ldots \\
 \tn_{k+1} &=& \tn_k^{l^*}
\eean
with an element $\tn_{k,l}^*\in \partial ( \| \tn_k^l - 1 \|_{L^p(M)}^p )\subset L^{p^*} (M)$, $A_k := R_{\tn_k}$, regularization parameter $\alpha_k > 0$, step size $\mu_l > 0$, $l = 
0,1,\ldots$ and reasonable stopping index $l^*\in\NN$. For $1<p<\infty$ we have
\[
 \partial \left( \| \tn_k^l - 1 \|_{L^2(M)}^p \right) = p\| \tn_k^l - 1 \|_{L^p(M)}^{p-2} \cdot (\tn_k^l-1).
\]
for $\tn_k^l \in L^p (M)$. For $p=1$ we use the standard soft threshold algorithm from \cite{daub}. This yields the following version of step (S1) from Algorithm \ref{alg:4_IterAdapMin}
as it is implemented.\\
\begin{algorithm}[Steepest descent method for $J_{\alpha_k}^{\tn_k}$] \label{alg:4_AbstiegLinTikh} \newline
 \emph{Input:} Current iterate $\tn_k$, regularization parameter $\alpha_k > 0$.%, Abbruchschranke $\delta_k > 0$.
 \begin{itemize}
  \item[(S0)] Define the operator $R_k := R_{\tn_k}$ and its adjoint $R^*_k = R^*_{\tn_k}$ by (\ref{Ra-adjoint}).
  \item[(S1)] Set $\tn_k^0 := \tn_k$ and iterate for $l = 0,1,2, \ldots,l^*$
  \begin{itemize}
   \item[(a)] Compute the residuum $r^{l-1} = R_k \tn_k^{l-1} - u^{meas}$ and stop if $\| r^{l-1} \|_Y = 0$.
   \item[(b)] Evaluate the adjoint operator $r_*^{l-1} = R^*_k r^{l-1}$.
   \item[(c)] Determine a new search direction by
               $$\Delta^{l-1} = r_*^{l-1} + \alpha_k \| \tn_k^{l-1} - 1 \|_{L^p(M)}^{p-2} \cdot (\tn_k^{l-1}-1)$$
   in case $1<p<\infty$ or use soft thresholding for $p=1$.
   \item[(d)] Choose a decent step size $\mu_{l-1} > 0$ and update
   \[
    \tn_k^l := \tn_k^{l-1} - \mu_{l-1} \Delta^{l-1}.
   \]
   \end{itemize}
  \item[(S2)] Set $\tn_{k+1} = \tn_{k}^{l^*}$.
 \end{itemize}
 \emph{Output:} new iterate $\tn_{k+1}$.
\end{algorithm}
\\
\begin{remark} \label{bem:4_Hilfsproblem}
 \begin{itemize}
\item The evaluation of the linearized forward operator $R_k$ in step (S1a) is computed quickly and depends on the discretization of the geodesics and the evaluation of $\tn_k$.
\item The computation of the adjoint in step (S1b) however is very time consuming. The difficulty lies in the interpolation of the geodesics. The challenge is to find the geodesics that pass through a 
point $y \in M$, emitted from a direction $\xi \in T\partial M$ at the boundary.
\item The step size $\mu_{l} > 0$ in step (S1d) can be determined by a line-search algorithm.
\item For convergence of the Landweber method in Algorithm \ref{alg:4_AbstiegLinTikh} to the minimum norm solution we refer to \cite{SCHOEPFER;LOUIS;SCHUSTER:06,skhk}.
 \end{itemize}
\end{remark}
\vspace{2mm}
Our aim is to prove that Algorithms \ref{alg:4_IterAdapMin}, \ref{alg:4_AbstiegLinTikh} generate a sequence of iterates $\{\tn_k\}$ of decreasing values $J_\alpha(\tn_k)$ 
for the Tikhonov functional $J_\alpha$ (\ref{def:4_TikhonovFunktional}). To this end we define the mapping $\mT:L^p (M) \to L^p (M)$ by
\[
 a \mapsto \arg \min_{f\in L^p (M)} J^{a}_\alpha (f) = \arg \min_{f \in L^p (M)} \left\{ \frac{1}{2} \left\| R_{a} f - u^{meas} \right\|_{L^p (M)}^2 + 
 \frac{\alpha}{p} \left\| f - 1 \right\|_{L^p (M)}^p \right\}.
\]
Agorithm \ref{alg:4_IterAdapMin} reads then as
\be\label{gl:4_IterationT}
\tn_{k+1} := \mT(\tn_{k}) = \arg \min_{f \in L^p (M)} J^{\tn_k}_\alpha (f) 
\ee
with given initial value $\tn_0\in L^p (M)$. Assuming that $\mT$ is a contraction, then Banach's fixed point theorem says that $\{\tn_k\}$ converges to a unique $\tn^*$ satisfying
\[
  \tn^* = \mT(\tn^*).
\]
\begin{lemma} \label{kor:4_MinimaleigenschaftFixpunkt}
Let $\mT : L^p (M)\to L^p (M)$ be a contraction and $\tn^*$ be the unique fixed point of $\mT$ in $L^p (M)$. Then $\tn^* = \lim_{k \rightarrow \infty} \tn_k$ and
$$J_\alpha(\tn^*) \leq J_\alpha^{\tn^*}(f)$$
holds true for all $f \in L^p (M)$ and $\alpha>0$.\\[1ex]
\end{lemma}
\begin{proof}
Since
 \[
  \tn^* = \mT(\tn^*) = \arg \min_{f\in L^p (M)} J^{\tn^*}_\alpha(f),
 \]
we have
\[
 J_\alpha (\tn^*) = J^{\tn^*}_\alpha(\tn^*) \leq J^{\tn^*}_\alpha(f)
\]
for all $f \in L^p (M)$.
\end{proof}
\newline
Note that in general we do not have
 \be\label{gl:4_MinimaleigenschaftGeodaeten}
 J_\alpha(\tn^*) \leq J_\alpha^{a}(\tn^*) \text{ for all } a\in X_p
 \ee
but assuming that (\ref{gl:4_MinimaleigenschaftGeodaeten}) holds at least in a neighborhood of $\tn^*$, then Algorithm (\ref{alg:4_IterAdapMin}) in fact generates a decreasing
sequence $J_\alpha (\tn^k)$.\\
\begin{lemma}\label{sat:4_LokaleKonvergenz}
Let $\{ \tn_{k} \}_{k = 0,1,\ldots}$ be the sequence of iterates generated by Algorithm \ref{alg:4_IterAdapMin}. We assume that there exists a $k^* \in \NN$ such that
%  \begin{itemize}
%   \item[i)] $J_\alpha(\tn_k) \leq J_\alpha^{\tn_k}(f)$ for all $f\in L^p (M)$ and
$$J_\alpha(\tn_k) \leq J_\alpha^{a}(\tn_k)\qquad \mbox{for all } a\in X_p$$
whenever $k \geq k^*$. Then for $k\geq k^*$ the Tikhonov functional $J_\alpha (\tn_k)$ is monotonically decreasing, we have
$$J_\alpha(\tn_{k+1}) \leq J_\alpha(\tn_{k}),$$ 
for all $k \geq k^*$.
\end{lemma}
\begin{proof}
 According to the assumptions for $k \geq k^*$ we have
 \[
  J_\alpha (\tn_{k+1}) \leq J_\alpha (\tn_{k+1}) + \left( J_\alpha^{\tn_k}(\tn_{k+1}) - J_\alpha (\tn_{k+1}) \right) 
  = J_\alpha^{\tn_k} (\tn_{k+1}) \leq J_\alpha^{\tn_k} (\tn_k) = J_\alpha (\tn_k),
 \]
because $\tn_{k+1}$ is a minimizer of $J_\alpha^{\tn_k}$.\\[1ex]
% We have for the surrogate functional $\tilde J_\alpha$ the estimation
% \[
%  \tilde J_\alpha (\tn_{k+1}, \tn_{k+2}) \leq J_\alpha(\tn_{k+1}) \leq J_\alpha(\tn_{k+1}) + \left( J_\alpha^{\tn_k}(\tn_{k+1}) - J_\alpha (\tn_{k+1}) \right) = \tilde J_\alpha(\tn_{k},\tn_{k+1})
% \]
%because of $J_\alpha^{\tn_k} (\tn_{k+1}) - J_\alpha (\tn_{k+1}) \geq 0$.
\end{proof}
\\[2mm]
Summarizing we found a criterion for the iteration sequence $\{\tn_k\}$ to have a weak limit point.\\
\begin{theorem}
Adopt the assumptions of Lemma \ref{sat:4_LokaleKonvergenz}. Then the sequence of iterates $\{\tn_k\}$ generated by Algorithm \ref{alg:4_IterAdapMin}
has a weakly convergent subsequence, if $1<p<\infty$.\\
\end{theorem}
\begin{proof}
Lemma \ref{sat:4_LokaleKonvergenz} shows that the sequence $J_\alpha (\tn_k)$ is monotonically decreasing and hence converging, since it is bounded from below.
Thus $\{J_\alpha (\tn_k)\}$ is bounded. Because
\[   \frac{\alpha}{p}\|\tn_k - 1\|_{L^p (M)}^p \leq J_\alpha (\tn_k) , \]
the sequence $\{\tn_k\}$ is bounded, too. The spaces $L^p (M)$ are reflexive for $1<p<\infty$ which implies that $\{\tn_k\}$ has a weakly convergent subsequence.\\
\end{proof}
\subsection{Implementation of Algorithms \ref{alg:4_IterAdapMin}, \ref{alg:4_AbstiegLinTikh}}

In this subsection we address issues of the implementation of the  Algorithms \ref{alg:4_IterAdapMin} and \ref{alg:4_AbstiegLinTikh}.
For simplicity we set $M:=\BB$, the closed unit disk in $\RR^2$.\\[1ex]
At first we describe the discretization of $M$ and $\tn$. Let $\tn\in L^p (\RR^2)\cap C^\infty (\RR^2)$. For a given discretization step size $h > 0$ 
we define the equally spaced grid
 \be\label{def:4_Diskretisierung}
  E^h := \left\{ (r_i, s_j) \in \RR^2 : r_i = ih, \, s_j = jh, \, \forall i,j \in \ZZ \right\} \subset \RR^2.
 \ee
 and $e_{ij}^h \subset \RR^2$ by
 \[
  e_{ij}^h := \left\{ (r,s) \in \RR^2 : r_i \leq r < r_{i+1}, \, s_j \leq s < s_{j+1}, \, (r_i,s_j), (r_{i+1}, s_{j+1}) \in E^h  \right\}
 \]
for $i,j \in \ZZ$. On $E^h$ we define the grid values $F^h(\tn)$ of $\tn$ by
\[
 F^h (\tn ) := \left\{ \tn_{ij} = \tn(r_i,s_j) : (r_i,s_j) \in E^h \right\}
\]
and bilinear functions $\varphi_{ij}^{\tn}:\RR^2 \to \RR$ by
\bean
 \varphi_{ij}^{\tn}(r,s) &=&\\
&&\hspace*{-2cm}                    \begin{cases}
                      F^h_{ij}(\tn) + \frac{r-r_i}{h}\left( F^h_{i+1,j}(\tn) - F^h_{i,j}(\tn) \right)
                      + \frac{s-s_j}{h} \left( F^h_{i,j+1}(\tn) - F^h_{i,j}(\tn) \right)  & \\
                      \, + \frac{(r-r_i)(s-s_j)}{h^2} \left( F^h_{i+1,j+1}(\tn) - F^h_{i+1,j}(\tn) - F^h_{i,j+1}(\tn) + F^h_{i,j}(\tn) \right) & \text{ for } (r,s) \in e_{ij}^h \\
                      0 & \text{ otherwise}
                     \end{cases}.
\eean
The unique bilinear interpolate $\tn^h:\RR^2 \to \RR$ of $\tn$ is then given by
\[
 \tn^h (r,s) := \sum_{i,j \in \ZZ} \varphi_{ij}^{\tn}(r,s).
\]
The set of all piecewise bilinear interpolates is denoted by
\[
 V^h := \Big\{ \tn^h (r,s) := \sum_{i,j \in \ZZ} \varphi_{ij}^{\tn}(r,s): \tn \in X_p \Big\}.
\]
The following properties of $\tn$ and its interpolate $\tn^h$ are obvious:\\[1ex]
\begin{itemize}
  \item[a)] $\tn(r_i,s_j) = \tn^h(r_i,s_j)$ for all $(r_i,s_j) \in E^h$\\
  \item[b)] $\tn^h$ ist continuous on every element $e_{ij}^h$ and differentiable in the interior of $e_{ij}^h$\\
\end{itemize}
\begin{remark}
For numerical purposes we cover the unit disk $M$ by a subset of the grid $\tilde E^h := E^h \cap [-1,1] \times [-1,1]$ and set $h=1/Q$ for an integer $Q\in\NN$. For the number of grid points we have $|\tE^h | = (2Q+1)^2$ and 
$$| M \cap \tE^h | = \left| \left\{ (r_i,s_j) \in \tE^h : 
\|(r_i,s_j)\| \leq 1 \right\} \right| \approx \frac{\pi}{4} |\tE^h | = \frac{\pi}{4}(2Q + 1)^2$$ 
such that number of degrees of freedom increases like $\mathcal{O}(Q^2)$.
%   \item Es sind durchaus auch andere Diskretisierungen und Ansatzfunktionen denkbar. Beispielsweise könnte man das Gebiet durch Dreieckselemente diskretisieren und darauf Hütchenfunktionen 
% definieren. Jedoch hat die obige Diskretisierung einen entscheidenden Vorteil. Zu einem gegebenen Punkt $(r,s) \in M$ ist es sehr einfach zu entscheiden in welchem Element $e^h_{ij}$ dieser liegt. 
% Ein Nachteil ist, dass komplizierte Geometrien nicht so gut approximiert werden können.
Note that in view of (\ref{gl:2_GeodDGL}) it is important that $\tn^h$ is differentiable on $M$. This is achieved by extending the gradient of $\tn^h$ continuously to the edges of $e_{ij}^h$.\\
\end{remark}
The next step is the discretization of the measure data $u^{meas}$. In practical applications we of course have only a finite number of TOF measurements. That means the ray transform
$R(\tn)(x,\xi)$ is given for a finite number of pairs $(x,\xi)\in P\subset \partial_+ \Omega M$. This is why we define finite sets
\[
 X^{N_x} := \left\{ x_i = \begin{pmatrix}\sin(2\pi\frac{i-1}{N_x}) \\
			  \cos(2\pi\frac{i-1}{N_x})\end{pmatrix} : i = 1, \ldots ,N_x \right\}
\]
of source points $x_i\in \partial M$ and ray directions
\[
 \Xi_i^{N_\xi} := \left\{ \xi_i^j = \begin{pmatrix}\sin\left(2\pi\left(\frac{i-1}{N_x} + \frac{j-1}{N_\xi} - \frac{1}{4}\right)\right) \\ 
			      \cos\left(2\pi\left(\frac{i-1}{N_x} + \frac{j-1}{N_\xi} - \frac{1}{4}\right)\right)\end{pmatrix}:
j=1,\ldots ,N_\xi \right\}
\]
for $i = 1, \ldots, N_x$. Hence any geodesic curve $\gamma_{x_i,\xi_i^j}$ for which we acquire measure data is characterized by an element $(x_i,\xi_i^j)$ of the set
\[
 P := P^N := \left\{ (x_i,\xi_i^j): x_i \in X^{N_x}, \xi_i^j \in \Xi_i^{N_\xi},\: i=1, \ldots, N_x,\: j=1,\ldots,N_\xi \right\}
\]
where $N:=(N_x,N_\xi)$. For the implementation of Algorithm \ref{alg:4_IterAdapMin} we use then the discrete Tikhonov functional $J_\alpha^{a,h} : V^h \to \RR$
\[
J_\alpha^{a,h} (\tn^h) := \frac{1}{2} \sum_{(x_i,\xi_i^j) \in P^N} \left| R_a (\tn^h)(x_i,\xi_i^j) - u^{meas}(x_i,\xi_i^j) \right|^2 + 
\frac{\alpha}{p} \sum_{(r_i,s_j) \in \tE^h}| \tn^h(r_i,s_j) - 1 |^p
\]
The according set of geodesic curves is given by
\[
\Gamma_N^h := \big\{ \gamma_{x_i,\xi_i^j} : \gamma_{x_i,\xi_i^j} \text{ solves \eqref{gl:3_AWP}, }  x_i \in X^{N_x}, \xi_i^j \in \Xi_i^{N_\xi}, i=1, \ldots, N_x, j=1,\ldots,N_\xi \big\}
\]
and the initial value problem \eqref{gl:3_AWP} is solved by a Runge-Kutta method with stepsize control.\\
The next important ingredient is the implementation of the backprojection operator $R_a^*$ for given $a\in X_p$. We recall that
\[
R_a^* \omega (y) = \int_{S^1} \omega\big( \Pi (\theta, \Phi(y,\theta) )\big)\, 
                    \Big|\mathrm{det} \frac{\partial}{\partial (\theta, s)} \Pi \big(\theta,\Phi ( y,\theta ) \big) \Big|\,\mathrm{d} \theta\,,\qquad y\in M
\]
Our implementation is similar to that of the backprojection operator $R^*$ in standard 2D computerized tomography, but there are two crucial issues to address:\\
\begin{itemize}
  \item[1.] The determinant
  \[  \Big|\mathrm{det} \frac{\partial}{\partial (\theta, s)} \Pi \big(\theta,\Phi ( y,\theta ) \big) \Big|  \]
  can not be computed, since the geodesic projection $\Pi$ as well as the phase function $\Phi$ are not explicitly known. Two possible substitutions are to set the determinant equal to $1$
  or to use the expression for the Euclidean geometry
  \[  \Big|\mathrm{det} \frac{\partial}{\partial (\theta, s)} \Pi \big(\theta,\Phi ( y,\theta ) \big) \Big| = \Big| \frac{1}{\sqrt{1- \langle y,\theta^\perp \rangle^2}} \Big|   \]
  what we have done in our computations. This perfectly fits to our a priori assumption that $\tn$ varies only slightly from $1$.\\
  \item[2.] In standard CT for a given point $y\in M$ it is easy to compute the corresponding boundary point $x\in\partial M$ such that a line, outgoing from $x$ in direction of $\theta$
  meets $y$. For $\tn\not=1$ the determination of $(x,\xi)\in \partial_+ \Omega M$ such that the geodesic $\gamma_{x,\xi}$ meets a given $y\in M$ is a difficult task.\\
\end{itemize}
We developed the following algorithm to compute the backprojection operator $R_{\tn_k}^*$ for given $k$-th iterate $\tn_k$ in Algoritm \ref{alg:4_IterAdapMin} as it was used in our implementation.
The unit vectors $\theta\in S^1$ are discretized by
\[   \theta_m = (\cos\varphi_m, \sin\varphi_m)^\top,\quad \varphi_m=\frac{m-1}{N_\theta}2\pi,\quad m=1,\ldots,N_\theta   \]
for given $N_\theta\in\NN$.\\
\begin{algorithm}[Computation of the adjoint operator $R_{\tn_k}^* \omega$] \label{alg:4_AdjOperator} \newline
 \emph{Input:} Set of geodesic curves $\Gamma_{N}^h$ associated with $\tn_k$ and the corresponding TOF measuremensts $\omega (x_i,\xi_i^j)\in \RR^{N_x+N_\xi}$
               with $\gamma_{x_i,\xi_i^j}\in \Gamma_N^h$, $i=1,\ldots,N_x$, $j=1,\ldots,N_\xi$, reconstruction point $y\in M$.
 \begin{itemize}
  \item[(S0)] Set $0 = w \in \RR^{N_\theta}$.
  \item[(S1)] Iterate for $m = 1,2, \ldots, N_\theta$
  \begin{itemize}
   \item[(a)] Compute $\theta_m$ and $(\theta_m)^\perp$.
   \item[(b)] Project the point $y$ in direction $\theta_m$ on the boundary point $x_e \in \partial M$ and determine $x_i \in X^{N_x}$, which is the nearest neighbor to $x_e$.
   \item[(c)] Determine the unique $j^*\in \{1,\ldots,N_\xi\}$ such that
              \[  \mathrm{dist} \big( y,\mathrm{Tr} (\gamma_{x_i,\xi^{j^*}_i}) \big) \]
              is minimal.
   \item[(d)] If $y \in \gamma_{x_i,\xi^{j^*}_i}$, set $w_m = \omega(x_i,\xi^{j^*}_i)$ and go to step (S1).
   \item[(e)] If $y \not\in \gamma_{x_i,\xi^{j^*}_i}$, then determine $\tau_i^{j^*}\in \RR$ such that $\gamma_{x_i,\xi_i^{j^*}}(t)=y+\tau_i^{j^*} \theta_m^\perp$ for (unique) $t\in \RR$.
   \item[(f)] If $\tau_i^{j^*} < 0$, then
   \begin{itemize}
    \item[i)] check if $\tau_{i+1}^{j^*}<0$. If yes then set $i = i+1$ and go to step (S1f).
    \item[ii)] Otherwise interpolate $w_m$ linearly between $\omega(x_{i},\xi^{j^*}_i)$ and $\omega(x_{i+1},\xi^{j^*}_{i+1})$ and go to step (S1).
   \end{itemize}
   \item[(g)] If $\tau_i^{j^*}>0$, then
   \begin{itemize}
    \item[i)] check if $\tau_{i-1}^{j^*}>0$. If yes then set $i = i-1$ and go to step (S1g).
    \item[ii)] Otherwise interpolate $w_m$ linearly between $\omega(x_{i},\xi^{j^*}_i)$ and $\omega(x_{i-1},\xi^{j^*}_{i-1})$ and go to step (S1).
   \end{itemize}
  \end{itemize}
  \item[(S2)] Approximate the backprojection by
          $$R^*_{\tn_k} (y) \approx \frac{2 \pi}{N_\theta} \sum_{m=1}^{N_\theta} w_m\,\frac{1}{\sqrt{1- \langle y,\theta_m^\perp \rangle ^2}}.$$
 \end{itemize}
 \emph{Output:} Backprojection $R^*_{\tn_k}(y)$.\\[1ex]
\end{algorithm}
The algorithm works as follows: At first we project $y$ to the boundary in direction $\theta_m$ in the Euclidean sense. This gives $x_e\in\partial M$ (S1b). Next we choose
from evey geodesic $\gamma_{x_i,\xi_i^j}$ starting in $x_i\in\partial M$ that one whose trace has minimal distance from $y$. The corresponding index of the tangent is denoted by $j^*$ (S1c).
The next step consists of computing the intersection point of $\gamma_{x_i,\xi_i^{j^*}}$ with the line $y+\tau \theta_j^\perp$ yielding $\tau_i^{j^*}$. If the geodesic runs to the
'left' of $y$ (in the sense of $\theta_m^\perp$), i.e. $\tau_i^{j^*} < 0$, we increment $i$ and repeat the procedure (S1fi). If the geodesics $\gamma_{x_i,\xi_i^{j^*}}$,
$\gamma_{x_{i+1},\xi_{i+1}^{j^*}}$ are located on different 'sides' of $y$ (in the sense of $\theta_m^\perp$) we use linear interpolation to compute $w_m$ (S1fii). If $\tau_i^{j^*} > 0$
we proceed in the same way (S1g). Finally the backprojection $R_{\tn_k} \omega (y)$ is computed using the trapezoidal sum. Please note that both concepts, the usage of the trapezoidal
sum with respect to the projections $\theta_m$ as well as linear interpolation, are adopted from the standard backprojection step for FBP algorithm in 2D computerized tomography,
see e.g. \cite{NATTERER:86}. Algorithm \ref{alg:4_AdjOperator} is emphasized in Figure \ref{fig:4_AdjOperator}.\\
\begin{remark}
 \begin{itemize}
  \item If a point $y\in M$ is located close to the boundary $\partial M$, then it is possible that there are no neighboring geodesics $\gamma_{x_i\xi_i^j}$. In this case we interpolate to zero.
  \item It is importand to fix what runs 'left' or 'right' of $y$ means. Our choice is with respect to the line $y + \spann (\theta_m^\perp)$ but is not necessarily the best choice.
  \item When incrementing or decrementing the index $i$ we have to compute modulo $N_x$.
  \item The interpolation of $w_m$ can be done in several ways. Here we interpolate linearly with respect to the distance of $y$ to the intersection points $\gamma_{x_{i},\xi^{j^*}} \cap (y + \spann 
(\theta_m^\perp))$ and $\gamma_{x_{i\pm 1},\xi^{j^*}} \cap (y + \spann (\theta_m^\perp))$.
 \end{itemize}
\end{remark}
\begin{figure}[H]
 \centering
 \includegraphics[width=.5\textwidth]{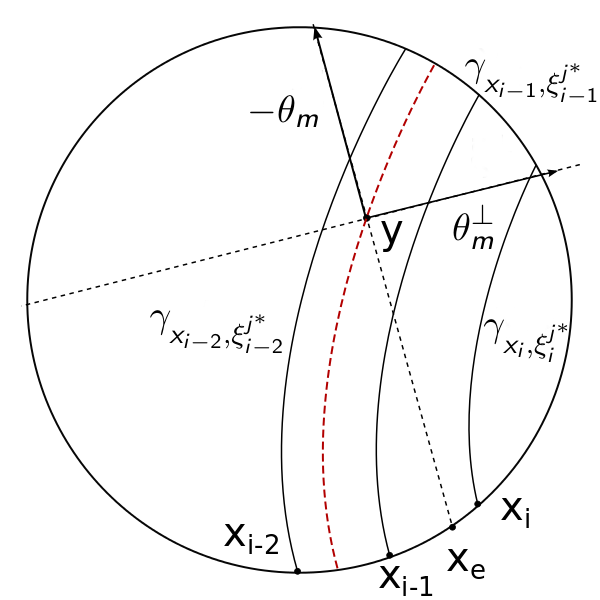}
 \caption{Scheme of Algorithm \ref{alg:4_AdjOperator}. The red dotted line is the geodesic curve, which runs exactly throug $y$ bit is not in the measured data $u^{meas} (x_i,\xi_i^j)$,
$(x_i,\xi_i^j)\in P^N$.}
 \label{fig:4_AdjOperator}
\end{figure}
%

%%%%%%%%%%%%%%%%%%%%%%%%%%%%%%%%%%%%%%%%%%%%%%%%%%%%%%%%%%%%%%%%%%%%%%%%%%%%%%%%%%%%%%%%%%%%%%%%%%%%%%%%%%%

\section{Numerical experiments}

%%%%%%%%%%%%%%%%%%%%%%%%%%%%%%%%%%%%%%%%%%%%%%%%%%%%%%%%%%%%%%%%%%%%%%%%%%%%%%%%%%%%%%%%%%%%%%%%%%%%%%%%%%%

In this section we demonstrate the performance of our method on the basis of several test examples using synthetic data. We implemented Algorithms \ref{alg:4_IterAdapMin}, \ref{alg:4_AbstiegLinTikh}
and \ref{alg:4_AdjOperator} as shown in Section 3. To solve the forward operator, i.e. to compute the ray transform $R_a \tn$ for given $a$ and $\tn$ we applied 
an extrapolation step control method based on the classical 
Runge-Kutta method of 4th order. Thereby the tolerance parameter was chosen as $\rho=0.9$, the initial stepsize as $h_0=0.025$ and the stopping index was $\varepsilon = 10^{-6}$.
%All computation were performed using an Intel Xeon E5-2650 with 2.0 GHz and 384 GB RAM. 
Since the geodesic curves $\gamma_{x_i,\xi_i^j}$ can be computed independently from each other,
the computation of $R_a \tn$ as well as of the synthetic TOF data $u^{meas} = R(\tn)$ can be parallelized what we have done using 30 cores. The objective functional
\[
J_\alpha^{\tn_k,h} (\tn^h) = \frac{1}{2} \sum_{(x_i,\xi_i^j) \in P^N} \left| R_{\tn_k^h} (\tn^h)(x_i,\xi_i^j) - u^{meas}(x_i,\xi_i^j) \right|^2 + 
\frac{\alpha}{p} \sum_{(r_i,s_j) \in \tE^h}| \tn^k(r_i,s_j) - 1 |^p
\]
is minimized subject of $\tn^h\in V^h$ in each iteration step using the steepest descent method given in Algorithm \ref{alg:4_AbstiegLinTikh}.

\subsection{Sound speed with 'peaks'}
\label{c-peaks}

In the first example we investigate the behavior of the reconstruction results with respect to a varying regularization parameter $\alpha > 0$ and for fixed norm $p=2$. 
In every step we chose $\alpha_k = \overline \alpha$ as a constant. The exact refractive index $\tn(x)=1/c(x)$ is given by the sound speed ``peaks'' as
\be\label{sound-peaks}
  c(x) = 1/\tn(x) = 1 + \sum_{i=1}^3\varphi_i(x)\chi_i(x) 
 \ee
with
 \[
  \varphi_i(x) = \vartheta_i e^{-\frac{1}{r_i - \|x - q_i \|}}
 \]
and 
 \[
  \chi_i(x) = \begin{cases}
               1 &\text{if } \|x - q_i\| \leq r_i \\
               0 &\text{otherwise}
              \end{cases}
 \]
for given center points $q_i \in M$, radii $r_i > 0$ and amplitudes $\vartheta_i \in (-1, \infty)$, $i=1,2,3$. In our experiments we set
\bean
 q_1 = \begin{pmatrix}\frac{1}{5} \\ \frac{2}{5} \end{pmatrix}, \, q_2 = \begin{pmatrix}-\frac{1}{3} \\ -\frac{1}{3} \end{pmatrix}, \, q_3 = \begin{pmatrix}\frac{1}{2} \\ -\frac{1}{2} \end{pmatrix}
\eean
for the center points,
\bean
r_1 = \frac{1}{4},\, r_2 = \frac{1}{5},\, r_3 = \frac{1}{6}
\eean
for the radii and
\bean
\vartheta_1 = \frac{1}{5},\, \vartheta_2 = -\frac{3}{20},\, \vartheta_3 = \frac{1}{10}
\eean
for the amplitudes. In each iteration step we chose $N_x = 100$ (number of detectors) and $N_\xi = 100$ (number of signals per detector), such that in every step we have to compute
$\left|\Gamma_N^h \right| = 10000$ geodesics. The computation of a full set of geodesics $\Gamma_N^h$ lasts about 30 seconds. 
The evaluation of the adjoint $R^*_{\tn_k^h}$ lasts about one minute.

\subsubsection{Experiments with respect to the regularization parameter}

We compared results for regularization parameters $\alpha_i = 0.3 + 0.6i$, $i = 1, \ldots , 5$. As iteration step size in Algorithm \ref{alg:4_AbstiegLinTikh}  we used $\mu_l = 0.01$. 
The unit square was discretized using the step size $h=0.1$ yielding $346$ reconstruction points.\\
Figure \ref{fig:4_Ex1ZF} shows the values $J^{\tn_k,N}_\alpha(\tn_{k+1})$ for different regularization parameters $\alpha_i$. One can see that all functional values at first increase and then 
decrease (except at the first one) to a minimum which is less than the initial value. On the basis of these curves one is able to determine \emph{optimal stopping indices} $k^*_i\in\NN$, $i=1,\ldots,5$,
where
\[
 k^*_1 = 216,\, k^*_2 = 246,\, k^*_3 = 240,\, k^*_4 = 147 \text{ and } k^*_5 = 303.
\]
Figure \ref{fig:4_Ex1SGV} shows $c_{k^*_i}=1/\tn_{k^*_i}$ for $i=1,\ldots,5$ along with the exact solution $c(x)=1/\tn(x)$. The errors $c(x)-c_{k^*_i}(x)$ are plotted in
Figure \ref{fig:4_Ex1SGVDiff}. One realizes that for smaller values of $\alpha$ there are higher fluctuations in the reconstructions. For all reconstructions the peaks are well detected, but 
their quantity is smaller if $\alpha$ increases. 
For example in the peak $q_1=(0.2,0.4)$ we have $c(0.2,0.4) = 1.2$  and the reconstruction with $\alpha_1=0.3$ gives $c_1(0.2,0.4) \approx 1.15$, but for $\alpha_4=2.7$ we have 
$c_5(0.2,0.4) \approx 1.08$ (compare Figure \ref{fig:4_Ex1SGS}). A bigger $\alpha$ leads to a higher weighting of the regularization term which penalizes aberrations from $1$.
The weak amplitude at $q_2=(-\frac{1}{3},-\frac{1}{3})$ is only detected if $\alpha$ is big enough. Otherwise the peak can not be recognized because of the high fluctuations.

\begin{figure}[H]
\centering
\includegraphics[width=.45\textwidth]{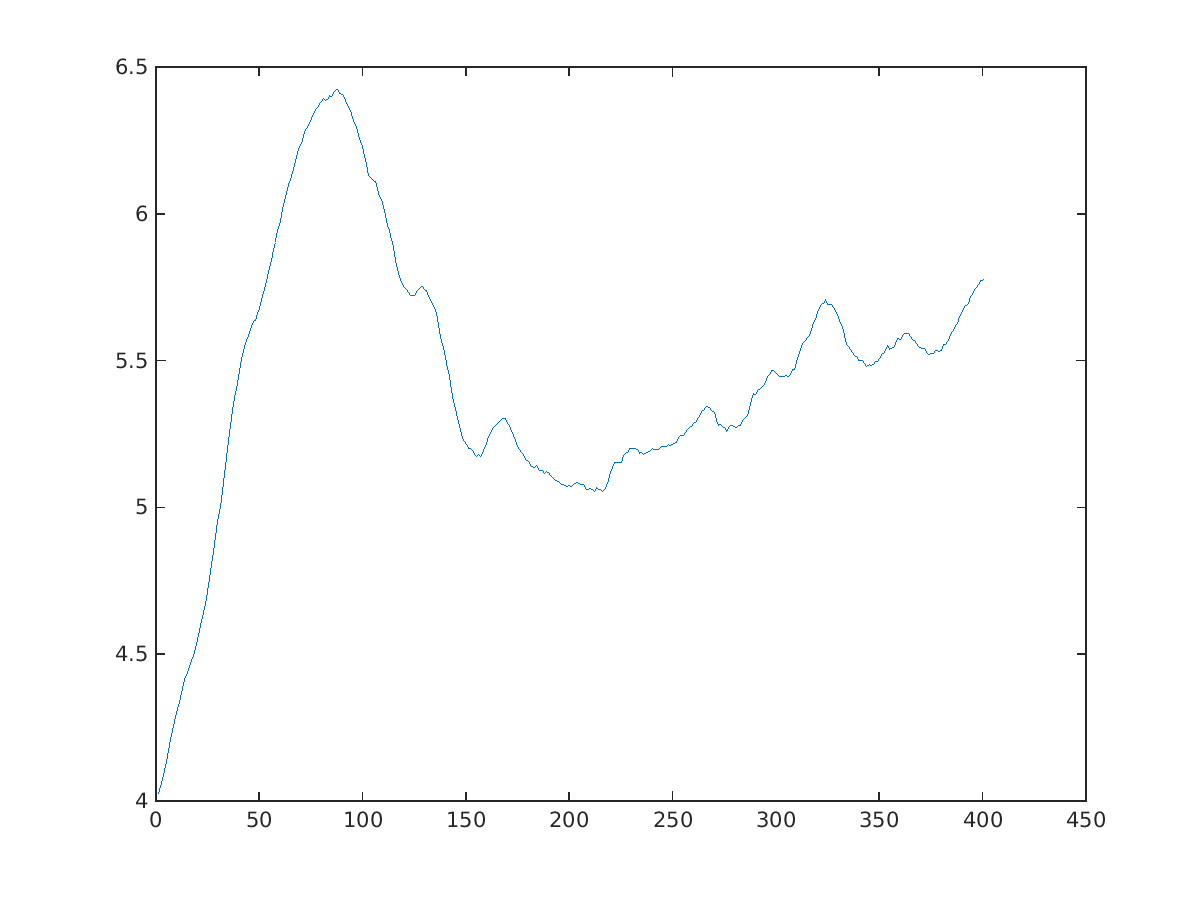}
\includegraphics[width=.45\textwidth]{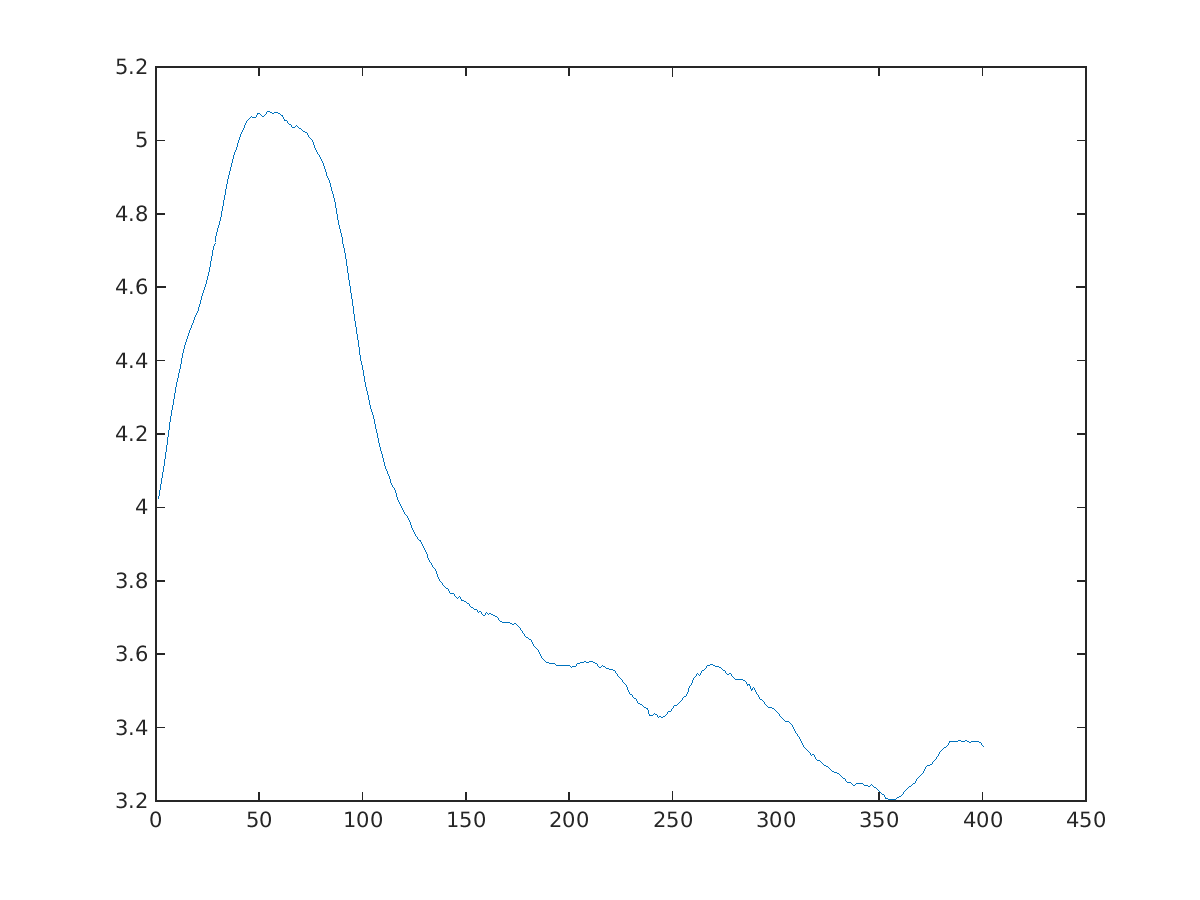} \\
\includegraphics[width=.45\textwidth]{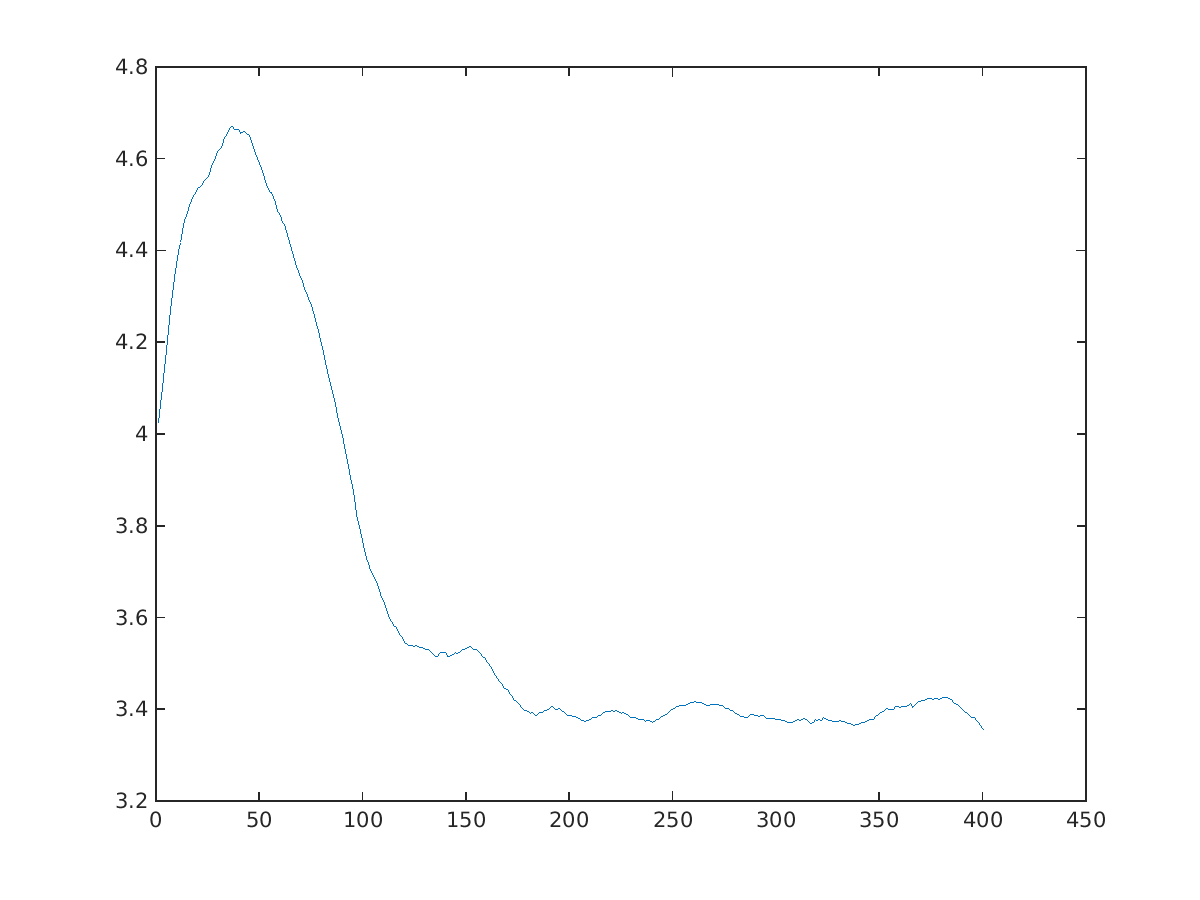} 
\includegraphics[width=.45\textwidth]{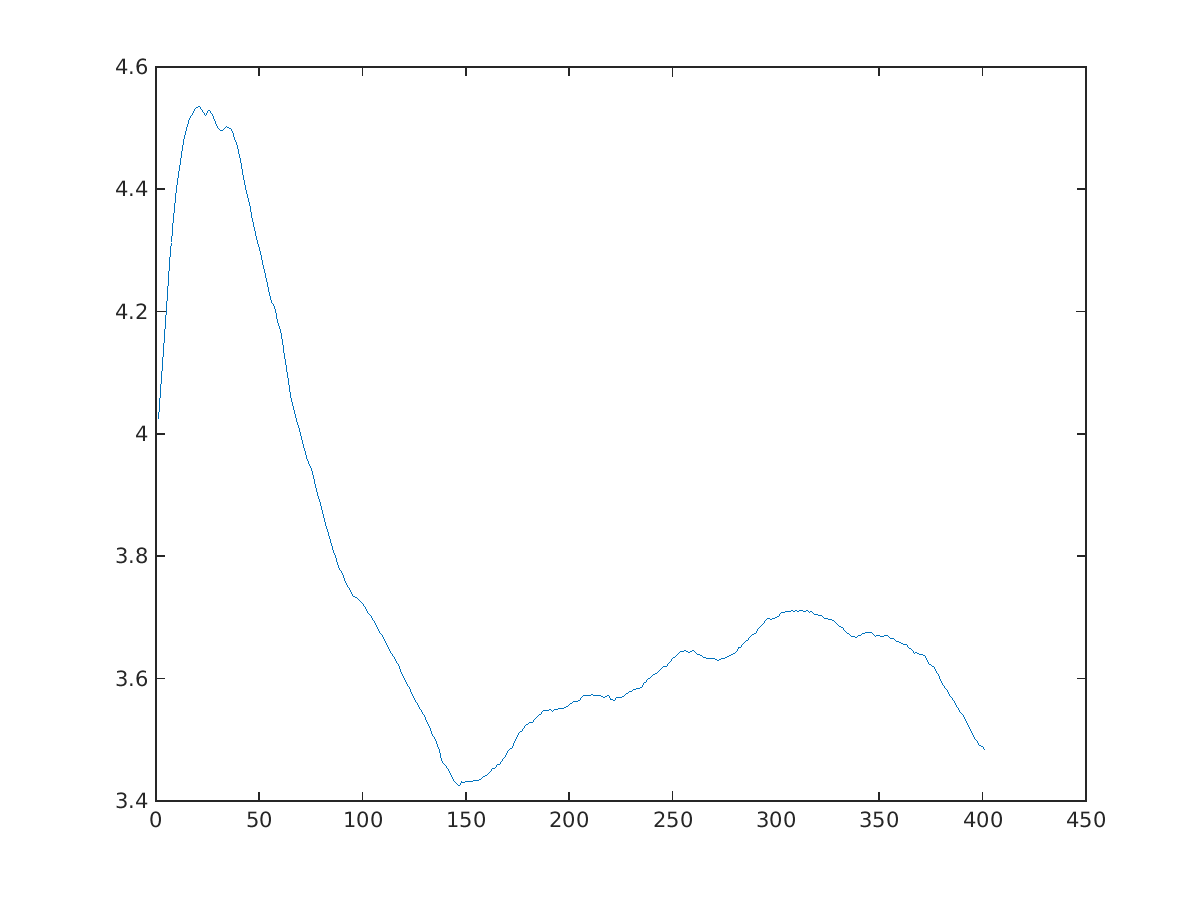} \\
\includegraphics[width=.45\textwidth]{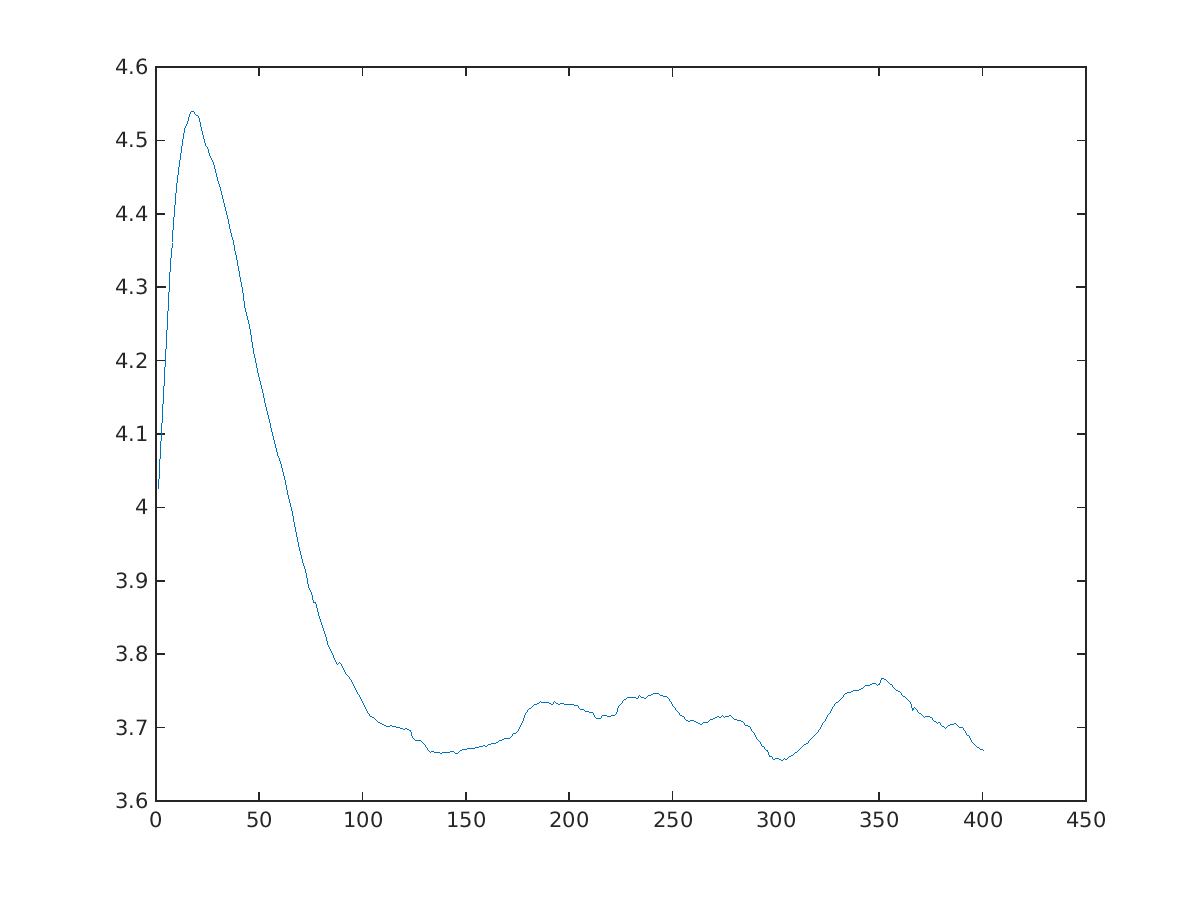}
 \caption{Values $J^{\tn_k,N}_\alpha (\tn_{k+1})$ for different regularization parameters $\alpha_i$ plotted against iteration index $k$. 
Top left: $\alpha_1 = 0.3$, top right: $\alpha_2 = 0.9$, middle left: $\alpha_3 = 1.5$, 
middle right: $\alpha_4 = 2.1$, bottom: $\alpha_5 = 2.7$.}
 \label{fig:4_Ex1ZF}
\end{figure}
\begin{figure}[H]
\centering
\includegraphics[width=.45\textwidth]{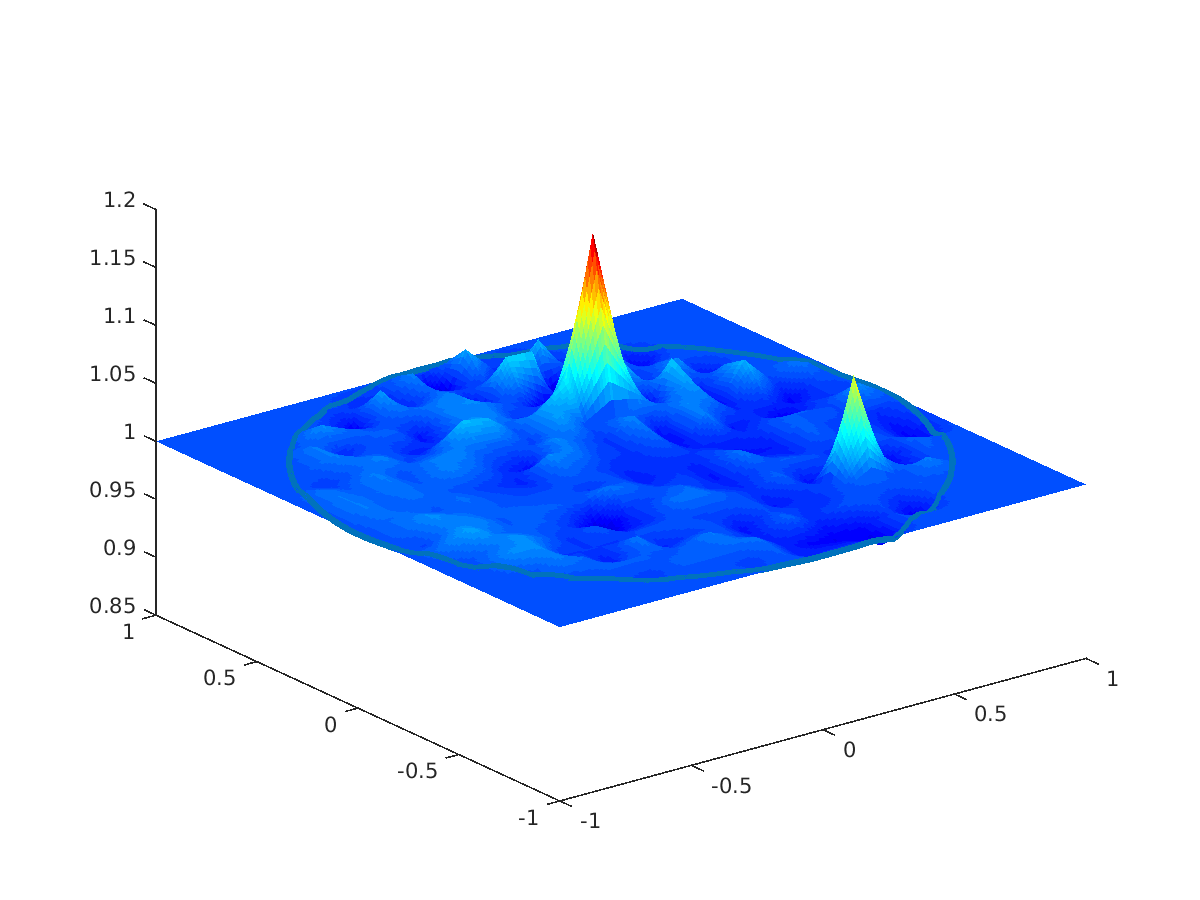}
\includegraphics[width=.45\textwidth]{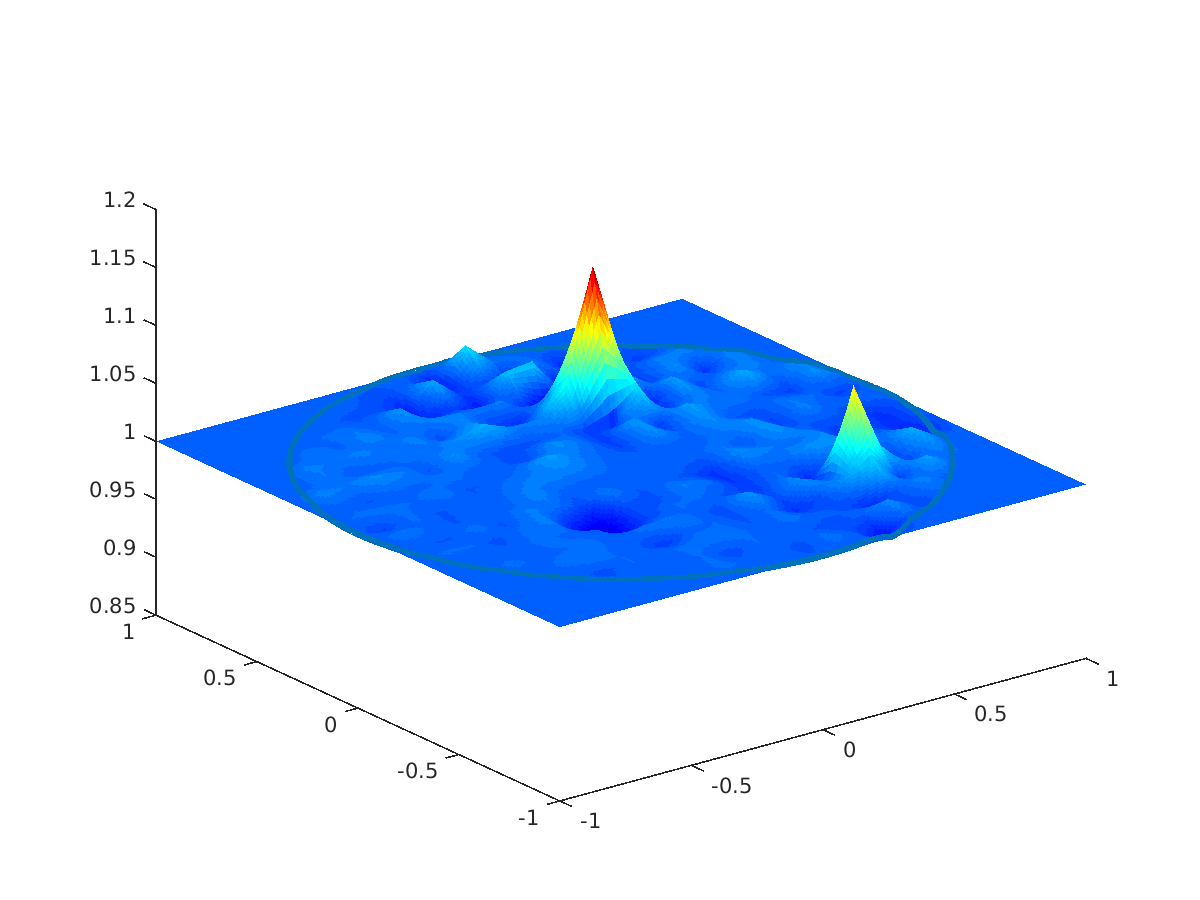} \\
\includegraphics[width=.45\textwidth]{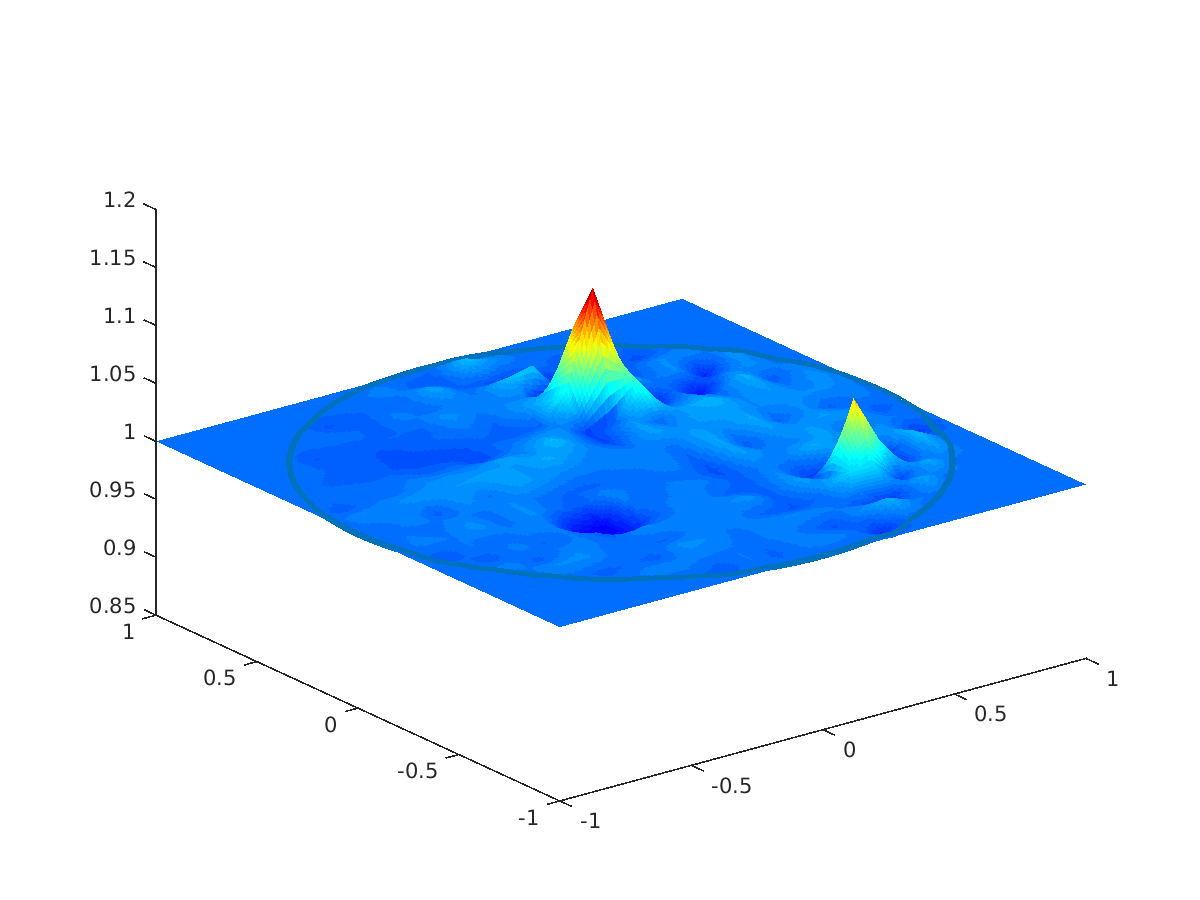} 
\includegraphics[width=.45\textwidth]{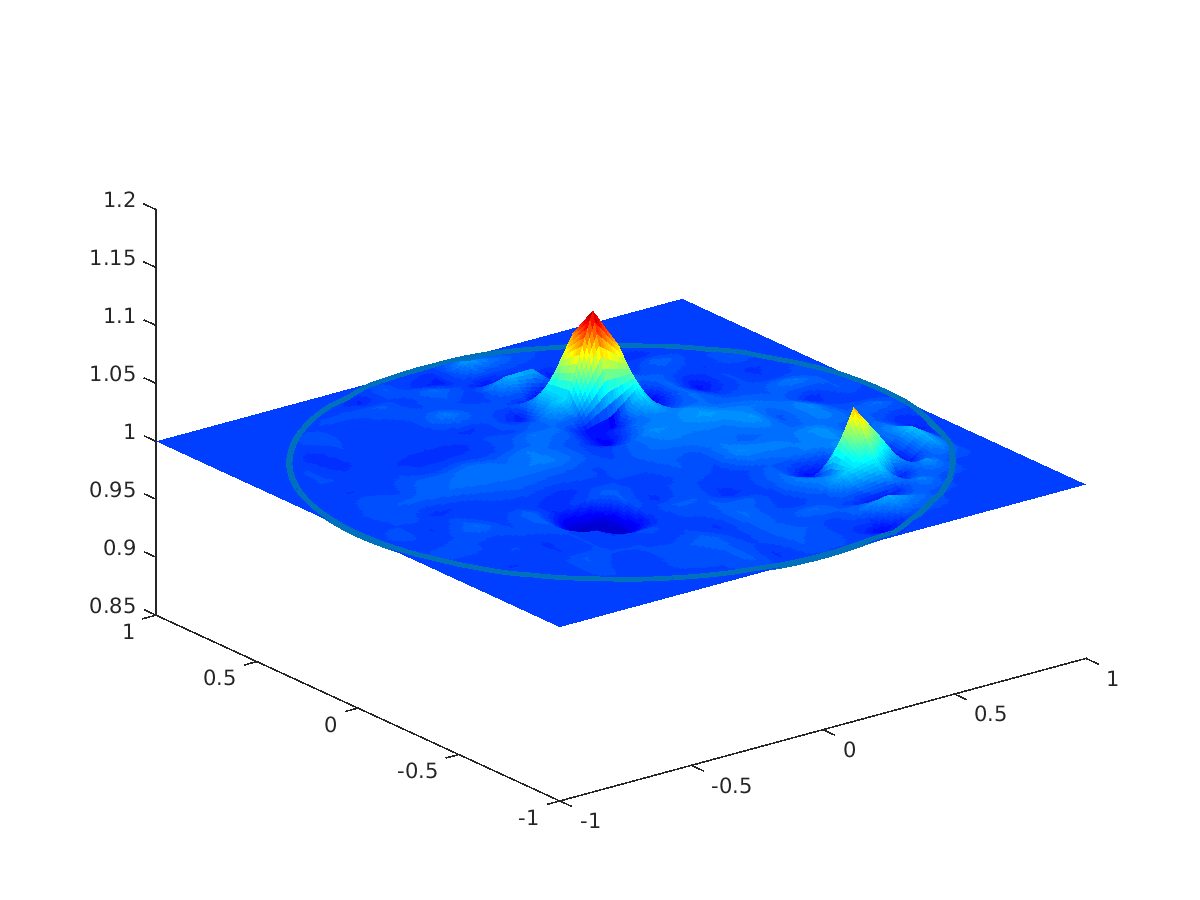} \\
\includegraphics[width=.45\textwidth]{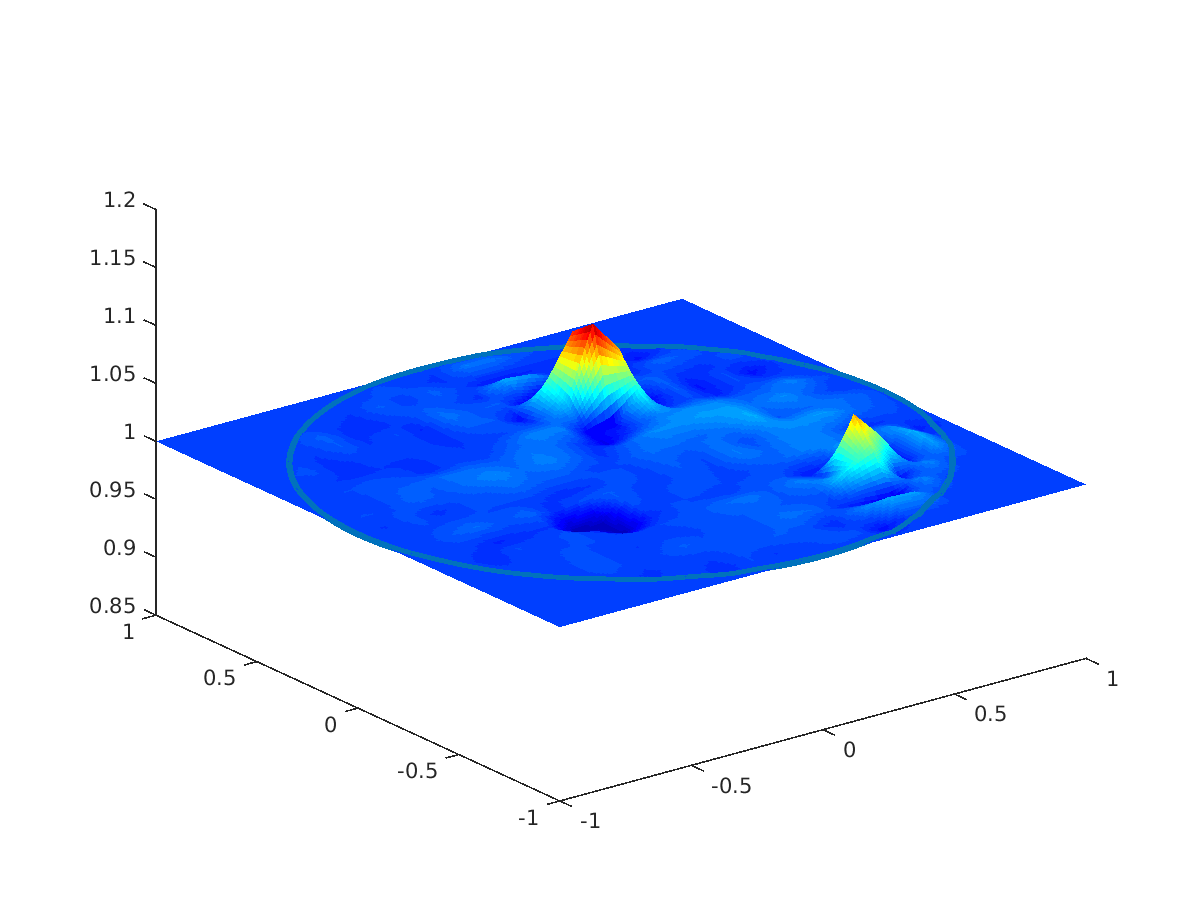}
\includegraphics[width=.45\textwidth]{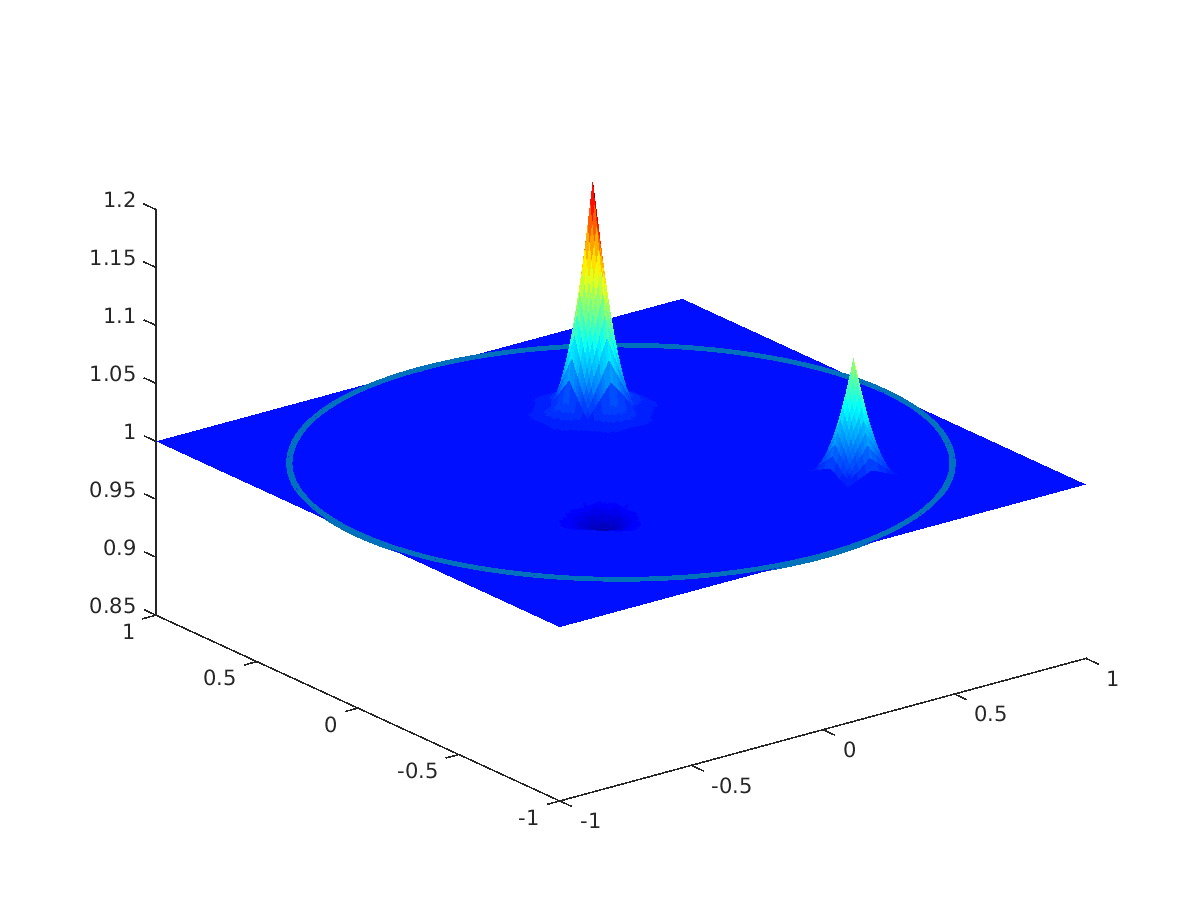}
 \caption{Exact sound speed $c(x)$ and reconstructions $c_{k^*_i}(x)$ for different regularization parameters $\alpha_i$. Top left: $\alpha_1 = 0.3$, top right: $\alpha_2 = 0.9$, middle left: $\alpha_3 = 
1.5$, middle right: $\alpha_4 = 2.1$, bottom left: $\alpha_5 = 2.7$, bottom right: original sound speed.}
 \label{fig:4_Ex1SGV}
\end{figure}
\begin{figure}[H]
\centering
\includegraphics[width=.45\textwidth]{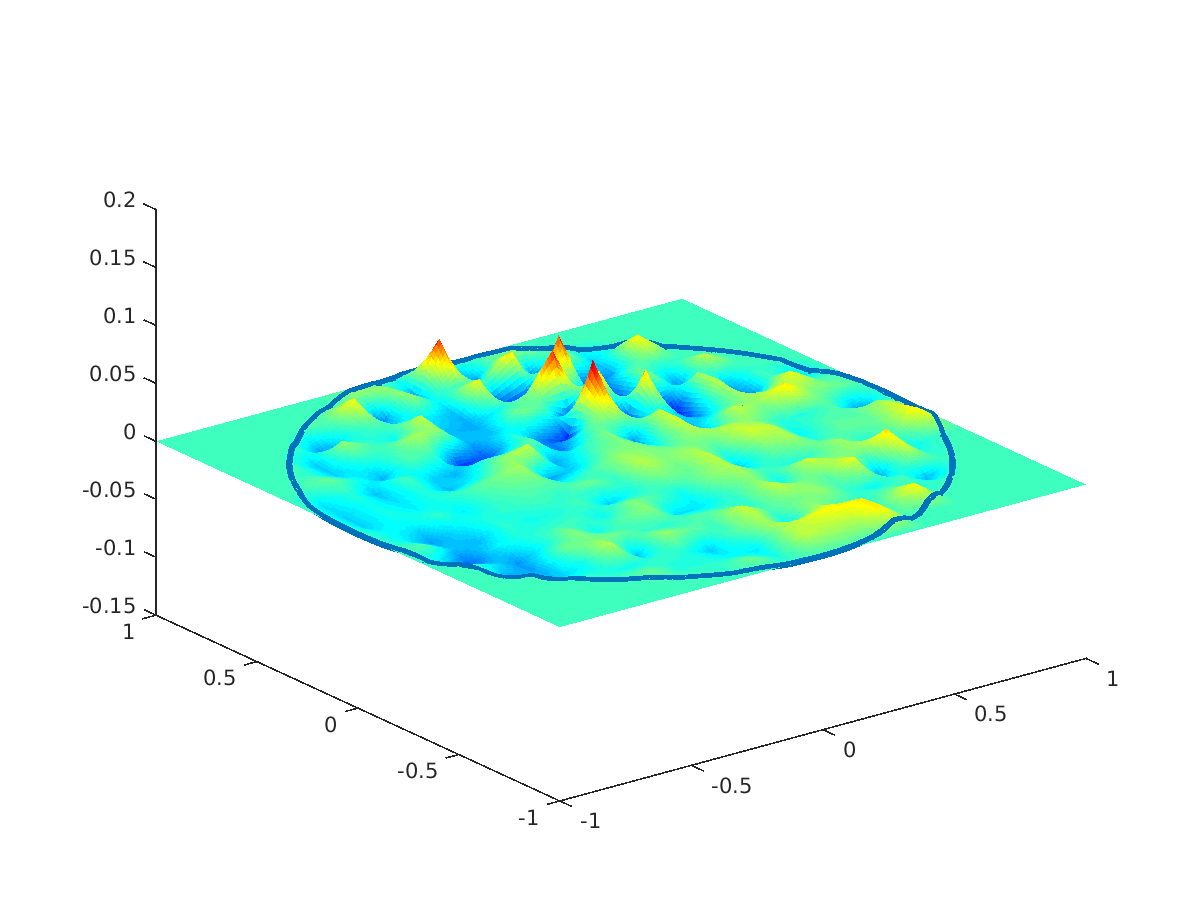}
\includegraphics[width=.45\textwidth]{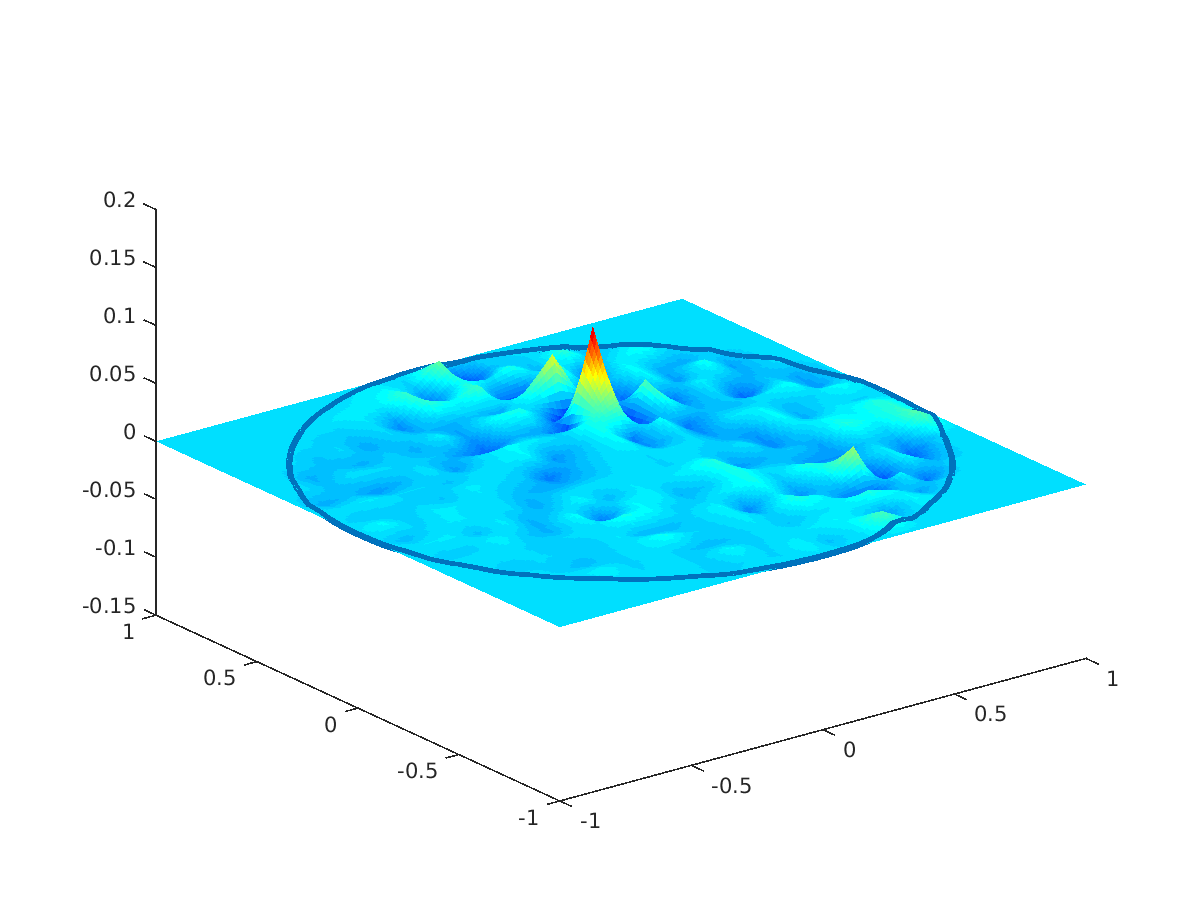} \\
\includegraphics[width=.45\textwidth]{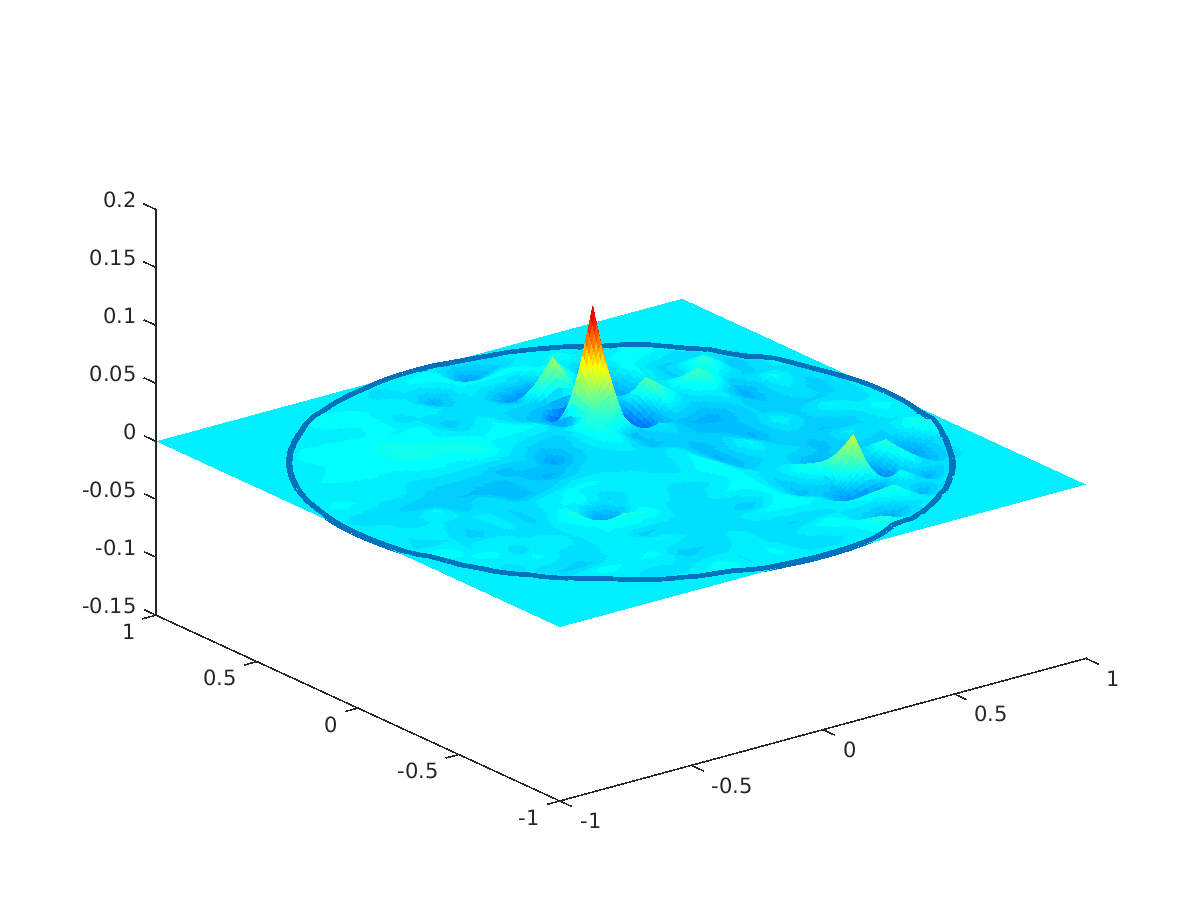} 
\includegraphics[width=.45\textwidth]{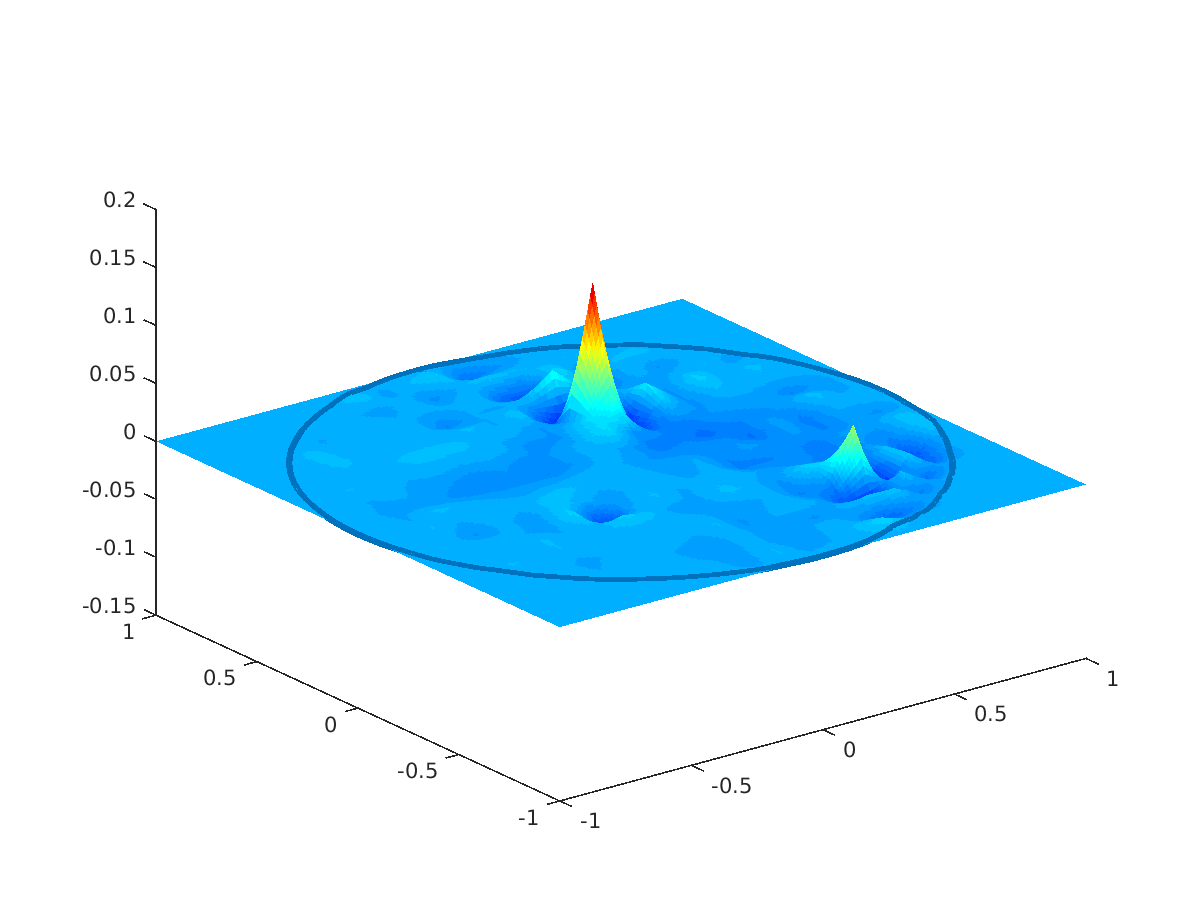} \\
\includegraphics[width=.45\textwidth]{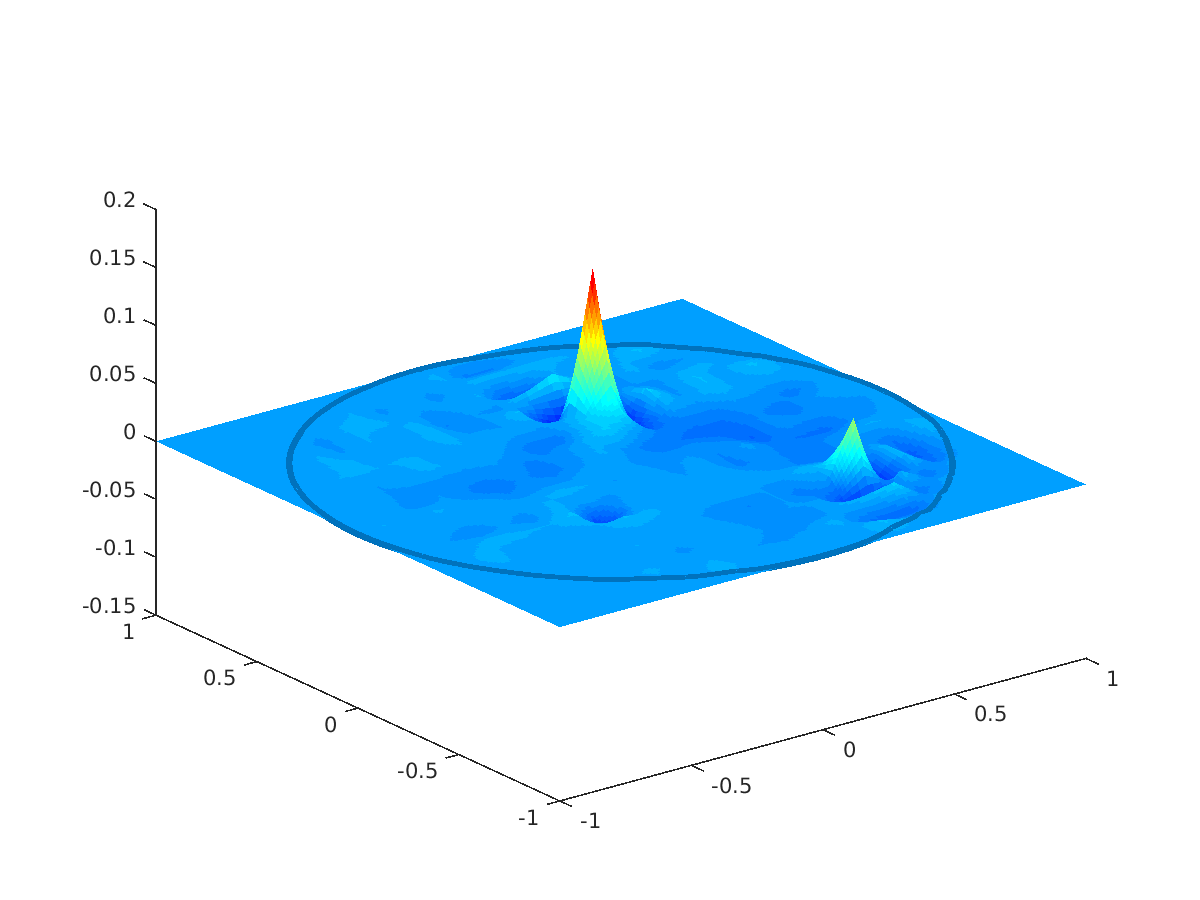}
 \caption{Errors $c(x) - c_{k^*_i}(x)$ for different regularization parameters $\alpha_i$. \emph{Top left:} $\alpha_1 = 0.3$, \emph{top right:} $\alpha_2 = 0.9$, 
\emph{middle left:} $\alpha_3 = 1.5$, \emph{middle right:} $\alpha_4 = 2.1$, \emph{bottom:} $\alpha_5 = 2.7$.}
 \label{fig:4_Ex1SGVDiff}
\end{figure}
\begin{figure}[H]
\centering
\includegraphics[width=.30\textwidth]{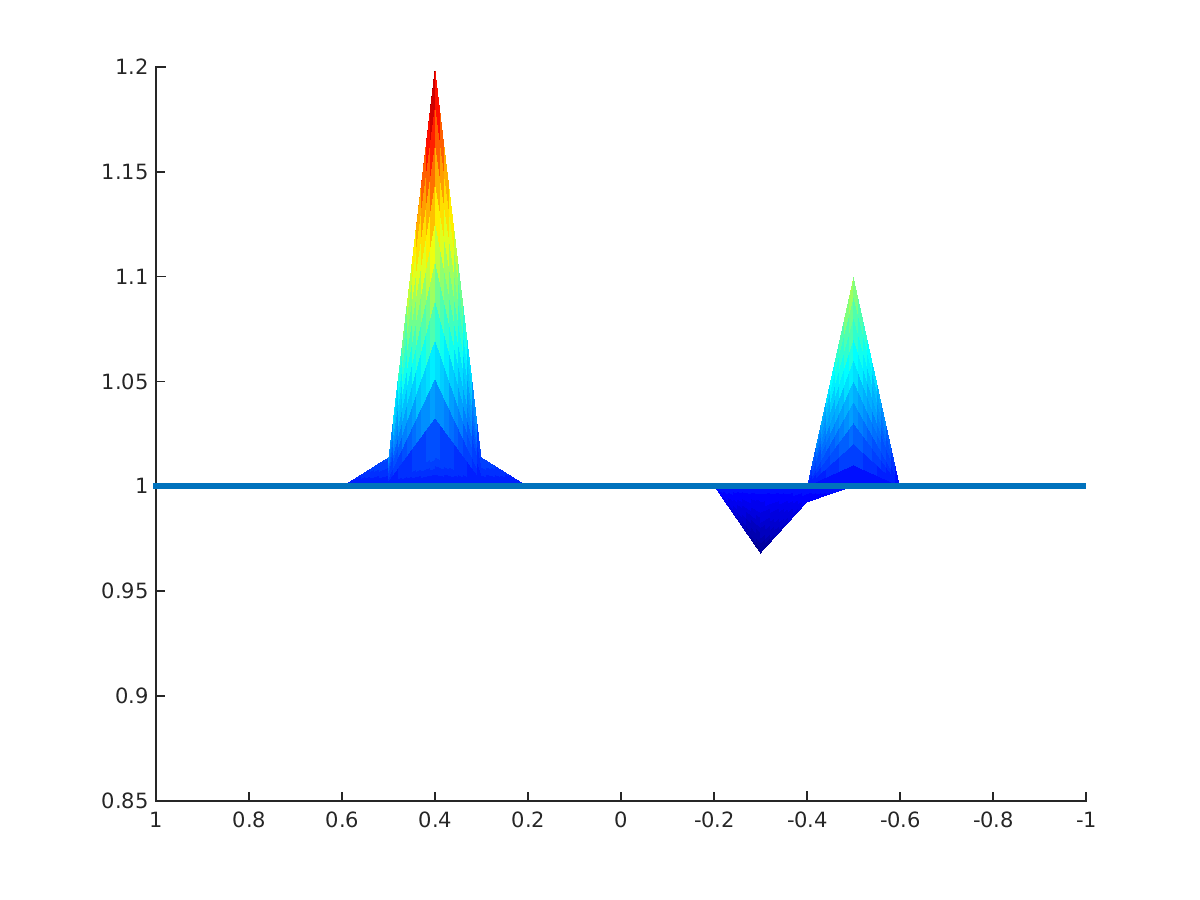}
\includegraphics[width=.30\textwidth]{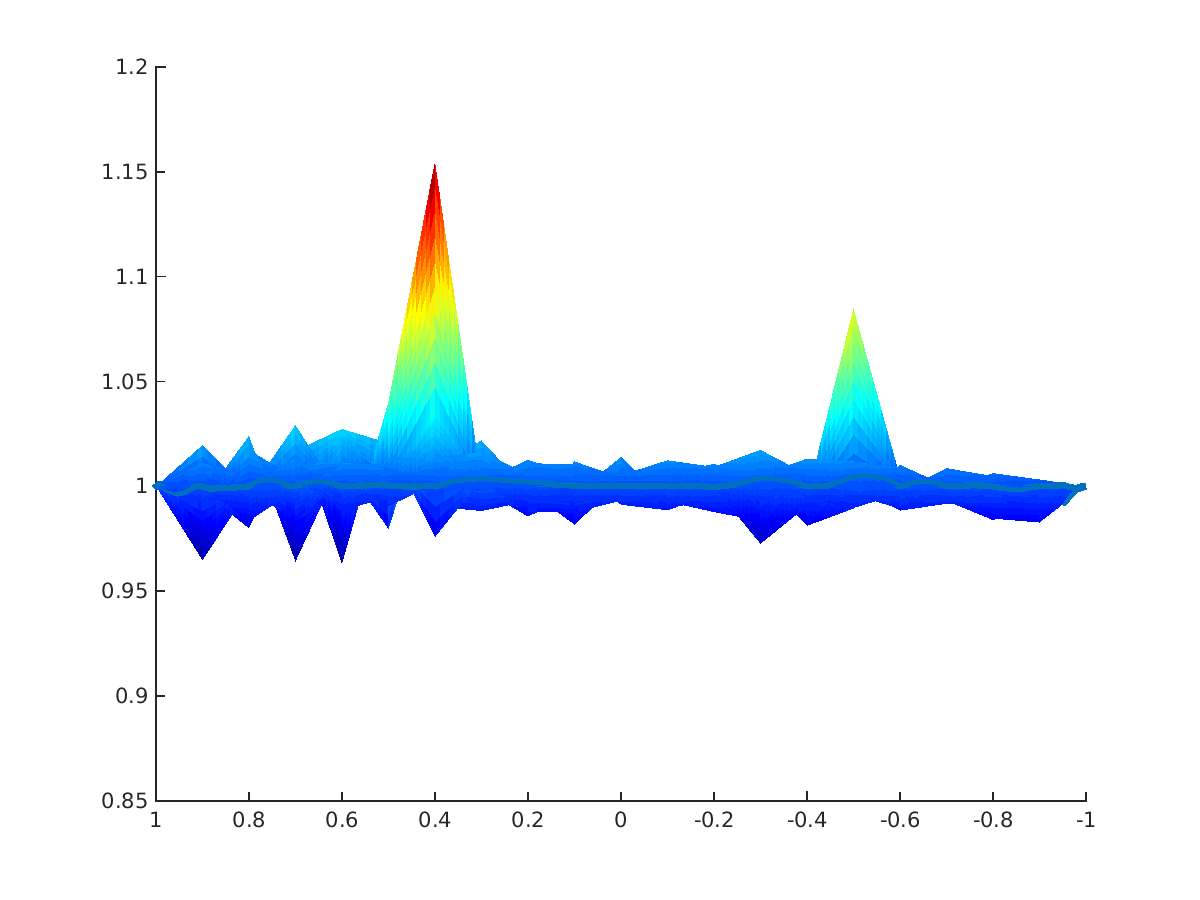}
\includegraphics[width=.30\textwidth]{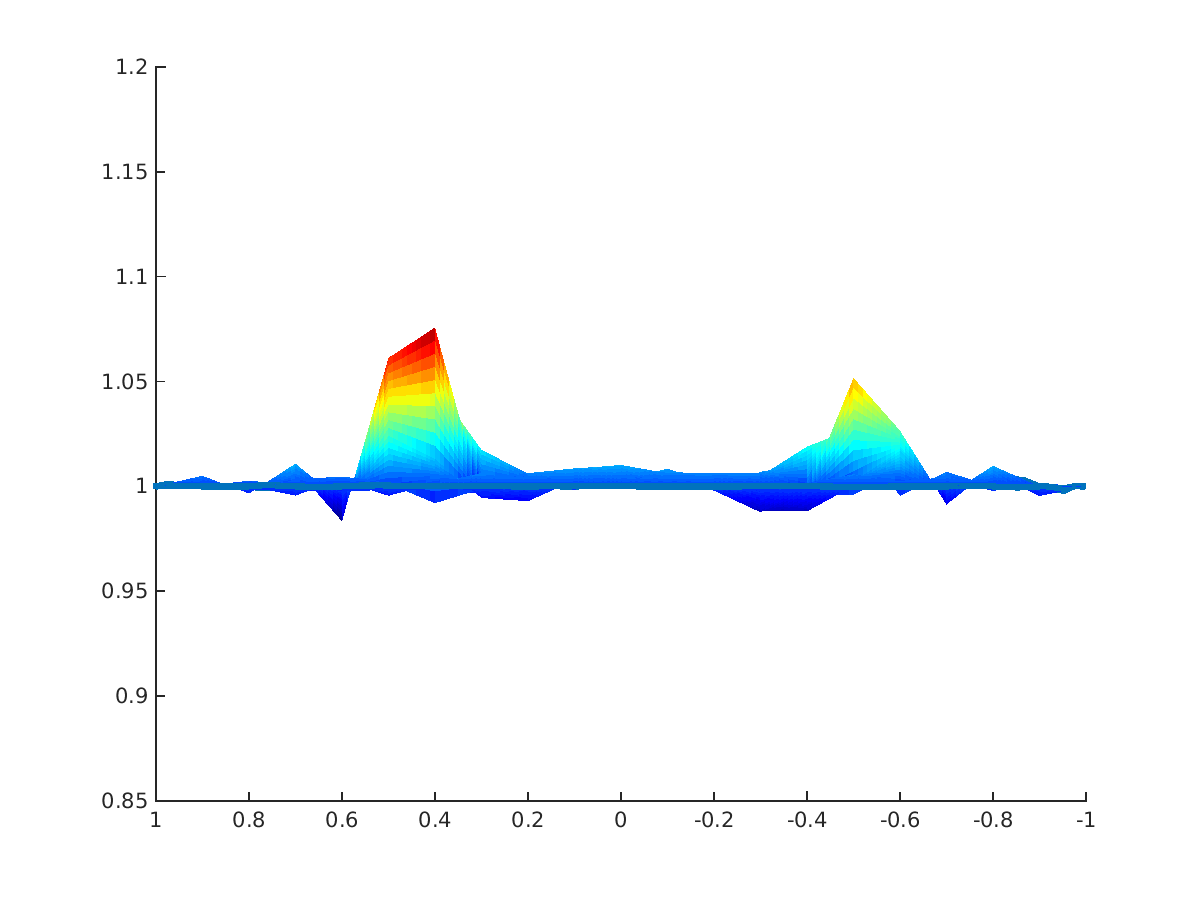}
 \caption{Exact sound speed $c(x)$ and reconstructions $c_{k^*_i}(x)$ for different regularization parameters $\alpha_i$. \emph{Left:} original, \emph{middle:} $\alpha_1 = 0.3$ \emph{right:} $\alpha_5 = 2.7$.}
 \label{fig:4_Ex1SGS}
\end{figure}
\begin{figure}[H]
\centering
\includegraphics[width=.45\textwidth]{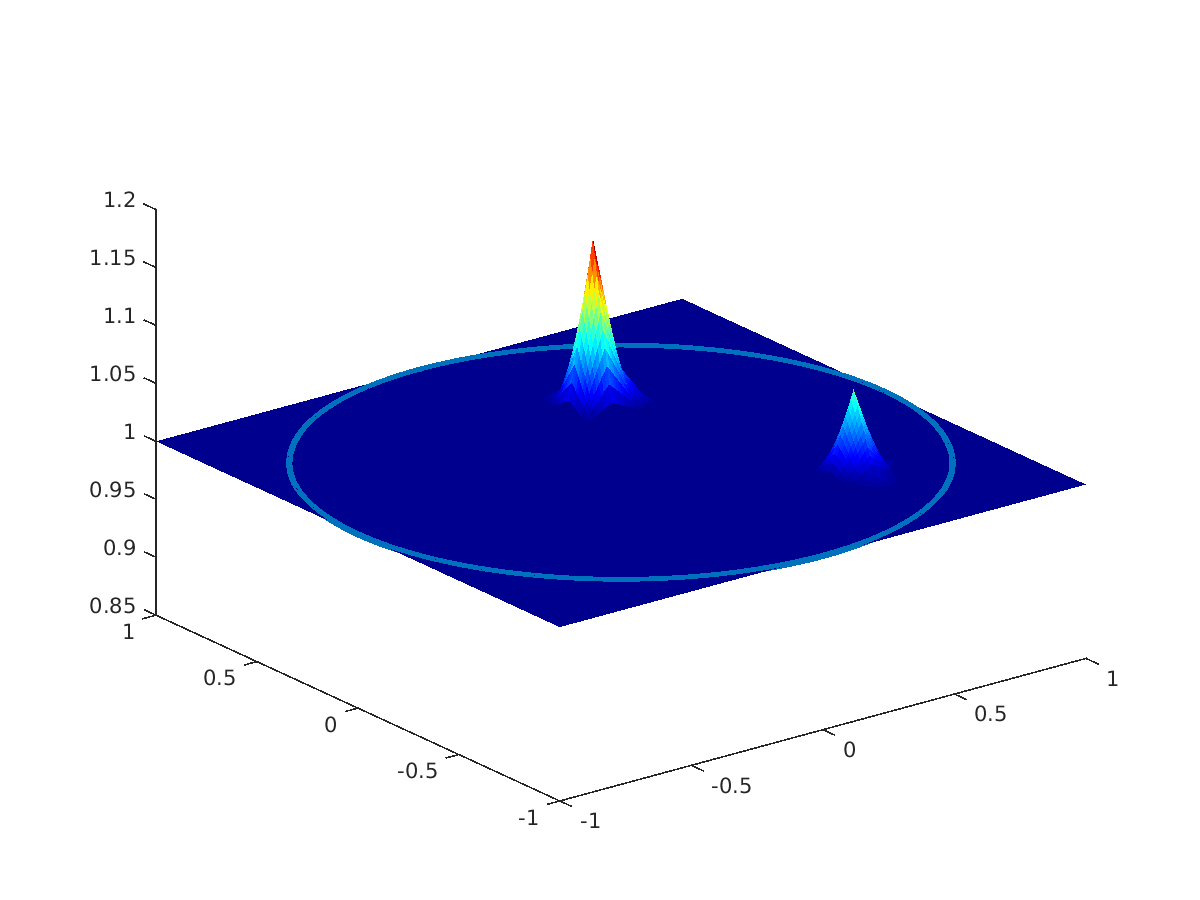}
\includegraphics[width=.45\textwidth]{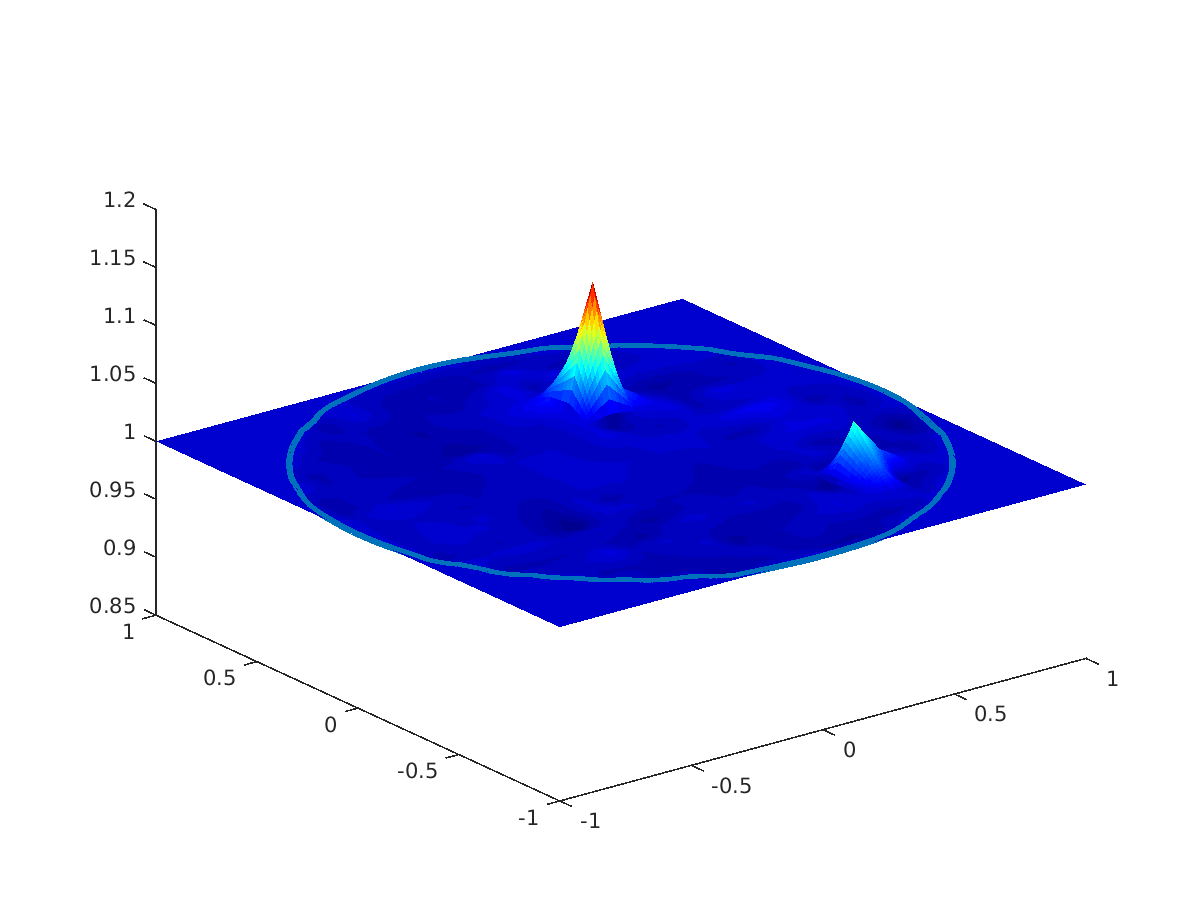} \\
\includegraphics[width=.45\textwidth]{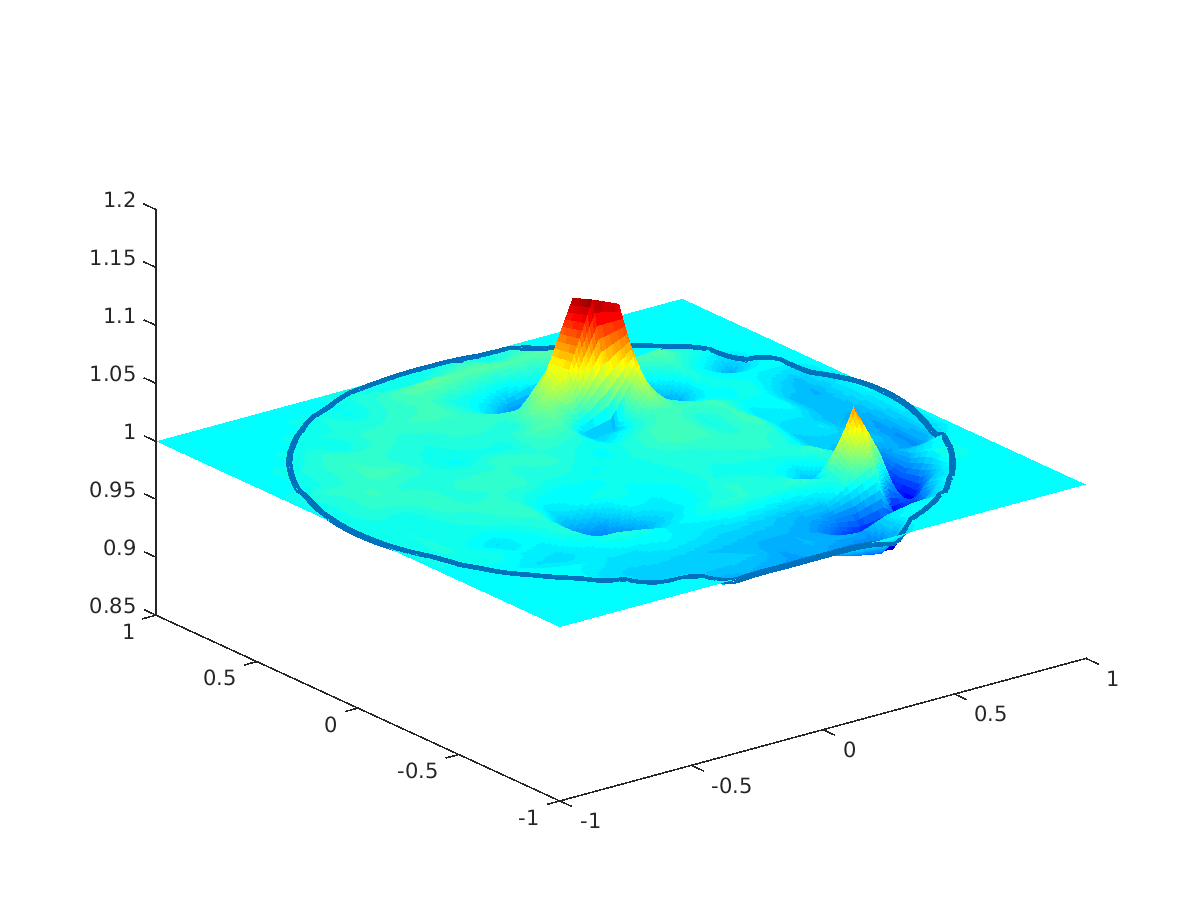} 
\includegraphics[width=.45\textwidth]{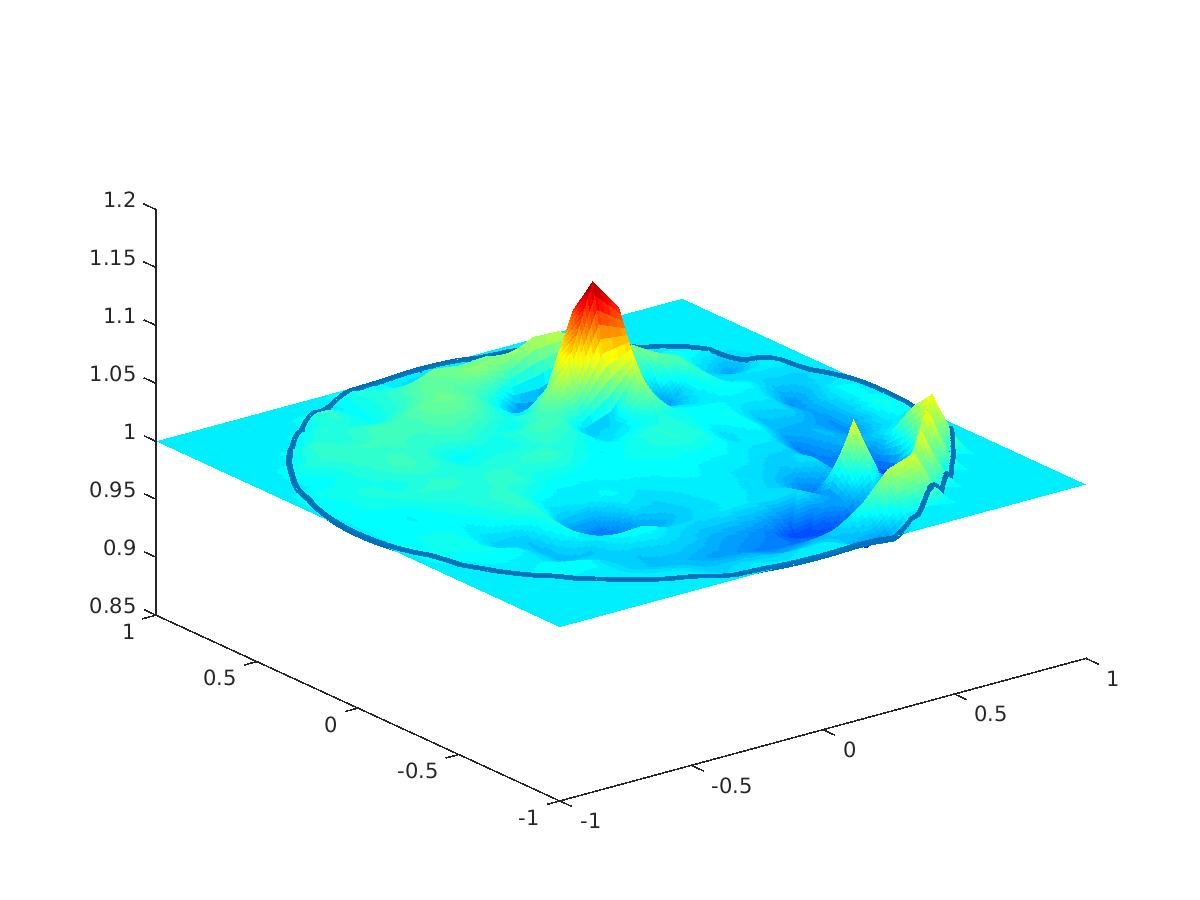} \\
\includegraphics[width=.45\textwidth]{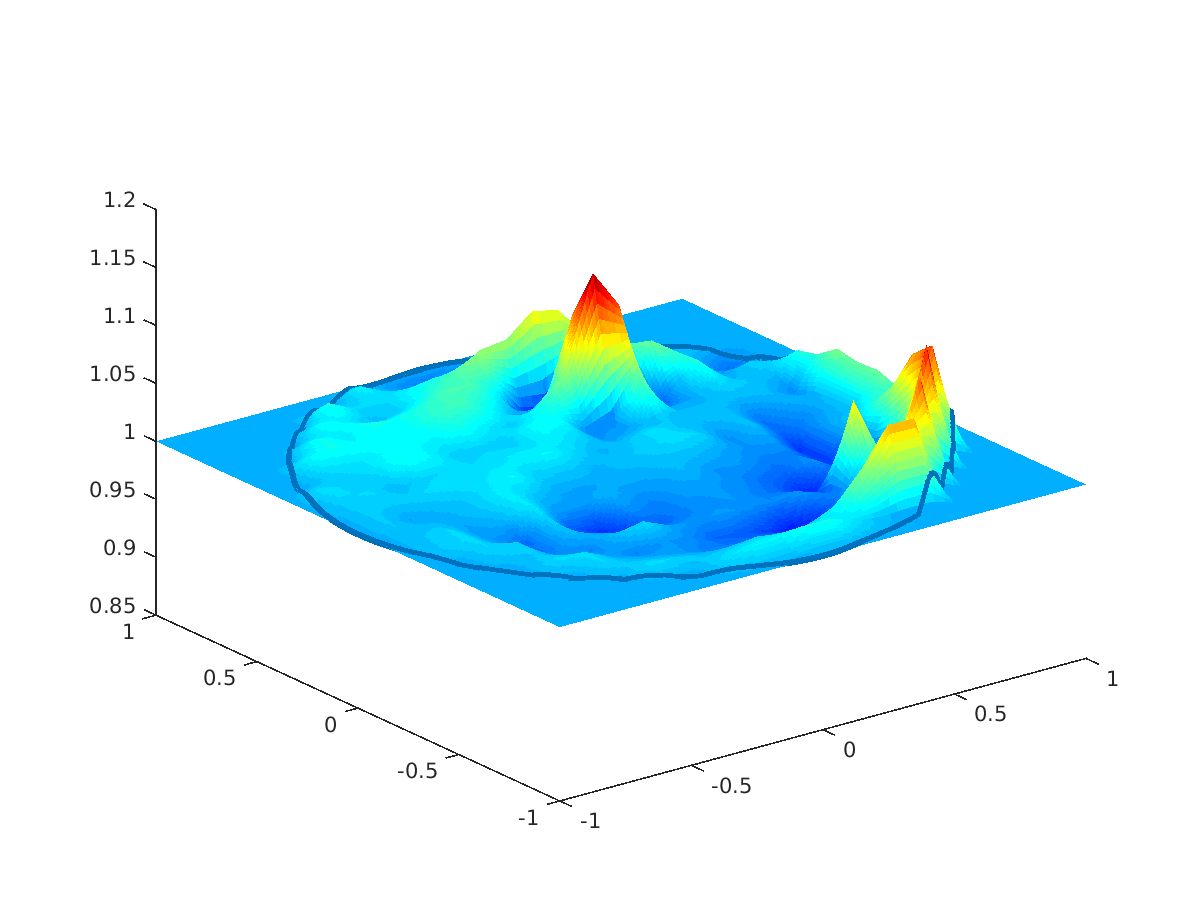}
\includegraphics[width=.45\textwidth]{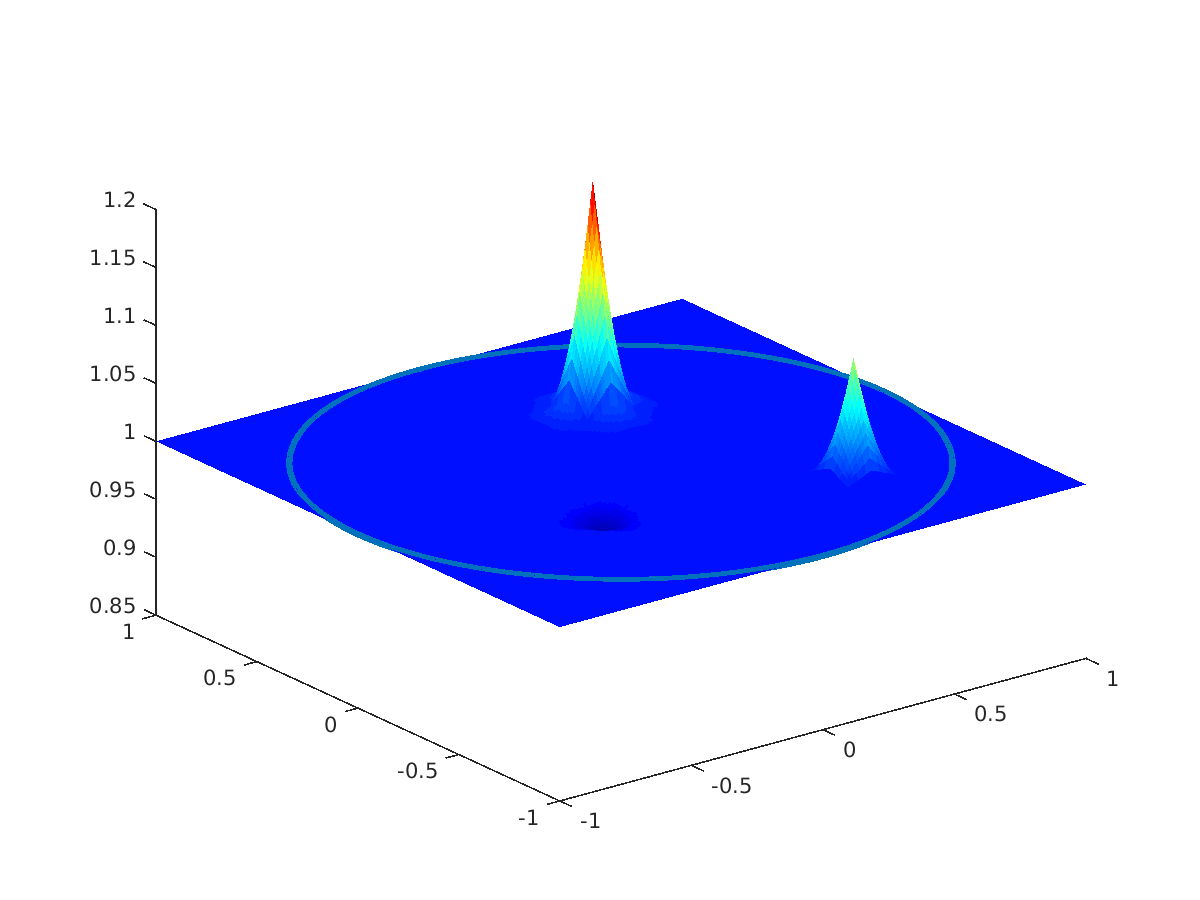}
 \caption{Reconstructions $c_{k^*_i}$ for different $p_i$-norms. \emph{Top left:} $p_1 = 1$, \emph{top right:} $p_2 = 1.1$, \emph{middle left:} $p_3 = 2$, \emph{middle 
right:} $p_4 = 4$, 
\emph{bottom left:} $p_5 = 10$, \emph{bottom right:} exact sound speed.}
 \label{fig:4_Ex2PeaksVogel}
\end{figure}
\begin{figure}[H]
\centering
\includegraphics[width=.45\textwidth]{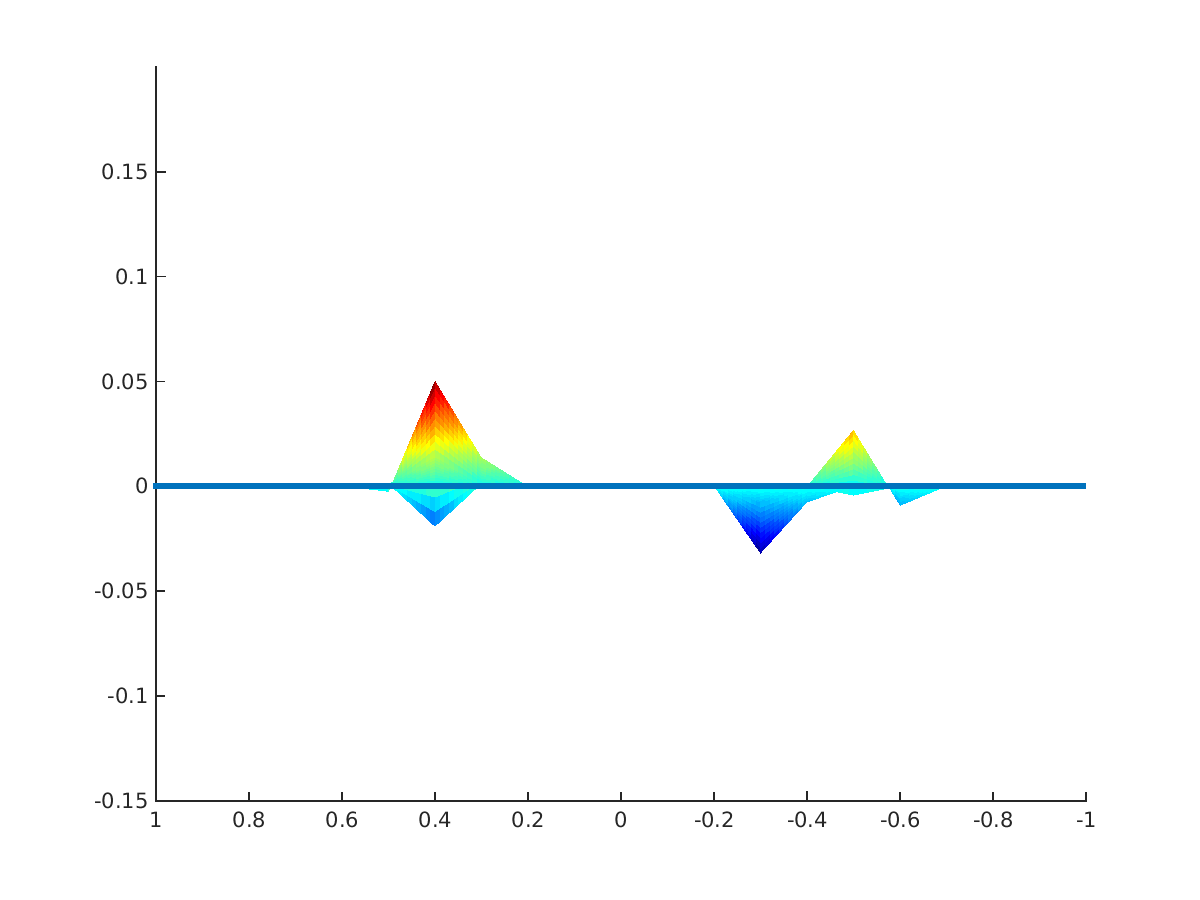}
\includegraphics[width=.45\textwidth]{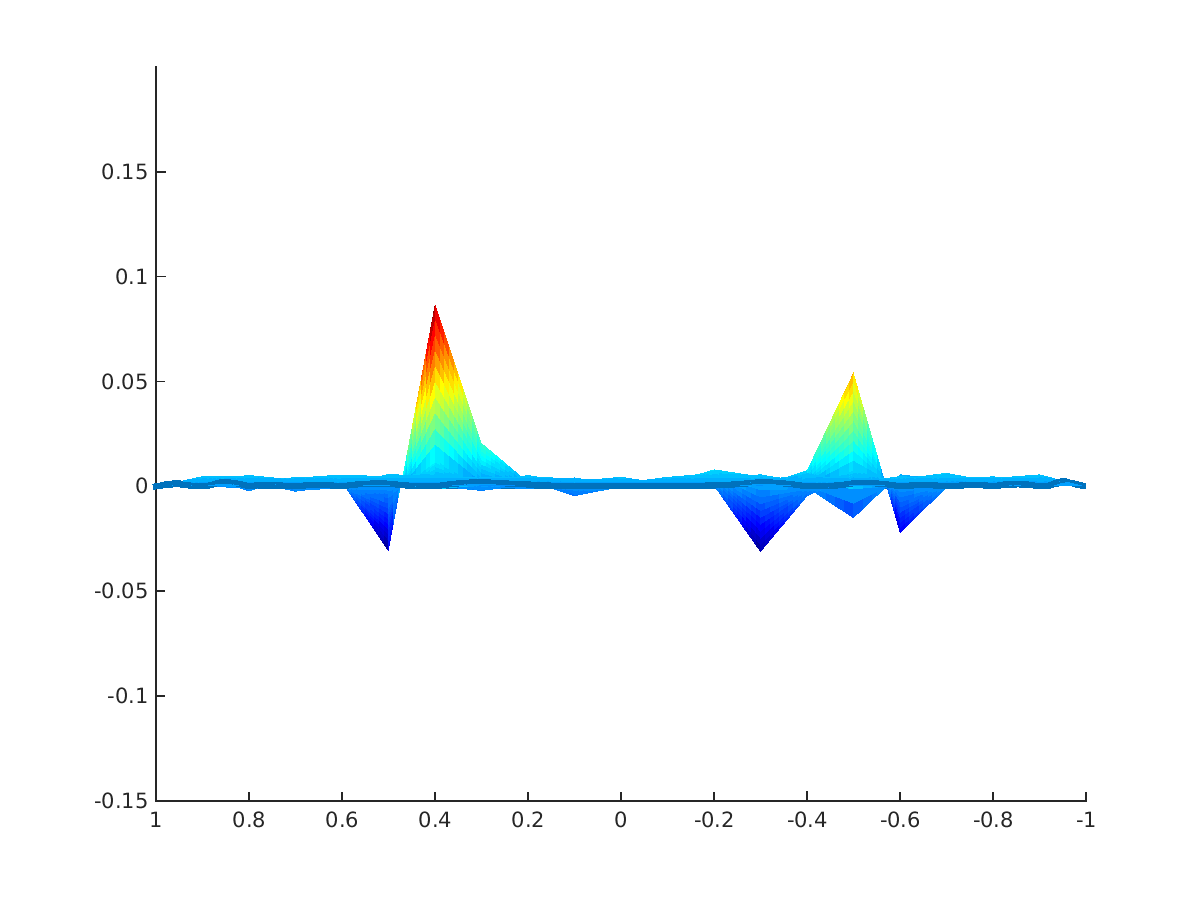} \\
\includegraphics[width=.45\textwidth]{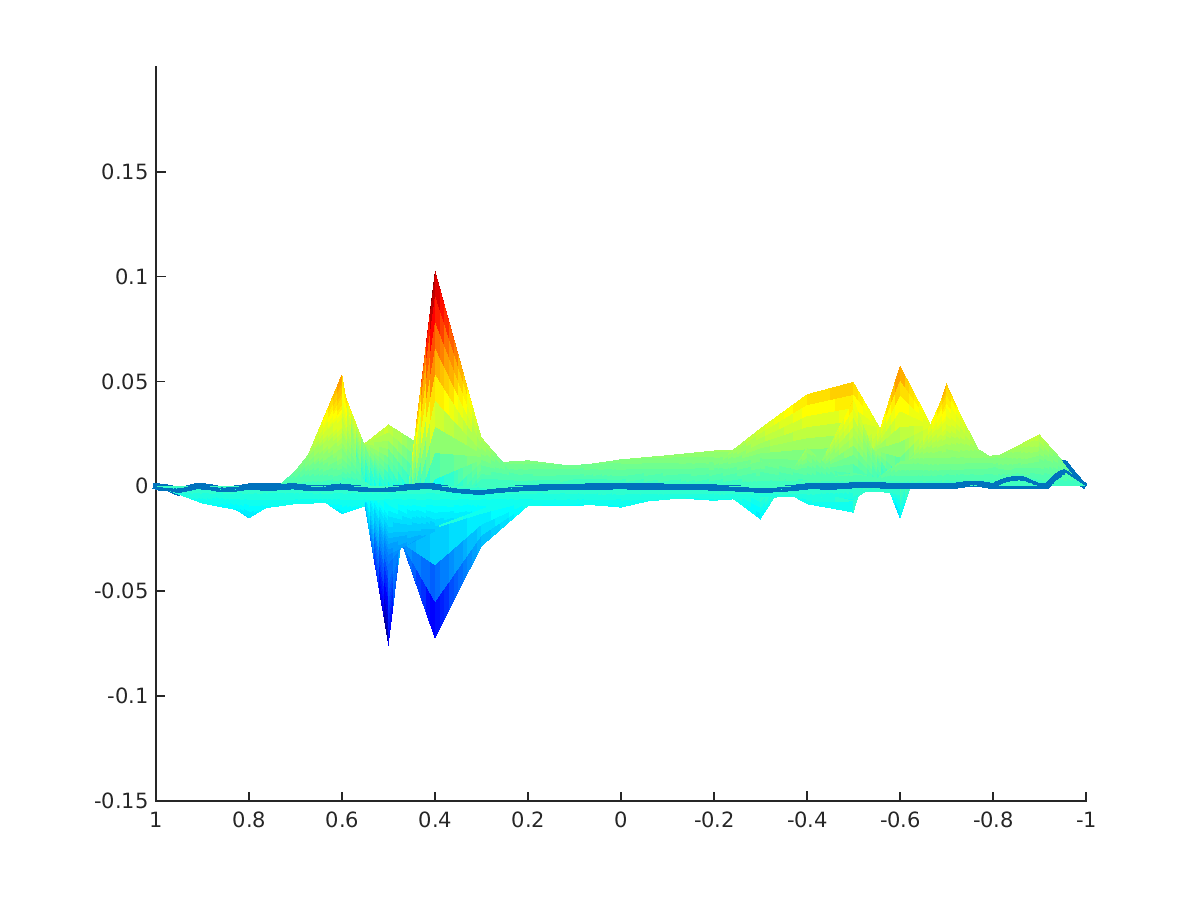} 
\includegraphics[width=.45\textwidth]{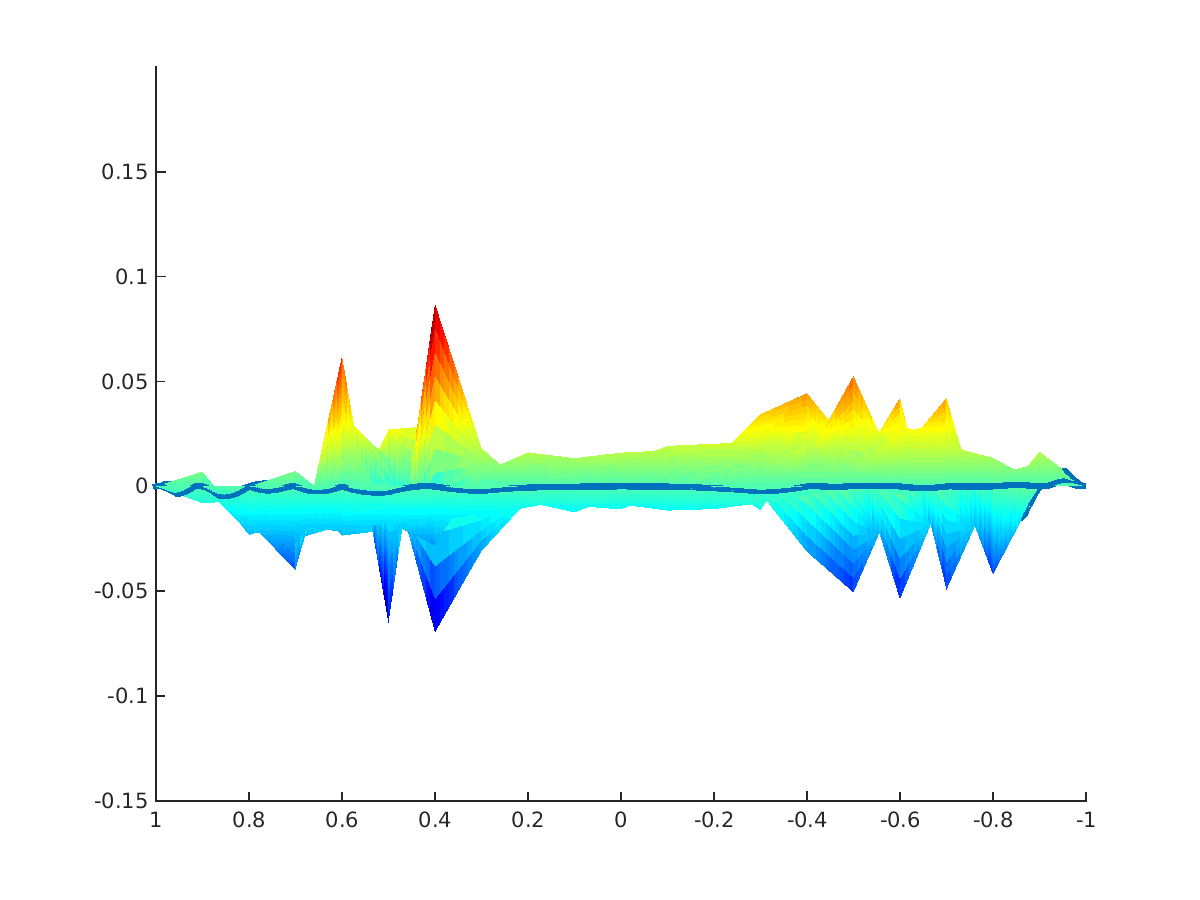} \\
\includegraphics[width=.45\textwidth]{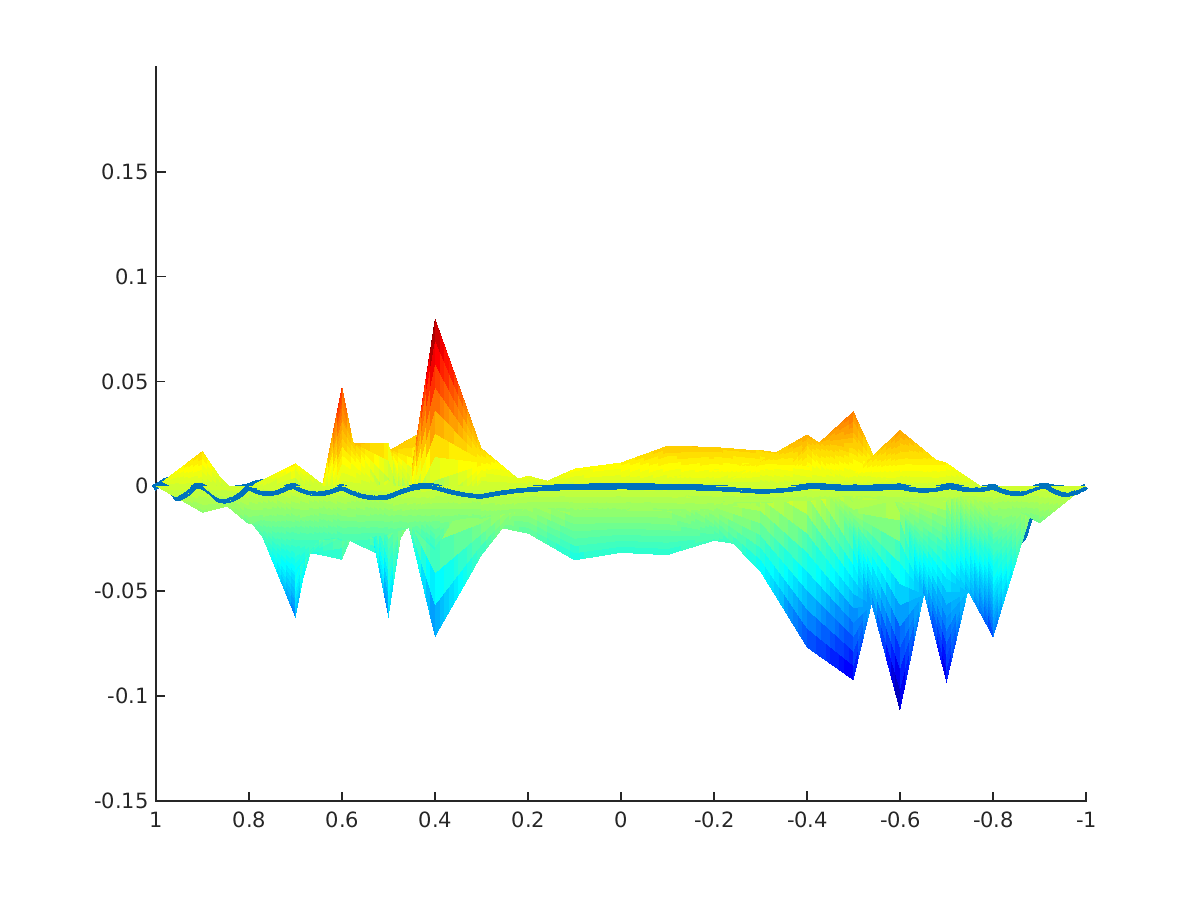}
 \caption{Errors $c(x) - c_{k^*_i}(x)$. Top left: $p_1 = 1$, top right: $p_2 = 1.1$, middle left: $p_3 = 2$, middle right: $p_4 = 4$, 
bottom: $p_5 = 10$.}
 \label{fig:4_Ex2PeaksDraufDiff}
\end{figure}

\subsubsection{Different p normes}

Again we consider the exact speed of sound (\ref{sound-peaks}) but now we calculate reconstructions with different $p$-norms. More explicitly we set
\[
  p_1 = 1,\, p_2 = 1.1,\, p_3 = 2,\, p_4 = 4 \text{ and } p_5 = 10.
\]
In case $p_1 = 1$ we implemented the soft threshold method as presented in \cite{daub} with $\alpha = 0.01$ in this case. For all other $p_i$ we set $\alpha = 0.2$. 
In this series of reconstructions we chose $N_x = 80$ and $N_\xi = 80$ yielding $\left|\overline \Gamma_N^h \right| = 6400$ geodesics to be computed in each iteration setp.
The computation of a full set of geodesics lasts about $70$ seconds, the evaluation of $R_{\tn_k}^*$ about 20 seconds. In every 
iteration step we make one descent step with step size parameter $\mu^1 = 0.05$. The uinit square was discretized again using $h=0.1$.
The optimal stopping indices $k^*_i$, $i=1,\ldots,5$ were
\[
 k^*_1 = 42,\, k^*_2 = 50,\, k^*_3 = 8,\, k^*_8 = 10 \text{ und } k^*_5 = 11.
\]
The reconstructions are visualized in Figure \ref{fig:4_Ex2PeaksVogel} compared with the exact $c$. As expected the $p_1=1$ and 
$p_2=1.1$ norms lead to the best reconstructions because $c(x)$ and thus $\tn (x)$ is sparse. Particularily the most part of the reconstruction is identical to $1$. 
For the choices $p_3=2$, $p_4=4$ and $p_5=10$ one discovers increasing smoothness but also fluctuations of the solution. It is interesting that the small peak at $q_3=(0.5, -0.5)$ seems
to cause severe artifacts at the boundary $\partial \BB$ for $p_5=10$. This is also emphasized in Figure \ref{fig:4_Ex2PeaksDraufDiff} where the errors $c(x)-c_{k^*_i}(x)$
are plotted. Indeed all peaks are detected correctly 
(the maximal error is about $0.1$), but for higher norms the fluctuations in the Euclidean areas, i.e. areas where $\tn=1$, increase significantly. 
Particularily the error is as large as the detected peaks. The artifacts close to the boundary $\partial\BB$ become also obvious when comparing the traces of the geodesics obtained for $p_1=1$ and $p_5=10$
to those of the exact solution $\tn(x)$, see Figure \ref{fig:4_Ex2PeaksGeodaeten}. The left picture shows a good match of the reconstructed and approximated set of geodesics $\mathcal{G}_{k^*}$,
whereas one clearly recognizes big aberrations in the picture to the right.\\
\begin{figure}[H]
\centering
\includegraphics[width=.45\textwidth]{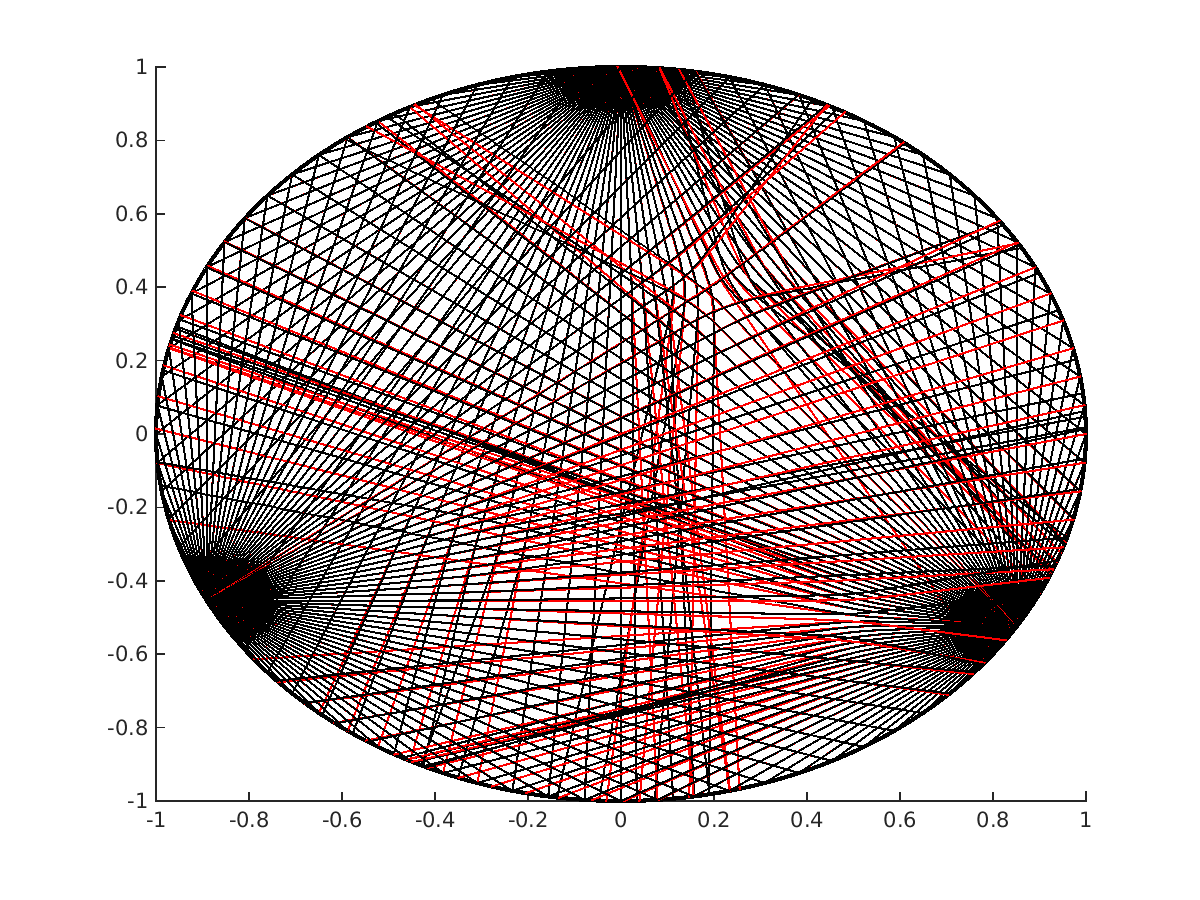}
\includegraphics[width=.45\textwidth]{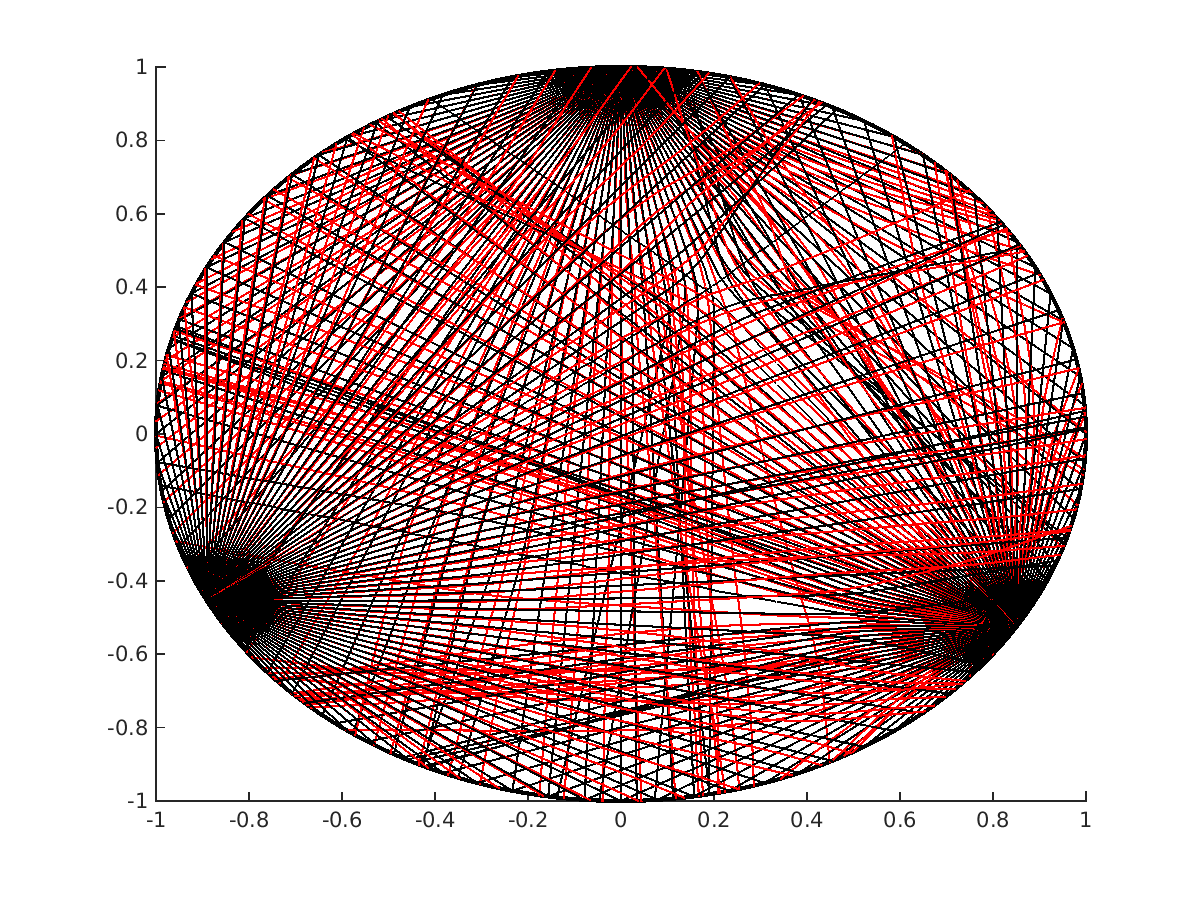} 
 \caption{Traces of reconstructed geodesics. The black curves are geodesics of the exact refractive index $\tn (x)$, the red curves are geodesics of the reconstructions $\tn_{k^*}$. 
Left picture: $p_1 = 1$, $k^* = 42$, Right picture: $p_5 = 10$, $k^* = 11$.}
 \label{fig:4_Ex2PeaksGeodaeten}
\end{figure}
%

%%%%%%%%%%%%%%%%%%%%%%%

\subsection{Sound speed with 'constant curvature'}

Next we consider a sound speed which is not sparse. Let
\[  c(x) = 1/\tn (x) := 1+ \varphi (x) \chi_{\BB} (x)    \]
with
\[  \varphi (x) = \frac{(R^2+d^2 |x|^2)^2}{4R^2}   \]
and the parameters $d=1.2$, $R=2$. The Riemannian manifold $(\BB,g^\tn)$ has then constant, positive Gaussian curvature $K=d^2$. Reconstructions for $p$-norms with
$p_1=1$, $p_2=2$ can be seen in Figure \ref{fig:constant_curvature}. The regularization parameters were $\alpha_1=0.01$ and $\alpha_2=0.4$, respectively. The stopping
indices were $k^*_1=50$ and $k^*_2=20$, respectively. As expected the reconstruction for $p_2$ is more accurate then the sparse reconstruction for $p_1$. We furthermore realize that
at the center the reconstruction deteriorates. This comes from the specific metric $g^\tn$ which generates a \emph{cold spot} in the center, that means a small region where almost
no geodesic curve, i.e. ultrasound signal, intersects. 
This fact is clearly visible when we consider the associated geodesic curves for $\tn$ and $\tn_{k^*_2}$ (Figure \ref{fig:constant_curvature_geodesics}). One sees that only few geodesics pass
the center of the disk.\\

\begin{figure}[H]
\centering
\includegraphics[width=.3\textwidth]{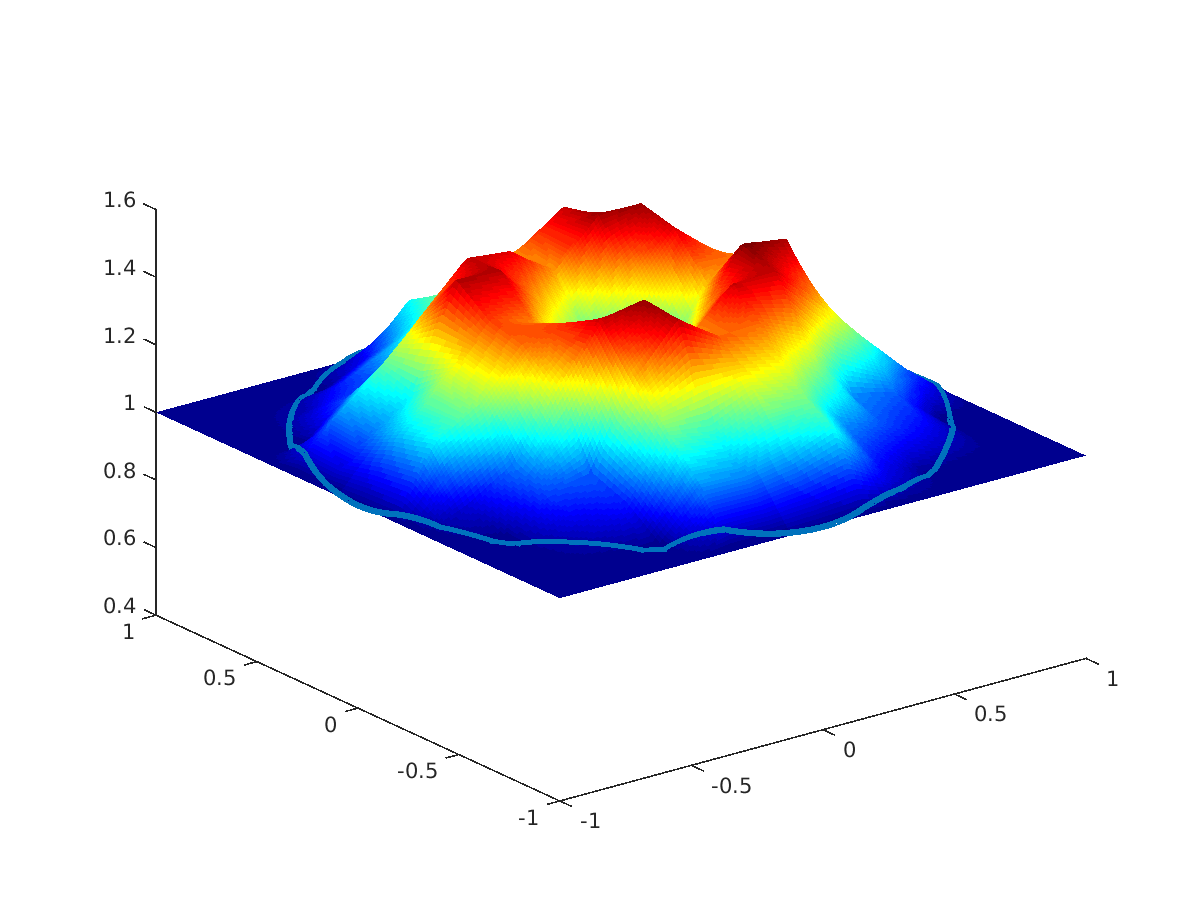}
\includegraphics[width=.3\textwidth]{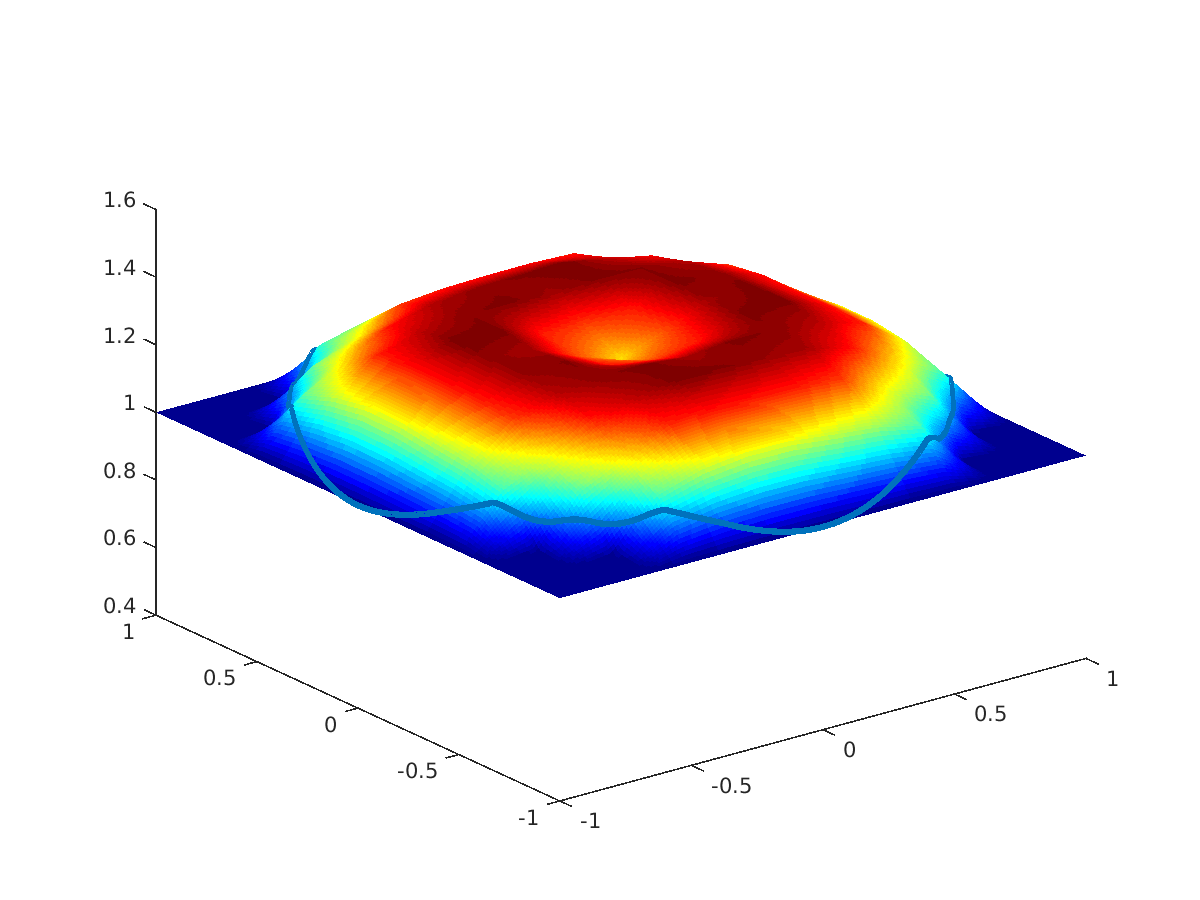}
\includegraphics[width=.3\textwidth]{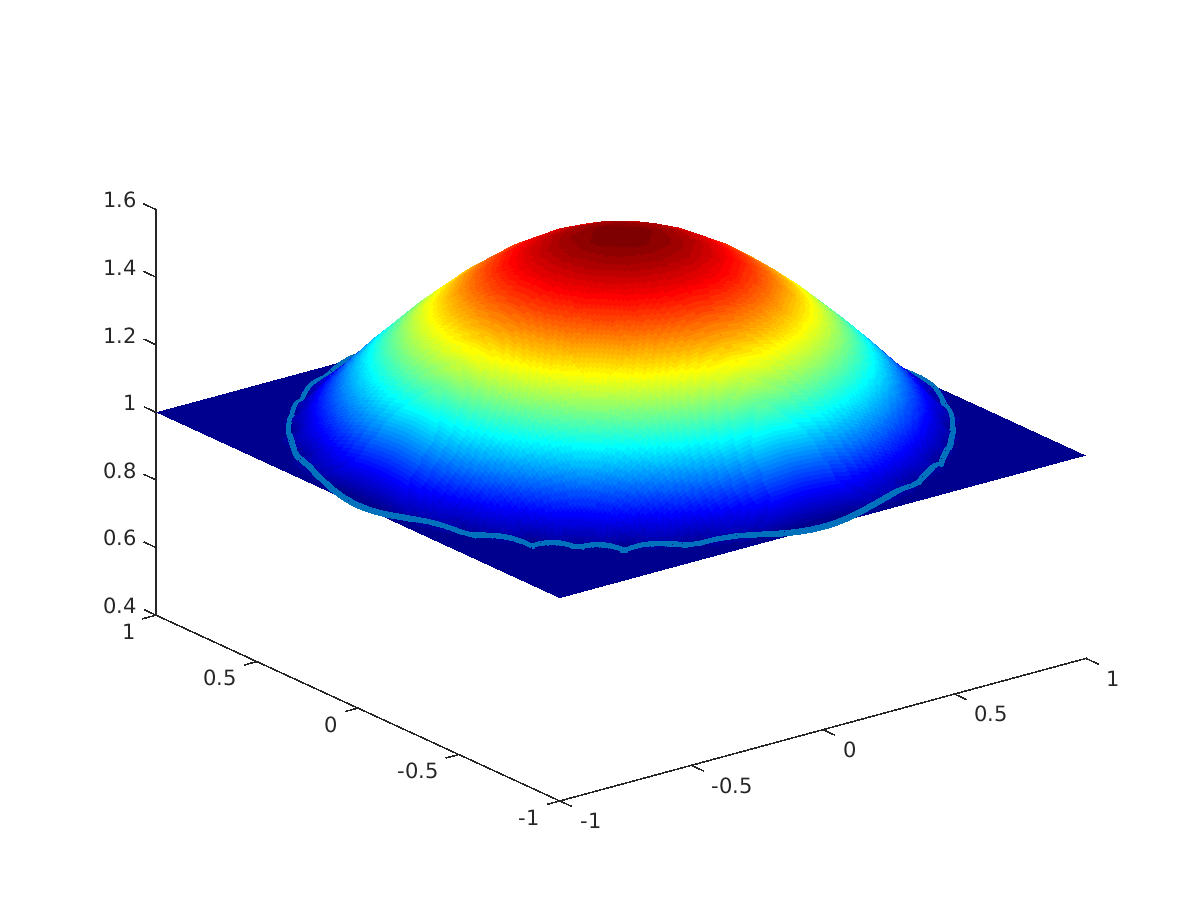}
\caption{Reconstructions of a refrective index with constant, positive curvature. \emph{Left:} $p_1=1$, \emph{Middle:} $p_2=2$, \emph{Right:} exact sound speed.}
 \label{fig:constant_curvature}
\end{figure}

\begin{figure}[H]
\centering
\includegraphics[width=.45\textwidth]{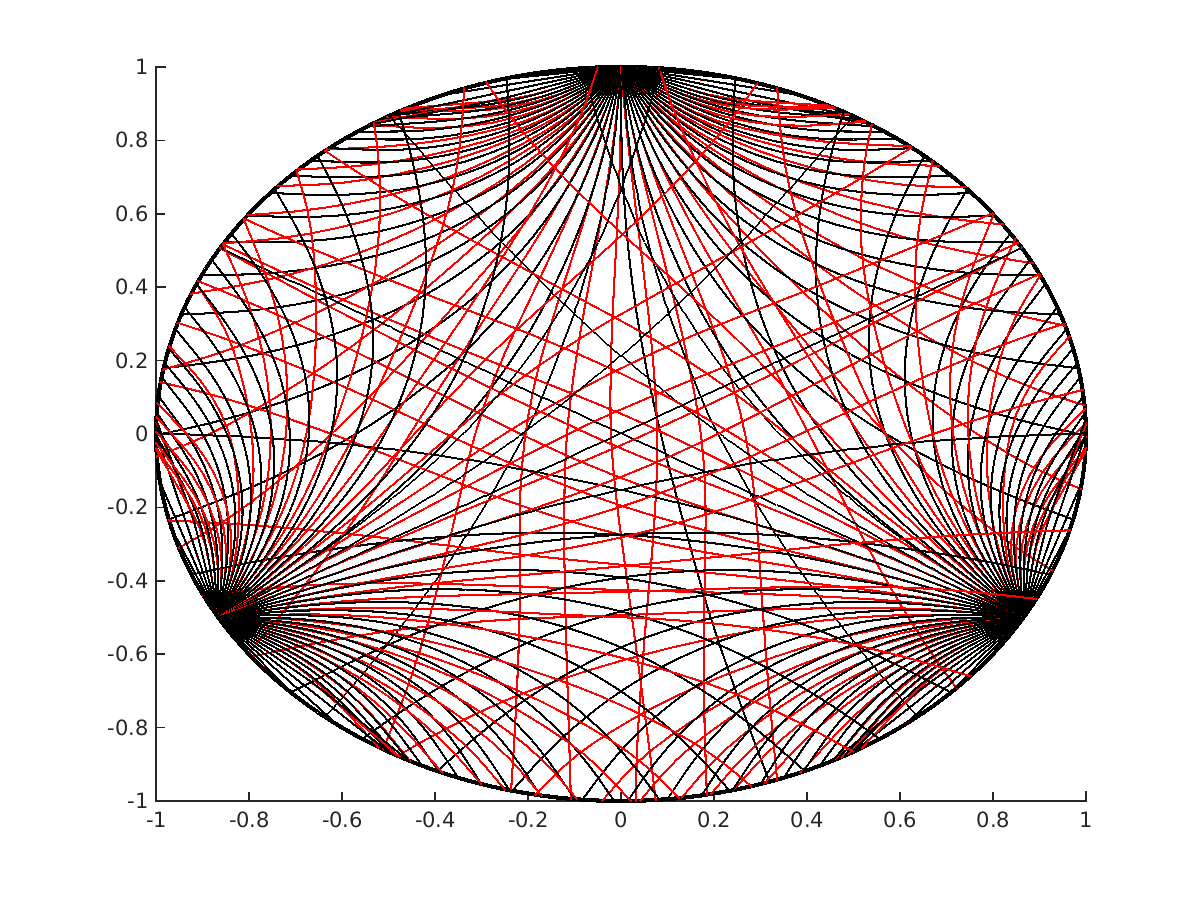} 
\caption{Geodesic curves for $\tn_{k^*_2}$ (red curves) and $\tn$ (black curves). The \emph{cold spot} in the center is clearly visible.}
 \label{fig:constant_curvature_geodesics}
\end{figure}

%%%%%%%%%%%%%%%%%%%%%%%%%%

\subsection{Experiments with noisy data}

At last we show a numerical test using noise contaminated data $u^{meas,\delta}$. We perform this test by means of the exact solution
\[   c(x) = 1/\tn (x) := 1 + \sum_{i=1}^3 \varphi_i (x) \chi_i (x) + \bar{\varphi} (x) \bar{\chi}(x)  \]
with $\varphi_i$, $\chi_i$ as in Subsection \ref{c-peaks} and
\[   \bar{\varphi}(x) = \bar{\vartheta} \cos\left( \pi \frac{|x-\bar{q}|}{\bar{r}} \right)\,,  \]
\[   \bar{\chi} (x) = \left\{ \begin{array}{r@{\,,\quad}l} 1 & \mbox{if } 1-4\bar{r} \leq |x-\bar{q}|\leq 1-2\bar{r}\\ 0 & \mbox{else}\,, \end{array} \right.  \]
and the parameters $\bar{r}=0.1$, $\bar{q} = (0,0)$. The measure data additionally have been contaminated by unifromly distributed noise $\delta$,
\[   u^{meas,\delta} := R (\tn^h) + \delta  \]
with relative error $|\delta | / |R (\tn^h)| = 0.1$ (i.e. $10 \%$ relative noise). Figure \ref{fig:noise} shows reconstructions with exact $u^{meas}=R(\tn)$ as well as with noisy data
$u^{meas,\delta}$. The parameters for the reconstruction with noisy data are $\mu_l=0.02$, $\alpha=1.5$, $p=2$ and $k^*=50$.\\

\begin{figure}[H]
\centering
\includegraphics[width=.3\textwidth]{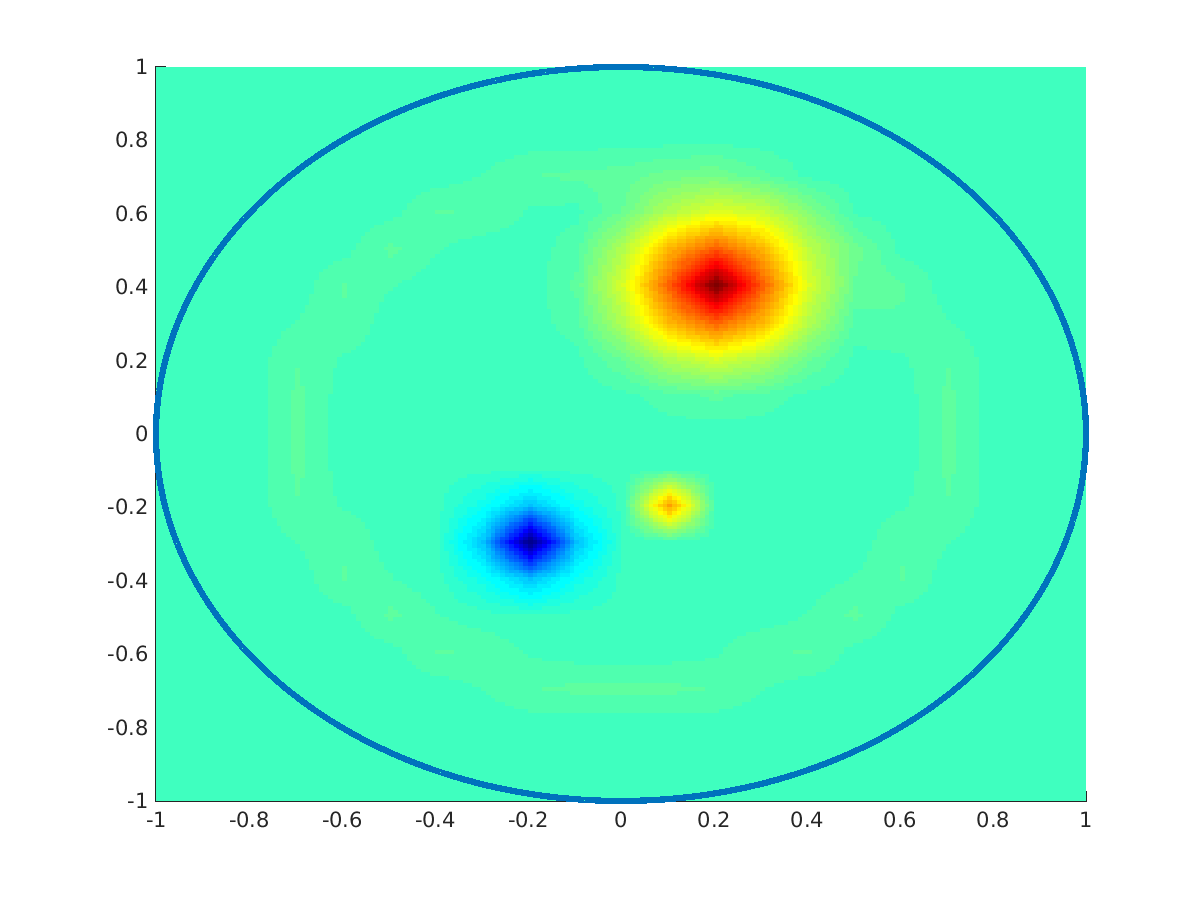}
\includegraphics[width=.3\textwidth]{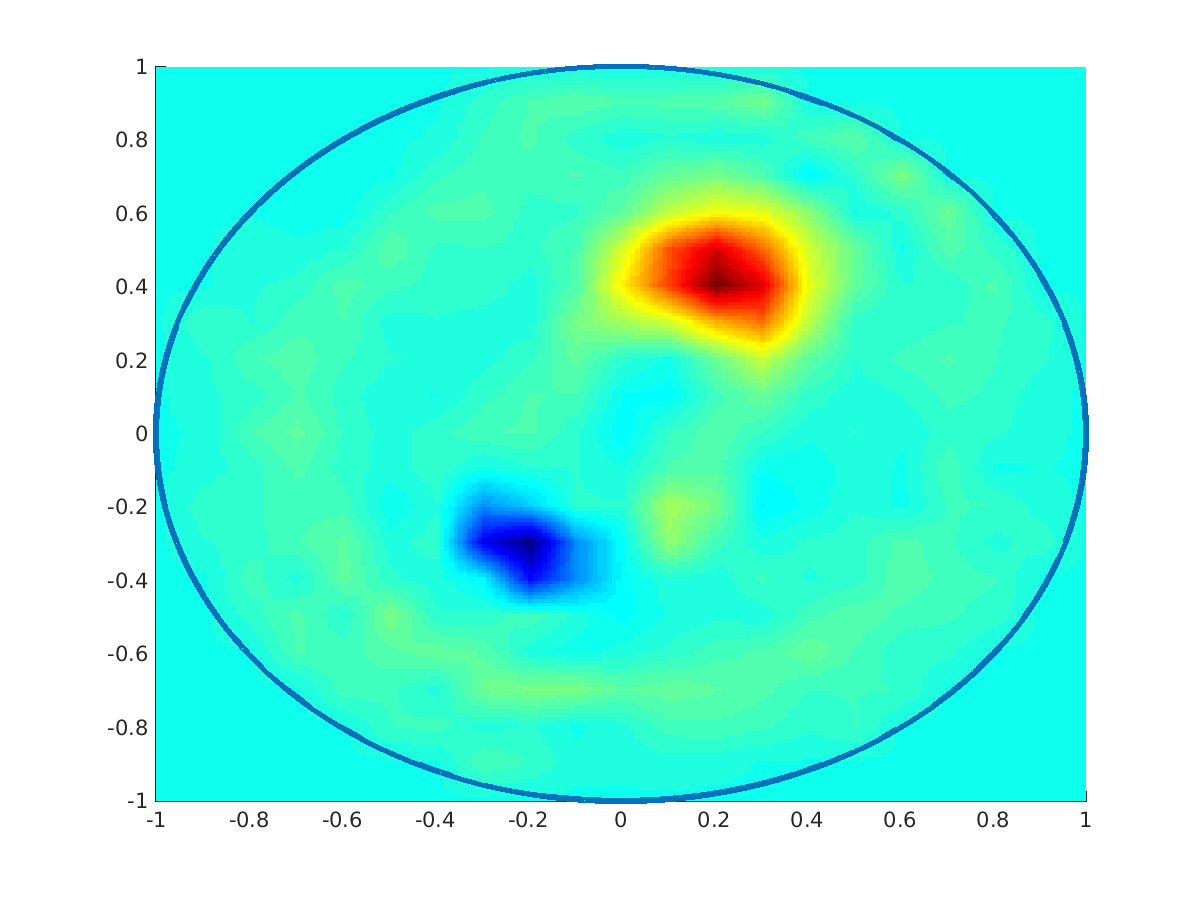}
\includegraphics[width=.3\textwidth]{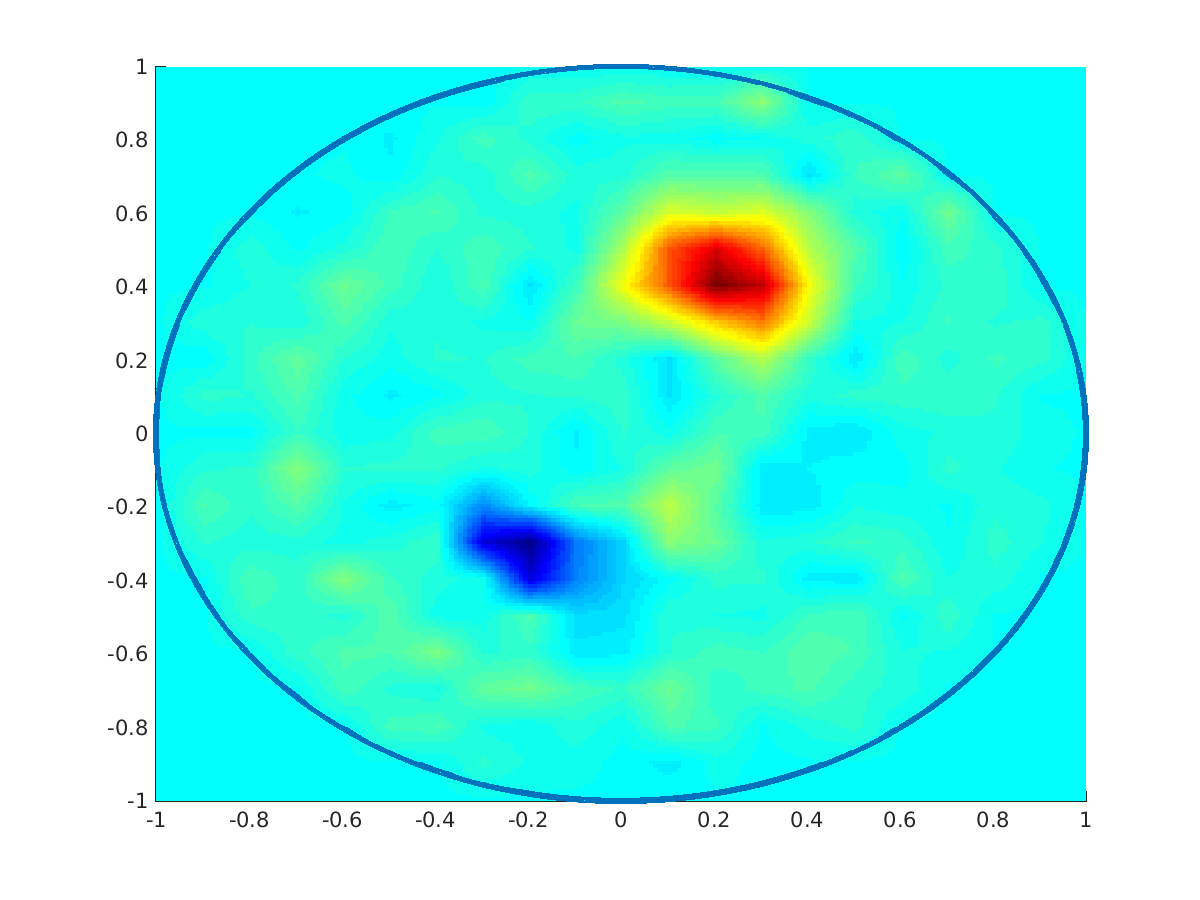}
\caption{Reconstructions of $c(x)$ (left picture) from exact data $u^{meas}$ (middle picture) and noisy data $u^{meas,\delta}$ (right picture).}
 \label{fig:noise}
\end{figure}

%%%%%%%%%%%%%%%%%%%%%%%%%%%%%%%%%%%%%%%%%%%%%%%%%

\section{Conclusions}

%%%%%%%%%%%%%%%%%%%%%%%%%%%%%%%%%%%%%%%%%%%%%%%%%

In the article we propose a numerical solution scheme for the computation of the refrecative index of a medium from boundary time-of-flight measurements in 2D.
The method relies on the minimization of a Tikhonov functional that penalizes aberrations from $\tn=1$.
The minimization is done by a steepest descent method where the linearization of the forward operator was achieved by using the old iteration $\tn_k$ as
refractive index to compute the propagation paths. The ultrasound signals were assumed to propagate along geodesics of the Riemannian metric 
$\mathrm{d} s^2 = \tn^2 (x) |\mathrm{d} x|^2$ which is due to Fermat's principle. We were able to prove that every sequence $\{\tn_k\}$ generated by our
iterative scheme has weak limit points. Of course this is a little unsatisfactory. It would be great if one could prove that these limit points are minimizers
of $J_\alpha$. One way to do this might be to use the concept of surrogate functionals, see e.g. \cite{RAMLAU;TESCHKE:06,lang}. In fact if we define $\tilde{J}_\alpha : X_p\times X_p \to \RR$ by
\[   \tilde{J}_\alpha (x,a) := J_\alpha^a (x)\,,   \]
then obviously our Algorithm \ref{alg:4_IterAdapMin} reads as
\[   \tn_{k+1} = \mathrm{arg} \,\min_{\tn\in L^p (M)} \tilde{J}_\alpha (\tn,\tn_k)  \,,\qquad k=0,1,\ldots . \]
If we assume that $J_\alpha (\tn) \leq \tilde{J}_\alpha (\tn,a)$ at least for all $a$ close to $\tn$, then $\tilde{J}_\alpha$ can be interpreted as a 
(local) surrogate functional. The investigation of convcergence as well as the derivation of the G\^{a}teaux derivative $R' (\tn) a$, $\tn,a\in X_p$
is subject of current research.\\
The numerical experiments show a good performance of the method, also if we have sparse solutions.\\
Another result of the article is the explicit representation of the backprojection operator $R_a^*$ for a non-Euclidean geometry as well as its numerical realization.
We showed the analogy to the conventional (Euclidean) backprojection operator as it is known from 2D computerized tomography.\\
At last we would like to mention that the results of this article do not only affect seismics or phase contrast TOF tomography, but also other tomographic problems
in inhomogeneous media such as vector and tensor field tomography.

%%%%%%%%%%%%%%%%%%%%%%%%%%%%%%%%%%%%%%%%%%%%%%%%%

\section*{Acknowledgments}

We are indebted to the Deutsche Forschungsgemeinschaft (German Science Foundation, DFG) which funded this project under Schu 1978/7-1.

%%%%%%%%%%%%%%%%%%%%%%%%%%%%%%%%%%%%%%%%%%%%%%%%%

\end{document}